\documentclass[11pt]{amsart}

\usepackage{amsmath}
\usepackage{amssymb}
\usepackage{amsthm}
\usepackage{enumerate}
\usepackage[margin=20mm]{geometry}
\usepackage{eucal}
\usepackage{bm}
\usepackage{xcolor}

\theoremstyle{plain}
 \newtheorem{thm}{Theorem}[section]
 \newtheorem{prop}[thm]{Proposition}
 \newtheorem{lem}[thm]{Lemma}
 \newtheorem{cor}[thm]{Corollary}
 \newtheorem{theorem}{Theorem}
 
\theoremstyle{definition}
 \newtheorem{dfn}[thm]{Definition}
 \newtheorem{exm}[thm]{Example}
 \newtheorem{rmk}[thm]{Remark}
 
 \newtheorem{fact}[thm]{Fact}
 
 \newtheorem{notat}[thm]{Notation}
 \newtheorem{question}[thm]{Question}

\newcommand{\ind}{\makebox[1em]{\raisebox{-.5ex}[0ex][0ex]{\makebox[0em]%
{$\smile$}}\raisebox{.4ex}[0ex][0ex]{\makebox[-.02em]{$|$}}}}
\newcommand{\dep}{\makebox[1em]{\raisebox{.3ex}[0ex][0ex]%
{$\not$}\makebox[.7em]{\ind}}}

\def\Mon{{\mathfrak C}}
\def\p{{\mathfrak p}}
\def\q{{\mathfrak q}}
\def\r{{\mathfrak r}}

\def\D{{\mathcal D}}

\def\Dp{{\mathcal D}}
\def\Dpf{{\mathcal D^*}}
\def\Rp{{\mathcal R}}
\def\Lp{{\mathcal L}}

\def\forces{\vdash}
\def\tp{{\rm tp}}

\def\dcl{{\rm dcl}}

\def\Aut{{\rm Aut}}
\def\Inv{{\rm Inv}}
\def\so{{\it so}}
\def\oo{{\it o}}
\def\cl{{\rm cl}}
\def\strok{\upharpoonright}
\def\at{{\rm at}}
\def\Pii{{\bigsqcup}}

\renewcommand{\leq}{\leqslant}
\renewcommand{\geq}{\geqslant}
\renewcommand{\emptyset}{\varnothing}

\DeclareMathOperator{\wor}{\perp\hspace*{-0.4em}^{\mathit w}}
\DeclareMathOperator{\nwor}{\not\perp\hspace*{-0.4em}^{\mathit w}}
\DeclareMathOperator{\fwor}{\perp\hspace*{-0.4em}^{\mathit f}}
\DeclareMathOperator{\nfor}{\not\perp\hspace*{-0.4em}^{\mathit f}}

\begin{document} 

\title{Stationarily ordered types  and the number of countable models}

\author{Slavko Moconja}
\address[S.\ Moconja]{Instytut Matematyczny, Uniwersytet Wroc\l{}awski, pl.\ Grunwaldzki 2/4, 50-384 Wroc\l{}aw, Poland\\ and
University of Belgrade, Faculty of Mathematics, Studentski trg 16, 11000 Belgrade, Serbia}
\email[S.\ Moconja]{slavko@matf.bg.ac.rs}
\thanks{The first author is supported by the Narodowe Centrum Nauki grant no.\ 2016/22/E/ST1/00450, and by the Ministry of Education, Science and Technological Development of Serbia grant no.\ ON174018}

\author{Predrag Tanovi\'c}
\address[P.\ Tanovi\'c]{Mathematical Institute SANU, Knez Mihailova 36, Belgrade, Serbia\\ and University of Belgrade, Faculty of Mathematics, Studentski trg 16, 11000 Belgrade, Serbia}
\email[P.\ Tanovi\'c]{tane@mi.sanu.ac.rs}
\thanks{The second author is supported by the Ministry of Education, Science and Technological Development of Serbia grant no.\ ON174026}


\maketitle

\begin{abstract} 
We introduce the notions of  stationarily ordered types  and theories; the latter generalizes weak   \oo-minimality and the former is a relaxed version of weak \oo-minimality localized at the locus of a single type. 
We show that forking, as a binary relation on elements realizing stationarily ordered types, is an equivalence relation and   that each stationarily ordered type in a model determines some order-type  as an invariant of the model. We study weak and forking non-orthogonality of stationarily ordered types, show that they are equivalence relations and prove that invariants of non-orthogonal types are closely related. The techniques developed are applied to prove that in the case of a binary, stationarily ordered theory with fewer than $2^{\aleph_0}$ countable models,  the isomorphism type of a countable model is determined by a certain sequence of  invariants of the model. In particular, we confirm Vaught's conjecture for binary, stationarily ordered theories.  
 \end{abstract}
\section*{Introduction}

In 1961.\ Robert Vaught conjectured that every countable, complete first order theory with infinite models   has either continuum or at most $\aleph_0$ countable models, independently of the Continuum Hypothesis. Concerning such theories of linearly ordered structures, there are three major results confirming the conjecture in special cases:
\begin{itemize}
\item Rubin in 1973.\   for theories of coloured orders (linear orders with unary predicates, \cite{Rub});
\item Shelah in 1978.\ for theories of linearly ordered structures with Skolem functions (\cite{Sh});
\item Mayer in 1988.\   for \oo-minimal theories (\cite{May}).
\end{itemize}
The next class of ordered theories for which the conjecture has been considered is the class of weakly \oo-minimal theories. It turned out to be a significantly harder problem to deal with and, by now, only very partial results in that direction were obtained by  Baizhanov, Kulpeshov, Sudoplatov and others (see for example \cite{AB}, \cite{ABK} and \cite{KuSu}).

Among the three major results,  the one that required   deepest analysis of  the structures in question is definitely Rubin's. One may say that the core of his analysis lies  in disassembling  the structure into minimal type-definable convex pieces (loci of interval types) and studying  possible isomorphism types of the pieces. As a result of this analysis Rubin  proved that the number of countable models of a complete theory of coloured orders is either continuum or finite  (even equals 1 if the language is finite). This analysis was recently modified by  Rast in \cite{Rast1}, where he improved Rubin's theorem by classifying the isomorphism relation for countable models of a fixed complete theory of coloured orders up to Borel bi-reducibility. 
The original motivation for our work was finding a broader, syntax-free context of theories that would include  theories of  coloured orders and  allow  isomorphism types of (countable) models to be studied in a  similar way.   In this article we    
do it in the context of binary, stationarily ordered  theories; these include all  complete theories of coloured orders and all binary weakly  $o$-minimal theories.  

The key notion in this paper is that of a stationarily ordered type.
A type $p\in S_1(A)$ (in any theory)   is stationarily ordered if there is an $A$-definable linear order $(D,<)$ containing the locus of $p$ such that any $\Mon$-definable set   is constant at the locus of $p$ at $\pm\infty$: it either contains  or is disjoint from some initial   part of $(p(\Mon),<)$  and the same   for some final part;  in this case we say that $\mathbf{p}=(p,<)$ is an \so-pair.  This property may be viewed as a very weak, local version of weak \oo-minimality,  it is also related to   dp-minimal ordered types that are stationary in the sense that Lstp=tp; this is    explained in Section \ref{s dp}.  A complete theory $T$ is stationarily ordered if every   $p\in S_1(\emptyset)$ is so.

Let $T$ be a complete binary, stationarily ordered theory.
We  develop techniques  for analysing   models of $T$  by, roughly speaking,   combining  Rubin's ideas with some of Shelah's  ideas (from his proof of Vaught's conjecture for $\omega$-stable theories in \cite{SHM}, say). As a result of the analysis, under some additional assumptions on  $T$  we will be able to code the isomorphism type of a  countable model $M$ by a certain sequence of order-types and prove:


\begin{theorem}\label{conclusion} Suppose that $T$ is a complete, countable, binary, stationarily ordered theory. Then:
\begin{enumerate}[(a)]
\item  $I(\aleph_0,T)=2^{\aleph_0}$ if and  only if at least one of the following conditions holds:
\begin{enumerate}[(C1)]
\item $T$ is not small;
\item there is a non-convex type $p\in S_1(\emptyset)$;
\item there is a non-simple type $p\in S_1(\emptyset)$;
\item there is a non-isolated forking extension of some $p\in S_1(\emptyset)$ over  1-element domain;
\item there are infinitely many weak non-orthogonality classes of non-isolated types in $S_1(\emptyset)$.
\end{enumerate}

\item $I(\aleph_0,T)=\aleph_0$ if and only if none of (C1)-(C5) holds and there are infinitely many forking non-orthogonality classes of non-isolated types in $S_1(\emptyset)$; 

\item $I(\aleph_0,T)<\aleph_0$ if and only if none of (C1)-(C5)  holds and there are finitely many forking non-orthogonality classes of non-isolated types in $S_1(\emptyset)$.
\end{enumerate}
\end{theorem}

In most of the paper we study   binary structure on  loci of stationarily ordered types  in an arbitrary first-order theory $T$. We {\it do not} assume that $T$ admits a definable linear ordering on the universe and, except for the last section where we deal with coloured orders and weakly quasi-$o$-minimal theories, even if it does admit one, we will not fix it. 
We will prove that forking, viewed  as a binary relation  $x\dep_A y$  on the set of all realizations of stationarily ordered types over $A$, is an equivalence relation.
One consequence of this fact is that  relations of weak ($\nwor$) and forking non-orthogonality ($\nfor$) are equivalence relations on the set of stationarily ordered types over $A$, with $\nfor$ refining   $\nwor$. Another consequence is the existence  of certain order-types, $\mathbf p$-invariants of models.  
They arise in the following way: The mentioned forking relation, when restricted to the locus of a single type, is a convex equivalence relation, meaning that the classes are convex (with respect to some, or as it will turn out, any definable order). Then for a given model $M$ and \so-pair $\mathbf p=(p,<)$, the order-type of any maximal set $I_p(M)$ of pairwise forking-independent elements of $p(M)$ is fixed (does not depend on the particular choice of $I_p(M)$); it is the $\bf p$-invariant of $M$, denoted by $\Inv_{\mathbf p}(M)$. 
We show that for a fixed $p$ and $M$, the invariants corresponding to distinct   orderings of $p(M)$ are either equal or reverse of each other. We prove that invariants behave very well under  non-orthogonality of convex (see Definition 1.4) stationarily ordered types: if $p\nfor q$, then the invariants $\Inv_{\mathbf p}(M)$ and $\Inv_{\mathbf p}(M)$ are either equal or reverse of each other; if $p\nwor q$ and $p\fwor q$, then  $\Inv_{\mathbf p}(M)$ and   $\Inv_{\mathbf q}(M)$ (or its reverse  $\Inv_{\mathbf p}(M)^*$) are shuffled; the latter means that they correspond to (topologically) dense disjoint subsets of some dense, complete linear order, where the correspondence is witnessed by a certain shuffling relation between the orders (defined in Definition 2.3). 
By applying the developed techniques to the context of a binary, stationarily ordered theory satisfying none of the conditions (C1)-(C4) from Theorem 1(a), we will be able to completely analyse the isomorphism type of a countable model  and prove:

\begin{theorem}\label{Thm_intro_main}
Suppose that $T$ is a complete, countable, binary, stationarily ordered theory satisfying none of the conditions (C1)--(C4) from Theorem 1(a). Let $\mathcal F_T$ be a set of representatives of all $\nfor$-classes of  non-isolated types from $S_1(\emptyset)$ and let $(\mathbf{p}=(p,<_p) \mid p\in\mathcal F_T)$ be a  sequence of \so-pairs. Then for all countable models $M$ and $N$:
\begin{center}
$M\cong N$ \   if and only if \ 
$(\Inv_{\mathbf{p}}(M)\mid p\in \mathcal F_T)=(\Inv_{\mathbf{p} }(N)\mid p\in \mathcal F_T)$.
\end{center}
\end{theorem}

\smallskip
The paper is organized as follows. Section 1 contains preliminaries. In Section 2  shuffling relations are introduced and studied; in particular, we describe isomorphism types of structures consisting of   sequences of linear orders with   shuffling relations. In Section 3 we introduce   stationarily ordered types and prove basic properties of forking $\dep$, $\nfor$ and $\nwor$ as binary relations. In Section 4, we distinguish direct (denoted by $\delta(\mathbf p,\mathbf q)$) and indirect non-orthogonality of \so-pairs $\mathbf p$ and $\mathbf q$ and show that $\delta$ is an equivalence relation.  Roughly speaking, $\delta(\mathbf p,\mathbf q)$ means that $\mathbf p$ and $\mathbf q$ have the same direction; for convex types $\delta(\mathbf p,\mathbf q)$ implies $\Inv_{\mathbf{p}}(M)=\Inv_{\mathbf{q}}(M)$, while $\lnot\delta(\mathbf p,\mathbf q)$ implies $\Inv_{\mathbf{p}}(M)=\Inv_{\mathbf{q}}(M)^*$.

 Section 5 is independent of the rest of the paper. There we establish a link with dp-minimality and prove that if every complete 1-type (over any domain)  is stationarily ordered, then the theory is dp-minimal  and tp=Lstp; the converse (for ordered theories) remains open.
Invariants are introduced in Section 6  and  a relationship between invariants of non-orthogonal types is studied. 
In Section 7 we start studying stationarily ordered types in binary theories and  
prepare several technical results needed in Section 8, where the proofs of Theorems  1 and 2  are completed.  
In Section 9 we apply Theorems 1 and 2 to  binary, weakly quasi-\oo-minimal theories and reprove Rubin's theorem. We also sketch an example of a binary, weakly \oo-minimal theory having $\aleph_0$ countable models; such a theory  was first found  by Alibek and Baizhanov in \cite{AB}.

\section*{Acknowledgement} We are deeply grateful to the anonymous referee for careful reading, whose numerous comments and suggestions have greatly improved the presentation of the paper.

\section{Preliminaries}

We will use standard model theoretic concepts and terminology. Usually we work in a large, saturated (monster) model $\Mon$ of a complete first order theory $T$ in a language $L$.  By $a,b,\ldots$ we will denote elements and by $\bar a,\bar b,\ldots$ tuples of elements of the monster. Letters $A,B,\ldots$ are reserved for small subsets and $M,N,\ldots$ for small elementary submodels, where `small' means of cardinality less than $|\Mon|$. By $\phi(\Mon)$ we will denote the solution set of the formula $\phi(\bar x)$ in (the appropriate power of) $\Mon$ and similarly for (incomplete) types. A subset $D\subseteq \Mon^n$ is called $A$-definable, or definable over $A$, if it is of the form $\phi(\Mon)$ for some $\phi(\bar x)\in L(A)$. 
$D\subseteq \Mon^n$ is {\it type-definable}   over $A$ if it is the intersection   of   $A$-definable sets. If $D'\subseteq D$ (and $D$ is usually type-definable over $A$), then $D'$ is {\it relatively definable  over $A$ within $D$} if $D'$ is the intersection of $D$ with some $A$-definable set; in this case if $D'=D\cap \phi(\Mon)$, then $\phi(x)$ is a {\it relative definition} of $D'$ within $D$.  

Complete types over small sets of parameters  are denoted by $p,q,\ldots$, and global types (types over $\Mon$) are denoted by $\p,\q,\ldots$;  $S_n(A)$ is the   space of all complete $n$-types over $A$. If $\phi(\bar x)$ is a formula with parameters from $A$ and $|\bar x|=n$, then $[\phi]\subseteq S_n(A)$ is the set of all $p\in S_n(A)$ 
containing $\phi(\bar x)$; $[\phi]$ is a clopen subset of $S_n(A)$. A type $p\in S_n(A)$ is {\it isolated} if there is a formula $\phi(\bar x)\in p$ such that $[\phi]=\{p\}$, 
or $\phi(\bar x)\vdash_T p(\bar x)$.  Transitivity of isolation holds:  $\tp(\bar a, \bar b/C)$ is isolated if and only if both $\tp(\bar a/C)$ and $\tp(\bar b/C\bar a)$ are isolated.
 We say that the set $B$ is atomic over $A$  if $\tp(\bar b/A)$ is isolated for any $\bar b\in B$. A weak version of transitivity of atomicity holds: if   $A$ is atomic over $C$ and  $B$ is atomic over  $AC$, then $AB$ is atomic over $C$. A theory  $T$ is {\it small} if $S_n(\emptyset)$ is at most countable for all $n$. The following fact is well known and can be found in any  basic model theory textbook (e.g. \cite{Marker}). 
 
\begin{fact}\label{Fact_smalltheory}
Suppose that $T$ is a countable complete theory.

(a) $T$ is small if and only if there is a countable saturated (equivalently universal) model $M\models T$.

(b) If $T$ is not small, then $I(\aleph_0,T)=2^{\aleph_0}$.

(c) If $T$ is small and $A$ is finite, then isolated types are dense in $S_n(A)$ for all $n$.  

(d) If $T$ is small and $A$ is finite, then a   prime model over $A$ exists; it is unique up to isomorphism. 
\end{fact}

A partial type $\pi(x)$ over $A$ is {\it finitely satisfiable} in $D\subseteq \Mon$ if every finite conjunction of formulae from $\pi(x)$ has a realization in $D$.  

\begin{lem}\label{Lema_fs_vs_nonisolation}
Suppose that  $p\in S_1(A)$, $\phi(x)\in p$ and $A\subseteq B$. 

(a) If $p$ is finitely satisfiable in $D$, then there is $q\in S_1(B)$  finitely satisfiable in $\phi(\Mon)\cap D$ with $p\subseteq q$.

(b) $p$  is non-isolated   if and only if it is finitely satisfiable in $D=\phi(\Mon)\smallsetminus p(\Mon)$. 

(c) If $p$ is non-isolated, then there is a non-isolated  $q\in S_1(B)$ extending $p$.
\end{lem}
\begin{proof}
(a) If $p$ is finitely satisfiable in $D$, then the partial type $\Sigma(x)$ consisting of $p( x)$ and the negations of all the formulas $\psi_i( x)$ over $B$ ($i\in I$) that  are not finitely satisfied in $D\cap \phi(\Mon)$ is consistent: otherwise, there would be $\theta( x)\in p( x)$ and a finite $I_0\subseteq I$ such that $\theta( x)\vdash \bigvee_{i\in I_0}  \psi_i( x)$ and then the formula $\theta(x)\land\phi(x)\in p(x)$ would not be satisfied in $D$. Clearly, any   $q\in S_1(B)$ containing $\Sigma(x)$ is finitely satisfiable in $D\cap \phi(\Mon)$. 
This proves part (a); part (b) is easy and (c) follows from (a) and (b).
\end{proof}   

 For  a type $p(\bar x)$ over $B$ and $A\subseteq B$ we denote by $p_{\strok A}(\bar x)$ the restriction of $p$ to parameters from $A$. Similarly for the restrictions of global types. A global type $\p$ is $A$-invariant, or invariant over $A$, if $\phi(\bar x,\bar a)\in\p$ implies $\phi(\bar x,f(\bar a))\in\p$ for all $f\in\Aut(\Mon/A)$ (the group of automorphisms of $\Mon$ fixing $A$ point-wise). For an $A$-invariant type $\p(\bar x)$, a sequence $(\bar a_i\mid i\in I)$ where $(I,<)$ is a linear order, is a Morley sequence in $\p$ over $B\supseteq A$ if $\bar a_i\models\p_{\strok B\bar a_{<i}}$ for all $i\in I$.

For $\phi(\bar x,\bar y)\in L(A)$ and $\bar b\in \Mon$, the formula $\phi(\bar x,\bar b)$ divides over $A$ if there exists a sequence $(\bar b_i\mid i\in\omega)$ of realizations of $\tp(\bar b/A)$ such that the set $\{\phi(\bar x,\bar b_i)\mid i\in\omega\}$ is $k$-inconsistent for some $k\in\omega$. A formula $\phi(\bar x,\bar b)$ forks over $A$ if it implies a finite disjunction of formulae dividing over $A$. A (partial) type divides (forks) over $A$ if it implies a formula dividing (forking) over $A$; $\bar a\dep_A\bar b$ means that $\tp(\bar a/A\bar b)$ forks over $A$.

\begin{dfn}
Let $p,q$ be complete types over $A$.

(a) $p$ and $q$ are {\it weakly orthogonal}, denoted by $p\wor q$, if $p(\bar x)\cup q(\bar y)$ determines a complete type over $A$;

(b) $p$ is  {\it forking-orthogonal} to $q$, denoted by $p\fwor q$, if $\bar a\ind_A \bar b$ for all    $\bar a\models p$ and $\bar b\models q$.
\end{dfn}

Let $(D,<)$ be a linear order and $X,Y\subseteq D$. By $X<Y$ we mean $x<y$ for all $x\in X$ and $y\in Y$ and $\{x\}<Y$ is denoted by $x<Y$. $X$ is a {\it convex} subset of $D$ if $x,x'\in X$ and $x<y<x'$ imply $y\in X$.
$X$ is a {\it bounded} subset  if there are $x,y\in D$ such that $x<X<y$. 

\begin{dfn}
(a) A type $p\in S_1(A)$ is {\it ordered} if there exists an $A$-definable linear order $(D_p,<_p)$ such that  $p(\Mon)\subseteq D_p$; in this case we say that that $(D_p,<_p)$ witnesses that $p$ is ordered. 

(b) An ordered type $p\in S_1(A)$ is {\em convex} if the witnessing  order $(D_p,<_p)$ can be chosen so that $p(\Mon)$ is a convex subset of $D_p$;  in this case we say that that $(D_p,<_p)$ witnesses the convexity of $p$. 

(c) A complete first-order theory $T$ is {\it ordered} if all types from $S_1(\emptyset)$ are ordered.
\end{dfn}

If $T$ is ordered, then by compactness  there exists an $L$-formula defining a linear ordering on $\Mon$.  

\begin{dfn}
Suppose that $(D_p,<_p)$ is an $A$-definable linear order   witnessing that $p\in S_1(A)$ is ordered, $\mathbf{p}=(p,<_p)$ and   $D\subseteq \Mon$.
\begin{enumerate}[(a)]  \item $D$  is {\em $\mathbf{p}$-left-bounded}    if there exists $a\models p$ such that $a<_pD\cap p(\Mon)$; analogously, $\mathbf p$-right-bounded and $\mathbf{p}$-bounded sets are defined.
\item  $D$ is {\em strongly $\mathbf{p}$-left-bounded}  if $D\subseteq D_p$ and there exists $a\models p$ such that $a<_p D$;  analogously, strongly $\mathbf p$-right-bounded and strongly $\mathbf{p}$-bounded sets are defined.
\item  A formula (with parameters) in one free variable is {\em (strongly) $\mathbf{p}$-(left/right-)bounded} if the set that it defines is. 
\item $D$ is {\it left (right)  eventual} in $(p(\Mon),<_p)$  if it contains an initial (final) part of $(p(\Mon),<_p)$. 
\end{enumerate}
\end{dfn}

\begin{lem}\label{fact osnovni} Let $p\in S_1(A)$ be an ordered type, witnessed by $(D_p,<_p)$ and let $\mathbf{p}=(p,<_p)$.
\begin{enumerate}[(a)]
\item If $\phi(x)\in L(\Mon)$ is $\mathbf{p}$-(left/right-)bounded, then there exists $\psi(x)\in p(x)$ such that $\phi(x)\wedge\psi(x)$ is strongly $\mathbf{p}$-(left/right-)bounded.
\item If $D\subseteq p(\Mon)$ is a relatively definable $\mathbf{p}$-(left/right-)bounded set, then it has a strongly $\mathbf{p}$-(left/right-) bounded definition.
\item A type $q\in S_1(B)$ extending $p$  contains a $\mathbf{p}$-bounded formula if and only if it contains a strongly $\mathbf{p}$-bounded formula.
\end{enumerate}
\end{lem}
\begin{proof} 
 (a)  is an easy application of compactness. For example, if   $\phi(x)$ is $\mathbf{p}$-bounded, then there exist $a,a'\models p$ such that $p(x)\cup\{\phi(x)\}\forces x\in D_p\land a<_px<_pa'$; the desired $\psi(x)\in p(x)$ is obtained by compactness. 
 (b) and (c) easily follow from (a). 
\end{proof}

We finish this section with few examples that will be useful later. By $\mathbf 1$ we will denote the   one-element 
order type and by $\bm{\eta}, \mathbf{1}+\bm{\eta}, \bm{\eta}+\mathbf{1}$ and $  \mathbf{1}+\bm{\eta}+\mathbf{1}$   the order types of the rational numbers from the intervals $(0,1),[0,1),(0,1]$ and $[0,1]$ respectively. 

\begin{exm}\label{exm ehr}(A variant of Ehrenfeucht's theory with 3 countable models) 
Consider the language $L_0= \{<\}\cup\{C_n\mid n\in\omega\}$ where $<$ is binary and each $C_n$  is a unary predicate symbol called a {\em convex color}. The following (first-order expressible) properties determine a complete theory $T_0$:
\begin{itemize}
\item $<$ is a dense linear order without endpoints;
\item $(C_n\mid n\in\omega)$ is a sequence of open convex subsets such that $C_0<C_1<\ldots$ and $\bigcup_{n\in\omega}C_n$ is an initial part of the domain.
\end{itemize}
The theory $T_0$  eliminates quantifiers and has a unique non-isolated 1-type $p(x)$ which is determined by $\{\lnot C_n(x)\mid n\in\omega\}$. The isomorphism type of any countable model $M$ is determined by the order-type of $(p(M),<)$, hence we have three countable models: $M_\emptyset$ omitting $p$, $M_\infty$ with  $(p(M_\infty),<)$ of order-type $\bm{\eta}$, and $M_\bullet$ with $(p(M_\bullet),<)$ of order-type $\bm{1}+\bm{\eta}$. 
\end{exm}

\begin{exm}\label{exm dense colors}
Consider the language $\{<\}\cup \{D_i\mid i\in\omega\}$, where each $D_i$ is a unary predicate symbol   called a {\em dense color}. Let $T$ be the theory of a dense linear order without endpoints in which the sets $D_i$  are mutually disjoint and each $D_i$ is a (topologically) dense subset of the domain. 
$T$ is complete, eliminates quantifiers and has a unique non-isolated 1-type: the type of an uncoloured element $q(x)=\{\neg D_i(x)\mid i\in\omega\}$. By a standard back-and-forth argument, one proves that $T$ has a unique   {\em completely coloured} model (i.e.\ omitting $q(x)$). However, $T$ has $2^{\aleph_0}$  countable models because  $(q(M),<)$ can have an arbitrary  countable order-type. 
\end{exm}

\begin{exm}\label{exm ehr b}
Consider now the theory $T_1\supset T_0$ in the language $L_1= L_0\cup \{D_i\mid i\in\omega\}$,   where the unary predicate symbols $D_i$ are interpreted as dense, pairwise disjoint colors in  models of $T_0$. 
The theory $T_1$ is complete and eliminates quantifiers.
As in Example \ref{exm dense colors},  we have a type of an uncoloured element $q(x)=\{\neg D_i(x)\mid i\in\omega\}$ and we will say that a model of $T_1$ is   completely coloured  if it omits this type. 
We also have the  type $p(x)=\{\lnot C_n(x)\mid n\in\omega\}$; both $p(x)$ and $q(x)$ are incomplete types. If $M\models T_1$ and $p(M)\neq \emptyset$, then $p(M)$ is a dense linear order without right end, and from the basic relations, besides the order, only the $D_n$ are interpreted in the substructure on $p(M)$. The isomorphism type of a completely coloured countable model $M\models T_1$ is determined by the order-type of $(p(M),<)$ and the colour of its minimum (if the minimum exists). Hence $T_1$ has $\aleph_0$ countable, completely coloured models: $M_{\emptyset}$ omitting $p$, $M_\infty$ with $(p(M_\infty),<)$ of order-type $\bm{\eta}$, and $M_i$ ($i\in\omega$)  with $(p(M_i),<)$ of order-type $\mathbf 1+\bm{\eta}$ and the minimum coloured by $D_i$.
\end{exm}

\section{Shuffling relations}

In this section we introduce and study  shuffling relations between (arbitrary) linear orders. The main result is Proposition \ref{Prop projekcije izomorfne iff} in which  the isomorphism type  of a  sequence of pairwise shuffled, countable linear orders is described. 
We  use the following notation: for a binary relation $S\subseteq A\times B$, $a\in A$ and $b\in B$,  \begin{center}
$S(a,B)=\{y\in B\mid (a,y)\in S\}$ \ and \  $S(A,b)=\{x\in A\mid (x,b)\in S\}$
\end{center} denote its fibers. Let us emphasize that  by $\subset$ we denote the strict inclusion.

By \cite{SimLO}, a {\em monotone relation} between linear orders $(A,<_A)$ and $(B,<_B)$  is a   binary relation $S\subseteq A\times B$
such that $x'\leqslant_A x$, $(x,y)\in S$ and $y\leqslant_B y'$ imply $(x',y')\in S$. Basic examples of monotone relations are  those induced by 
certain increasing functions: if $(C,<_C)$ is a linear order and $f:A\to C$, $g:B\to C$  increasing functions, 
then $S=\{(a,b)\in A\times B\mid f(a)<_Cg(b)\} $ is a monotone relation. In particular,  if $(A,<_A)$ and $(B,<_B)$ are suborders of $(C,<_C)$, 
then   $S=\{(x,y)\in A\times B\mid x<_Cy\}$ is a monotone relation. The following is easy to prove:  

\begin{lem}\label{Fact_monotonicity}
If
$(A,<_A)$ and $(B,<_B)$ are linear orders and  $S\subseteq A\times B$, then $S$ is a monotone relation if and only if  either of the following equivalent conditions holds:

(1) \   $(S(A,b)\mid b\in B)$    is an increasing sequence of initial parts  of $A$: \ $b<_B b'$ implies $S(A,b)\subseteq S(A,b')$;  

(2) \    $(S(a,B)\mid a\in A)$ is a   decreasing sequence  of final parts  of $B$.\qed
\end{lem}

Later on we will meet with situations in which the orders are type-definable subsets of some first order structure  and the monotone relation between them is  relatively definable. 

\begin{lem}\label{Fact_rel definable monotone relation}
Suppose that $(D_1,<_1)$ and $(D_2,<_2)$ are $A$-definable linear orders, $P_1\subseteq D_1$ and $P_2\subseteq D_2$ are sets type-definable  over $A$,  and that $\sigma(x,y)$  is a formula over $A$   relatively defining a monotone relation on $P_1\times P_2$. Then there are $A$-definable sets $D_1'$ and $D_2'$ such that $P_1\subseteq D_1'\subseteq D_1$ and $P_2\subseteq D_2'\subseteq D_2$ and $\sigma(x,y)$ defines a monotone relation between  $(D_1',<_1)$ and $(D_2',<_2)$.
\end{lem}
\begin{proof}Let $p_i(x)$ be (a possibly incomplete) type defining $P_i$ for $i=1,2$. The fact that $\sigma(P_1,y)$  for $y\in P_2$  is an initial part of $P_1$ is expressed by:
\begin{center}
$p_1(x_1)\cup p_1(x_2)\cup\{x_1<_1x_2\}\cup p_2(y)\vdash \sigma(x_2,y)\Rightarrow\sigma(x_1,y)$,  
\end{center}
and that $(\sigma(P_1,y)\mid y\in P_2)$ is an increasing family by:
\begin{center}
$p_2(y_1)\cup p_2(y_2)\cup\{y_1<_2y_2\}\cup p_1(x)\vdash \sigma(x,y_1)\Rightarrow\sigma(x,y_2)$. 
\end{center}
By compactness the conclusion follows.
\end{proof}

A monotone relation is {\em strictly monotone} if the sequence $(S(A,b)\mid b\in B)$    strictly increases. For our purposes a subclass of the class of strictly monotone relations is particularly important:  

\begin{dfn}\label{Defin_shuffling} A {\em shuffling relation} of  linear orders  $(A,<_A)$ and $(B,<_B)$ is a non-empty relation $S\subseteq A\times B$ satisfying:
 \begin{enumerate}[\hspace{1em} (1)]
 \item  $(S(A,b)\mid b\in B)$ is a strictly increasing sequence of  initial parts  of $A$  none of which has a supremum in $A$;
\item  $(S(a,B)\mid a\in A)$ is a strictly decreasing family of  final parts  of $B$.   
\end{enumerate}
\end{dfn}

Before continuing, we note that it is not hard to see  that conditions (1) and (2) imply that $(A,<_A)$ and $(B,<_B)$ are dense linear orders and that none of the fibers $S(a,B)$, for $a\in A$, has an infimum in $B$.

As an example of a shuffling relation, consider the ordered rational and irrational numbers as suborders of the real line $(\mathbb R, <)$. They are shuffled by the restriction of $<$ to $\mathbb Q\times (\mathbb R\smallsetminus \mathbb Q)$. Similarly, if we  take any  dense linear order and   two topologically dense, mutually disjoint suborders, then the appropriate restriction of the ordering is a shuffling relation between the suborders. 
We will show that any shuffling relation of two orders may be obtained in a similar way by embedding them into the Dedekind completion of one of them. 
  Recall that any linear order $(A,<_A)$ has a unique, up to isomorphism, completion: a complete  (every subset has a supremum) order in which $(A,<_A)$ is embedded as a topologically dense subset; let us emphasize that by $\sup(\emptyset)=a$ we will mean   that $a\in A$ is the minimal element of $A$.  Usually, the completion is obtained by considering all initial parts ordered by the inclusion and identifying those of them that have the same supremum; in that way each initial part of $A$ is identified with at most one other;  in fact, two initial parts are identified  if and only if they are of the form 
  $(-\infty,a)_A$ and $(-\infty,a]_A$ for some  $a\in A$ with no  immediate predecessor.   Here we will deal only with Dedekind completions of a dense linear orders.

By  the  completion of  a dense linear order $(A,<_A)$ we will mean the order $(\frak D(A),\subset)$, where $\frak  D(A)$ consists of all (topologically) open initial parts of $A$,  including $\emptyset$ and excluding $A$ if it has maximum.    $(\frak  D(A),\subset)$ is a complete dense linear order and the canonical mapping $\pi_A:A\to \frak  D(A)$ defined by $\pi_A(a)=(-\infty,a)_A$ is an order embedding identifying $A$ with a   dense subset of $\frak  D(A)$. 

\begin{lem}\label{Lema_shufling 2 orders}
Suppose that the relation $S$ shuffles  orders $(A,<_A)$ and $(B,<_B)$. Let    $\pi_A:A\to \frak D(A)$ be defined by $\pi_A(a)=(-\infty,a)_A$ and let $\pi_B:B\to \frak D(A)$ be defined by $\pi_B(b)=S(A,b)$. Then:  
\begin{enumerate}[(a)]
\item $\pi_A$ and $\pi_B$ are embeddings of the corresponding orders into $(\frak D(A),\subset)$.
\item The images  of $A$ and $B$ are mutually disjoint, topologically  dense subsets of the completion.  
\item $(A,<_A)$ and $(B,<_B)$ are dense linear orders that have isomorphic Dedekind completions. At most one of them has a minimum (maximum).
\item For all $a\in A$ and $b\in B$: \ $\pi_A(a)\subset \pi_B(b)$ \ if and only if \ $(a,b)\in S$. 
\end{enumerate}
\end{lem}
\begin{proof} (a) As we discussed above $\pi_A$ is an order embedding directly by the definition, while $\pi_B$ is an order embedding since $S$ is a shuffling relation.

To prove part (b), notice that 
for any $b\in B$ condition (1) says that the set $\pi_B(b)=S(A,b)$ has no supremum in $A$, hence $\pi_A(a)=(-\infty,a)_A\neq \pi_B(b)$ holds for all $a\in A$ because $\sup(-\infty,a)_A=a$. We conclude  that the images $\pi_A(A)$ and $\pi_B(B)$ are disjoint.
As for the density, since the image of $A$ is (topologically) dense in $\frak D(A)$,  it suffices to prove that between the images of any two elements $a_1<_Aa_2$,   an element of $\pi_B(B)$ can be found. Condition (2) (from the definition) implies that $S(a_1,B)\supset S(a_2,B)$ holds, so there is $b\in S(a_1,B)\smallsetminus S(a_2,B)$.  Consider the set $S(A,b)$. It is an initial part of $A$ containing $a_1$, so $(-\infty,a_1)_A\subseteq S(A,b)$ holds. Since the images are disjoint we have  $(-\infty,a_1)_A\subset  S(A,b)$. Similarly, $a_2\notin S(A,b)$ implies $S(A,b)\subset  (-\infty, a_2)_A$. Therefore $(-\infty,a_1)_A\subset S(A,b)\subset (-\infty,a_2)_A$, proving the density of the image of $B$.

\smallskip
(c) We have already noticed that $(A,<_A)$ and $(B,<_B)$ are dense linear orders as there is a shuffling relation between them. It is not hard to check that by $J\mapsto\bigcup_{b\in J}S(A,b)$ is defined an order isomorphism between $(\D(B),\subset)$ and $(\D(A),\subset)$.
To prove the last clause notice that, by density of the images,  a minimal (maximal) element of either $A$ or $B$ is mapped by the corresponding projection to a minimal (maximal) element of $\frak  D(A)$. Since the images are disjoint, at most one of the orders has a minimum (maximum).

\smallskip
(d) For right-to-left direction, assume that $(a,b)\in S$. Then $a\in S(A,b)=\pi_B(b)$ holds  and, because $S(A,b)$ is an initial part with no supremum in $A$, we derive $\pi_A(a)=(-\infty,a)_A\subset \pi_B(b)$, proving the first direction. The other direction is even easier:  $\pi_A(a)=(-\infty,a )_A\subset S(A,b)=\pi_B(b)$,   since $S(A,b)$ is an initial part and the inclusion is strict, implies $(a,b)\in S$.  
\end{proof}

\begin{rmk} \label{Remark posssible order-types}
Suppose that $((A_i,<_i)\mid i\in \alpha)$ is a family of countable linear orders  any pair of which is linked by a shuffling relation. By  Lemma \ref{Lema_shufling 2 orders} they are dense linear orders, at most one of them has a minimum and at most one has maximum.  Here we have $(|\alpha|+1)^2$ possibilities:

 (1) \ All of them have order-type $\bm{\eta}$;
 
(2) \ One  has   order-type $\mathbf 1+\bm{\eta}$ and all  the others order-type $\bm{\eta}$;

(3) \  One   has   order-type $\bm{\eta}+\mathbf 1$ and all  the others order-type $\bm{\eta}$;

(4) \  One  has  order-type $\mathbf 1+\bm{\eta}+\mathbf 1$ and all  the others order-type $\bm{\eta}$;

(5) \ One  has  order-type $\mathbf 1+\bm{\eta}$, some other  $\bm{\eta}+\mathbf 1$  and  all the others order-type $\bm{\eta}$.
 \end{rmk}

Now we consider shuffling many orders. The idea is to embed them into the completion of one of them so that the images are dense and mutually disjoint there.
 Recall that the composition of binary relations $R\subseteq A\times B$ and $S\subseteq B\times C$ is defined by: $S\circ R:=\{(a,c)\in A\times C\mid (a,b)\in R \mbox{ and } (b,c)\in S\mbox{ for some $b\in B$}\}$.

Let  $\alpha$ be an ordinal.  We say that a family of relations $(S_{i,j}\subseteq A_i\times A_j\mid i<j\in\alpha)$ {\it shuffles the sequence} of dense linear orders $((A_i,<_i)\mid i\in\alpha )$  if  each $S_{i,j}$ shuffles $(A_i,<_i)$ and $(A_j,<_j)$, and the relations $S_{i,j}$  satisfy the {\it coherence condition}: $S_{j,k}\circ S_{i,j}=S_{i,k}$ (for all $i<j<k<\alpha$).  
In this situation we define canonical embeddings $\pi_i:A_i\to \frak D(A_0)$ by $\pi_0(x)=(-\infty,x)_{A_0}$ and $\pi_i(x)=S_{0,i}(A_0,x)$ for $i>0$.

\begin{lem}\label{Lema_shufle many orders}
Suppose that $((A_i,<_i)\mid i\in\alpha )$  is a sequence of dense linear orders shuffled by a family $(S_{i,j}\subseteq A_i\times A_j\mid i<j\in\alpha)$. 

(a) The sets $(\pi_i(A_i)\mid i\in\alpha)$ are mutually disjoint, topologically dense subsets of  $\frak D(A_0)$.

(b) $((A_i,<_i)\mid i\in\alpha )$ is a sequence of  dense linear orders that have isomorphic Dedekind completions. At most one of them has a minimum (maximum).

(c) For each $i<j$ we have: $\pi_i(a_i)\subset \pi_j(a_j)$ if and only if $(a_i,a_j)\in S_{i,j}$. 
\end{lem}
\begin{proof}(a) 
By Lemma \ref{Lema_shufling 2 orders}(b) each $\pi_i(A_i)$ is a topologically dense subset of $\frak D(A_0)$, so it remains to show that $\pi_i(A_i)$ and $\pi_j(A_j)$ are disjoint for $i\neq j$. By Lemma \ref{Lema_shufling 2 orders}(b) this holds when $i=0$, so assume that $0<i<j$,  
 $a_i\in A_i$ and $a_j\in A_j$. We will prove that $\pi_i(a_i)\neq\pi_j(a_j)$.
By Lemma \ref{Lema_shufling 2 orders} applied to $A_i$ and $A_j$,  we have $S_{i,j}(A_i,a_j)\neq (-\infty,a_i)_{A_i}$. We distinguish two cases here. The first is when $(-\infty,a_i)_{A_i}\subset S_{i,j}(A_i,a_j)$. 
Then  $a_i\neq \sup S_{i,j}(A_i,a_j)$ implies $a_i\in S_{i,j}(A_i,a_j)$. Since $S_{i,j}(A_i,a_j)$ is an initial part of $A_i$ with no supremum and contains $a_i$, there is $a_i'>_ia_i$ such that $(a_i',a_j)\in S_{i,j}$.
 Then  $S_{0,i}(A_0,a_i)\subset S_{0,i}(A_0,a_i')$ implies that  we can find $b\in A_0$ such that $(b,a_i)\notin S_{0,i}$ and $(b,a_i')\in S_{0,i}$.  On the one hand,  $(b,a_i')\in S_{0,i}$, $(a_i',a_j)\in S_{i,j}$ and $S_{0,j}= S_{i,j}\circ S_{0,i}$ together imply  $b\in S_{0,j}(A_0,a_j)=\pi_j(a_j)$. On the other hand,  $(b,a_i)\notin S_{0,i}$ implies  $b\notin\pi_i(a_i)$, which combined  with   $b\in\pi_j(a_j)$ proves $\pi_i(a_i)\neq\pi_j(a_j)$. 

The proof in the second case is similar. Assuming  $S_{i,j}(A_i,a_j)\subset (-\infty,a_i)_{A_i}$ we derive $(a_i,a_j)\notin S_{i,j}$; hence $S_{i,j}(A_i,a_j)<_i a_i$. Since $a_i$ is not the supremum of $S_{i,j}(A_i,a_j)$ we can find $a_i''<_ia_i$ such that $(a_i'',a_j)\notin S_{i,j}$.
Then $S_{0,i}(A_0,a_i'')\subset S_{0,i}(A_0,a_i)$ implies that there is  $c\in A_0$ such that $(c,a_i)\in S_{0,i}$ and $(c,a_i'')\notin S_{0,i}$. We prove that $c$ witnesses $\pi_i(a_i)\neq\pi_j(a_j)$. Clearly, $c\in\pi_i(a_i)$ holds.  If  $c\in\pi_j(a_j)$ were true, then $S_{0,j}= S_{i,j}\circ S_{0,i}$   would imply that   some $d\in A_i$ satisfies $(b,d)\in S_{0,i}$ and $(d,a_j)\in S_{i,j}$; then  $(c,d)\in S_{0,i}$ and $(b,a_i'')\notin S_{0,i}$ would imply $a_i''<_i d$ by monotonicity, which combined  with $(d,a_j)\in S_{i,j}$  implies $(a_i'',a_j)\in S_{i,j}$; contradiction. Therefore,  $c\notin\pi_j(a_j)$ holds  and  $\pi_i(a_i)\neq\pi_j(a_j)$. This completes the proof  of part (a).

\smallskip
(b) This follows directly from Lemma \ref{Lema_shufling 2 orders}(c).

\smallskip
(c) If $ (a_i,a_j)\notin S_{i,j}$, then $S_{i,j}(A_i,a_j)<_ia_i$, so exactly by the same calculation  as in the proof of part (a) we get $\pi_j(a_j)\subset\pi_i(a_i)$, and thus $\pi_i(a_i)\not\subset\pi_j(a_j)$. On the other hand, if $ (a_i,a_j)\in S_{i,j}$ and $x\in\pi_i(a_i)$, then $(x,a_i)\in S_{0,i}$, hence $(x,a_j)\in S_{i,j}\circ S_{0,i}$, i.e.\ $(x,a_j)\in S_{0,j}$ and thus $x\in\pi_j(a_j)$; this proves  $\pi_i(a_i)\subseteq\pi_j(a_j)$ from which,  by part (a),  $\pi_i(a_i)\subset\pi_j(a_j)$ follows. 
\end{proof}

For a sequence $\mathcal A= ((A_i,<_i)\mid i\in\alpha )$    of dense linear orders shuffled by $\mathcal S=(S_{i,j}\mid i<j<\alpha)$ we  form its {\it limit structure} \  
$(\Pii_{\mathcal S}\mathcal A,\subset, \pi_i(A_i))_{i\in\alpha}$,
\ 
where $\Pii_{\mathcal S}\mathcal A= \bigcup_{i\in\alpha}\pi_i(A_i) \subseteq \frak D(A_0)$.   

By Lemma \ref{Lema_shufle many orders}, the limit structure  is a dense linear order completely coloured in $\alpha$ dense colors; it is a model of the theory from Example \ref{exm dense colors}. Assuming that the $A_i$ are pairwise disjoint, each $\pi_i:A_i\to\Pii_{\mathcal S}\mathcal A$   is an  order-isomorphism with the order of  the corresponding  color   and we can  consider a natural bijection $\pi: \bigcup_{i\in\alpha}A_i\to\Pii_{\mathcal S}\mathcal A$ given by $\pi=\bigcup_{i\in\alpha}\pi_i$ (that it is a bijection follows  by Lemma \ref{Lema_shufle many orders}).

\begin{prop}\label{Prop projekcije izomorfne iff} Suppose that $\mathcal A= ((A_i,<_{A_i})\mid i\in\alpha)$ and $\mathcal B= ((B_i,<_{B_i})\mid i\in\alpha)$ are countable sequences of countable linear orders   shuffled by families $\mathcal S_A=(S_{i,j}^A\mid i<j<\alpha)$ and $\mathcal S_B(S_{i,j}^B\mid i<j<\alpha)$ respectively. Then the following statements are equivalent:
\begin{enumerate}[(1)]
\item $(\Pii_{\mathcal S_A}\mathcal A,\subset, \pi^A_i(A_i))_{i\in\alpha}\cong  (\Pii_{\mathcal S_B}\mathcal B,\subset, \pi^B_i(B_i))_{i\in\alpha}$; 
\item There are order-isomorphisms $f_i:A_i\to B_i$ such that for all $i<j<\alpha$, $x\in A_i$ and $y\in A_j$: \begin{center}
$(x,y)\in S_{i,j}^A$ \  if and only if \ $(f_i(x),f_j(y))\in S_{i,j}^B$.
\end{center}
\item $(A_i,<_{A_i})\cong (B_i,<_{B_i})$ for all $i\in\alpha$.
\end{enumerate}
\end{prop}
\begin{proof}
(1)$\Rightarrow$(2) \ 
Suppose that  $\Pii_{\mathcal S_A}\mathcal A \cong\Pii_{\mathcal S_B}\mathcal B$ and let $f:\Pii_{\mathcal S_A}\mathcal A\longrightarrow\Pii_{\mathcal S_B}\mathcal B$ be an isomorphism. Since $f$ respects colors, $f_{\strok \pi_i(A_i)}$ is an order-isomorphism $(\pi^A_i(A_i),\subset)\longrightarrow(\pi^B_i(B_i),\subset)$, thus we can define an order-isomorphism $f_i:(A_i,<_{A_i})\longrightarrow (B_i,<_{B_i})$ by: \begin{center}
\ $f_i(a_i)=b_i$ \ iff \ $f(\pi^A_i(a_i))= \pi^B_i(b_i)$ \ (i.e.\ $f_i= (\pi_i^B)^{-1}\circ f_{\strok \pi_i^A(A_i)}\circ \pi_i^A$).
\end{center}
For $x\in A_i, y\in A_j$ we have the following equivalences: \begin{center}
$(x,y)\in S^A_{i,j}$ \ iff \ $\pi_i^A(x)\subset \pi_j^A(y)$ \  \ iff \  $f(\pi_i^A(x))\subset f(\pi_j^A(y))$ \  \ iff  \ $\pi_i^B(f_i(x))\subset \pi_i^B(f_j(y))$ \  \ iff \ $(f_i(x),f_j(y))\in S^B_{i,j}$ ;
\end{center}
The first and the last one hold by Lemma \ref{Lema_shufle many orders}(c), the second because   $f$ is an isomorphism and the third follows by the definition of $f_i$ and $f_j$. This proves (2).

\smallskip
(2)$\Rightarrow$(3) \ is trivial, so we prove (3)$\Rightarrow$(1).  Assume that $(A_i,<_{A_i})\cong (B_i,<_{B_i})$ holds for all $i$. 
Note that $A_i$ has minimum iff $\Pii_{\mathcal S_A}\mathcal A$ has minimum which is coloured by the $i$-th color  $\pi_i^A(A_i)$: Since $(A_i,<_{A_i})$ is order-isomorphic to $(\pi_i^A(A_i),\subset)$, the minimal element of $A_i$ is mapped onto the minimal element of $\pi_i^A(A_i)$, and this one is the minimal element of $\Pii_{\mathcal S_A}\mathcal A$ since by Lemma \ref{Lema_shufle many orders}  $\pi_i^A(A_i)$ is dense within $\Pii_{\mathcal S_A}\mathcal A$. Thus (3) implies that one of $\Pii_{\mathcal S_A}\mathcal A$ and $\Pii_{\mathcal S_B}\mathcal B$   has minimum iff both of them do, in which case these minimums are equally coloured. Similarly for maximums.

We can now construct an isomorphism by the usual back-and-forth argument: Both $\Pii_{\mathcal S_A}\mathcal A$ and $\Pii_{\mathcal S_B}\mathcal B$ are countable dense linear orders. By the previous discussion, if they exist, we map minimum (maximum) of $\Pii_{\mathcal S_A}\mathcal A$ to minimum (maximum) of $\Pii_{\mathcal S_B}\mathcal B$. The rest of the construction is standard and uses that every point in these orders is coloured, and every colour is dense within them.
\end{proof} 

\begin{exm}\label{Exm_shuffled _elquant}
Suppose that $\mathcal A= ((A_i,<_{i})\mid i\in\omega)$ is a countable sequence of disjoint  linear orders  of order type $\bm{\eta}$  shuffled by the family $\mathcal S=(S_{i,j}\mid i<j<\omega)$. Form a first order structure in the language $\{C_i,<_i,S_{i,j}\mid i<j<\omega\}$ in which each $C_i$ is interpreted as $A_i$, while $<_i$ and $S_{i,j}$ are interpreted in an obvious way. Let the $L$-theory $T$ describe that $\mathcal S$ shuffles the sequence of (order-type $\bm{\eta}$) orders $\mathcal A$. For $M\models T$, define $C(M):=\bigcup_{i<\omega}C_i(M)$ and $p(M)=M\smallsetminus C(M)$. By Proposition \ref{Prop projekcije izomorfne iff} for all countable $M,N\models T$ we have $C(M)\cong C(N)$ (as substructures) and $M\cong N$ iff $|p(M)|=|p(N)|$. By a standard argument one shows that any countable $M'\models T$ has a countable  elementary extension $M^*\succ M'$   with $|p(M^*)|=\aleph_0$; this implies that $T$ is complete, small and that  $M^*$ is saturated. We {\em claim} that any two tuples $\bar a,\bar b$ of $M^*$ having the same quantifier-free type are conjugated by an $M^*$-automorphism. To justify it, note that if $\bar a\subset C(M^*)$, then 
the images $\pi(\bar a)$ and $\pi(\bar b)$ have the same quantifier-free type in the limit structure (which eliminates quantifiers), so are conjugates by an automorphism $f$ of the limit structure;   $f$ naturally (as in the proof of $(1)\Rightarrow(2)$ in Proposition \ref{Prop projekcije izomorfne iff}) induces an automorphism of $C(M^*)$ moving $\bar a$ to $\bar b$. This proves the claim in the case $\bar a\subset C(M^*)$. In the general case  $\bar a$ and $\bar b$ may contain elements of $p(M)$; here it suffices to note that any permutation of $p(M)$ when extended by the identity map on $C(M^*)$ becomes an automorphism of $M^*$. 
Consequently, $T$ eliminates quantifiers. 
\end{exm}
 
\begin{exm}\label{Exm_shuffled _elquant2}
Consider the language $\{C_i,<,S_{i,j}\mid i<j<\omega\}$ and let $T$ be theory describing  that $(M,<)$ is a dense linear order without endpoints  with $C_0(M)<C_1(M)<...$ (topologically) open convex subsets  and $C(M):=\bigcup_{i\in\omega} C_i(M)$ its initial part, while the sequence $(S_{i,j}^M\mid i<j<\omega)$ shuffles the sequence of orders $((C_i(M),<)\mid i\in \omega)$. Put $p(M)=M\smallsetminus C(M)$. For all countable $M,N\models T$ we have $C(M)\cong C(N)$ (as substructures) and $M\cong N$ iff $(p(M),<)\cong (p(N),<)$. Arguing similarly as in the previous example, one shows that if $M^*$ is countable and $(P(M^*),<)$ has order type $\bm{\eta}$, then $M^*\models T$ is countably saturated and that   $T$ eliminates quantifiers. It follows that $T$ is binary and weakly \oo-minimal.
\end{exm}

\section{Stationarily ordered types}

In this section we introduce and start investigating stationarily ordered types in an arbitrary $T$.  Notably, we prove in Corollary \ref{Cor_ordered_nezavisi_odordera} that  stationarity of an ordered type does not depend on the choice of a witnessing order. Then we study the binary relation $x\dep_A y$; in Corollary \ref{Cor_forking_transitivity}  we prove that it is an equivalence relation on the set of all realizations of all stationarily ordered types over the fixed domain $A$. In Proposition \ref{cor_nwor_fwor_transitivity}  we deduce that both $\nwor$ and $\nfor$ are equivalence relations on the set of all stationarily ordered types over $A$.
Recall that  a type $p\in S_1(A)$ is ordered if there is an $A$-definable linear order $(D_p,<_p)$ with $p(\Mon)\subseteq D_p$, in which case we say that  $(D_p,<_p)$ witnesses that $p$ is ordered.
 
\begin{dfn}\label{Dfn_stat_type} Suppose that   $p\in S_1(A)$ is an ordered  type witnessed by   $(D_p,<_p)$. 

(a) We say that $<_p$ {\em stationarily orders} $p$, or that $\mathbf p=(p,<_p)$ is an {\it \so-pair} over $A$, if 
for every $\Mon$-definable set $D$  one of the sets $D$ and $\Mon\smallsetminus D$ is left eventual,  and one of them  is  right eventual in  $(p(\Mon),<_p)$.

(b) A complete type is {\em stationarily ordered} if there exists an ordering  which stationarily orders it.

(c) A complete theory $T$ is {\em stationarily ordered} if every   type from $S_1(\emptyset)$ is stationarily ordered.
\end{dfn}

Clearly,    weakly \oo-minimal theories are   stationarily ordered and in Section 9 we will show that weakly quasi \oo-minimal theories and theories of coloured orders are stationarily ordered, too. 

\begin{exm} (An $\aleph_0$-categorical, stationarily ordered theory having the independence property)\\
Consider  the complete theory $T$ of $(\mathbb Q,<,E,R)$, where $E$ is a convex equivalence relation on the ordered rationales such that the factor order $(\mathbb Q/E,<)$ is dense without endpoints, $R\subseteq E$ is a symmetric binary relation and on each $E$-class $R$ induces a random graph which is `independent' of $<$. This structure is $\aleph_0$-categorical and eliminates quantifiers. There is a unique complete 1-type and it is stationarily ordered. On the other hand, no complete type in $S_1(a)$ containing  $E(x,a)\land x\neq a$ is stationarily ordered. The theory $T$ has the independence property since the random graph is interpretable in our structure.
\end{exm}


\begin{rmk}
We have defined classes of ordered and stationarily ordered types of 1-types in 1-sorted theories.  However, there is an obvious way of extending the definition to an arbitrary multi-sorted context: it suffices to require that the witnessing order is chosen within the adequate sort. Note that for types in the real sort the property of being stationarily ordered is preserved under passing from $T$ to $T^{eq}$ and vice versa; the same holds for the relations $p\nwor q$ and $p\nfor q$. Although we will work mostly with types in the real sort, 
all the results from this section (and sections 4 and 6) describing properties of the relations of weak and forking non-orthogonality of stationarily ordered types  in an arbitrary first-order theory $T$ remain valid in the multi sorted context, too.   For example, Corollary  \ref{cor_nwor_fwor_transitivity} remains valid in $T^{eq}$:  $\nwor$ and $\nfor$ are equivalence relations on the set of all stationarily ordered types over a fixed domain  in $\Mon^{eq}$. 
\end{rmk}

\begin{dfn}For an \so-pair $\mathbf p=(p,<_p)$ over $A$  define its left ($\p_l$)  and right ($\p_r$) globalization:
\begin{itemize}
\item  \ $\p_l= \{\phi(x)\in L(\Mon)\mid \phi(\Mon)\mbox{ is left eventual in }p(\Mon)\}$, 
\item  \ $\p_r= \{\phi(x)\in L(\Mon)\mid \phi(\Mon)\mbox{ is right eventual in }p(\Mon)\}$.
\end{itemize}
\end{dfn}

 \begin{rmk}\label{Rmk notin p_l}
Let $\mathbf p=(p,<_p)$ be  an \so-pair over $A$ witnessed by $(D_p,<_p)$ and $\phi(x)\in L(\Mon)$. 

- Both $\p_l$ and $\p_r$ are complete, $A$-invariant  global types extending $p$. 

- The set $\phi(\Mon)$  is either $\mathbf{p}$-left-bounded or  is left-eventual in $(p(\Mon),<_p)$. 
In particular, $\phi(x)$ is $\mathbf p$-left-bounded if and only if $\phi(x)\notin \frak p_l$; similarly,  $\phi(x)$ is $\mathbf p$-right-bounded if and only if $\phi(x)\notin \frak p_r$. 

- If $q\in S_1(B)$ extends $p$, then $q\not\subseteq\p_l$ if and only if there is a   $\mathbf p$-left-bounded formula $\psi(x)\in q$ if and only if  (by Lemma \ref{fact osnovni}) there exists a strongly $\mathbf p$-left-bounded  $\psi(x)\in q$; similarly for $q\not\subseteq\p_r$.
 \end{rmk}
 
\begin{lem}\label{lem prop_left_right_basics} Suppose that $\mathbf p= (p,<_p)$ is an \so-pair over $A$ witnessed by $(D_p,<_p)$.
\begin{enumerate}[(a)]
\item $\p_l$ and $\p_r$ are the only global $A$-invariant extensions of $p$.
\item For any $B\supseteq A$, $\p_{l\strok B}(\Mon)$ is an initial part and $\p_{r\strok B}(\Mon)$ is a final part of $(p(\Mon),<_p)$. 
\item For all $a,b\models p$: \ $a\models\p_{l\strok Ab}$ \ if and only if  \ $b\models\p_{r\strok Aa}$. In other words, $(b,a)$ is a Morley sequence in $\p_l$ over $A$ \ if and only if \ $(a,b)$ is a Morley sequence in $\p_r$ over $A$.
\item If the order $(D,<)$ witnesses that  $p$ is ordered, then $\mathbf p'= (p,<)$ is an \so-pair    and   $\{\p_l',\p_r'\}= \{\p_l,\p_r\}$.
\end{enumerate}
\end{lem}
\begin{proof}
(a) By Remark \ref{Rmk notin p_l} $\p_l$ and $\p_r$ are $A$-invariant types. 
Let $\p$ be an $A$-invariant global extension of $p$. Then either $x<_pa$ belongs to $\p$ for all $a\models p$, or $a<_px$ does. Suppose that  the first  option holds  and let $\phi(x)\in\p$. Then the type $p(x)\cup\{\phi(x),x<_pa\}\subset \p$ is consistent for any $a\in p(\Mon)$; this means that $\phi(\Mon)$ contains arbitrarily $<_p$-small elements of $(p(\Mon),<_p)$, so $\phi(x)$ is not $\mathbf p$-left-bounded  and  by Remark \ref{Rmk notin p_l}   we have $\phi(x)\in \p_l$. 
Since $\phi(x)\in\p$ was arbitrary we deduce $\p\subseteq \p_l$ and hence $\p=\p_l$. Similarly, if the second option holds,  we get $\p=\p_r$.

\smallskip
(b) We  prove that   $P=p(\Mon)\smallsetminus \p_{l\strok B}(\Mon)$ is a final part  of $(p(\Mon),<_p)$. Assume that  $a\in P$, $a'\in p(\Mon)$ and $a<_pa'$.
Then  $\tp(a/B)\not\subseteq\p_l$, by Remark \ref{Rmk notin p_l}, implies that there is a  strongly $\mathbf p$-left-bounded formula $\phi(x)\in \tp(a/B)$. Then the   formula 
$\psi(x):=x\in D_p\land (\exists t)(\phi(t)\land t\leqslant_p x)$ is also strongly $\mathbf p$-left-bounded, so $\psi(x)\notin \p_l$. Clearly, $\psi(x)\in\tp(a'/B)$ witnesses $\tp(a'/B)\not\subseteq\p_l$, so $a'\in P$. This proves that  
 $P$ is a final part of $(p(\Mon),<_p)$ and it follows that  $\p_{l\strok B}(\Mon)$ is an initial part. Similarly,
 $\p_{r\strok B}(\Mon)$ is a final part of $(p(\Mon),<)$.

\smallskip
(c) Assume that $a\models\p_{l\strok Ab}$  and let $c\models p$ be such that  $a\models\p_{r\strok Ac}$; then $c<_pa<_pb$ holds. By part (b) of the lemma  we have $c\models\p_{l\strok Ab}$ and $b\models\p_{r\strok Ac}$. This implies  $\tp(a,c/A)= \tp(b,c/A)= \tp(b,a/A)$, where the first equality holds because $a,b\models\p_{r\strok Ac}$ and the second one follows by $c,a\models\p_{l\strok Ab}$. Now,   $\tp(a,c/A)= \tp(b,a/A)$, $a\models\p_{r\strok Ac}$ and the $A$-invariance of $\p_r$ imply $b\models\p_{r\strok Aa}$. 

\smallskip
(d) Assume that $(D,<)$ is an $A$-definable linear order and $p(\Mon)\subseteq D$. For all  $a\models p$, either $a<x$ or $x<a$ belongs to $\p_r$. We will continue the proof assuming $(a<x)\in \p_r$; the proof in the other case is obtained by applying the first case to $(D,<^*)$ where $<^*$ is the reverse order of $<$.
If $c\models \p_{r\strok Aa}$, then $(a,c)$ is a Morley sequence in $\p_r$ over $A$ so, by part (c) of the lemma, $(c,a)$ is a Morley sequence in $\p_l$ over $A$ and hence   $(x<a)\in\p_l$. The $A$-invariance of $\p_l$  implies  that $(x<c)\in\p_l$ for all $c\models p$.

We {\it claim} that  any right (left) eventual set in $(p(\Mon),<_p)$ is right (left)  eventual in $(p(\Mon),<)$, too; clearly, this implies that $\mathbf p'$ is an \so-pair over $A$, $\p_r=\p'_r$ and $\p_l=\p_l'$, and complets the proof of the lemma.  
By part (b) the set $R_B=\p_{r\strok B}(\Mon)$  for  $B\supset A$ is a final part of $(p(\Mon),<_p)$. Since any  definable  right $<_p$-eventual set contains   $R_B$ for some $B\supseteq A$,   in order to prove the claim for right eventual sets, it suffices to show that each  $R_B$ is a final part of $(p(\Mon),<)$; the proof for left eventual sets is analogous by considering initial parts $L_B=\p_{l\strok B}(\Mon)$. 
 Suppose, on the contrary, that $R_B$ is not the final part of $(p(\Mon),<)$:
$a\in R_B$, $b\models p$, $a<b$  and   $b\notin R_B$. By saturation, there exists $a'\in R_B$ such that $(a',a)$ is a Morley sequence in $\p_r$ over $A$; by part (c),  $(a,a')$ is a Morley sequence in $\p_l$ over $A$. Then   $a'<_pa$ and $a'<a$ hold; the latter is a consequence of $(x<a)\in\p_l$.   
Since   $R_B$ is a  final part of $(p(\Mon),<_p)$, $b\notin R_B$ and   $a'\in R_B$,   we have $b<_pa'$; this together with the fact that $(a,a')$ is a Morley sequence in $\p_l$ over $A$,  by part (b)   implies that $(a,b)$ is a Morley sequence in $\p_l$ over $A$. In particular,  $b<a$ holds; a contradiction. 
\end{proof}

 As an immediate consequence of part (d) of the previous lemma we have:

\begin{cor}\label{Cor_ordered_nezavisi_odordera} The property ``being stationarily ordered'' for an ordered type does not depend on the choice of the witnessing order. \qed
\end{cor}

\begin{lem}\label{lem_forking_iff_non_left_right} Suppose that  $\mathbf{p}=(p,<_p)$ is an \so-pair over $A$, $A\subseteq B$ and  $q\in S_1(B)$ is an extension of $p$. Then the  following conditions are all equivalent:
\begin{enumerate}[(1)]
\item $q(x)$ contains  a (strongly) $\mathbf{p}$-bounded formula.
\item $q(x)$ divides over $A$;
\item $q(x)$ forks over $A$;
\item $q(x)\nsubseteq \frak p_l(x)$ and   $q(x)\nsubseteq \frak p_r(x)$;
\end{enumerate}
\end{lem}
\begin{proof}
(1)$\Rightarrow$(2) suppose that $q(x)$ contains a $\mathbf p$-bounded formula. Then by Lemma \ref{fact osnovni} it contains a strongly $\mathbf p$-bounded formula $\phi(x,\bar b)$. If $a_0<_p\phi(\Mon,\bar b)<_pa_1$ for $a_0,a_1\models p$, then by considering $f\in \Aut(\Mon/A)$ such that $f(a_0)=a_1$ and a sequence $(\bar b_n)_{n\in\omega}$ defined by $\bar b_0=\bar b$ and $\bar b_{n+1}=f(\bar b_n)$ for $n\in\omega$, one easily sees that $\{\phi(x,\bar b_n)\mid n\in\omega\}$ is $2$-inconsistent, implying that $\phi(x,\bar b)$ and consequently $q(x)$ divides over $A$. 

(2)$\Rightarrow$(3) is trivial and (3)$\Rightarrow$(4) follows by $A$-invariance of $\p_l$ and $\p_r$. We prove   (4)$\Rightarrow$(1): if $q(x)\nsubseteq\p_l(x)$ and $q(x)\nsubseteq \p_r(x)$ then by Remark \ref{Rmk notin p_l} $q(x)$ contains a $\mathbf p$-left-bounded formula $\phi_1(x)$ and a $\mathbf p$-right-bounded formula $\phi_2(x)$. Thus   $q(x)$ contains a $\mathbf p$-bounded formula $\phi_1(x)\wedge\phi_2(x)$; by Lemma \ref{fact osnovni} it contains a strongly $\mathbf p$-bounded formula, too. 
\end{proof}

The following is an immediate consequence of the previous lemma. 

\begin{cor}\label{Cor_existence of nf extensions of so types}
A stationarily ordered type $p\in S_1(A)$ has exactly two global nonforking extensions:  $\p_l$ and $\p_r$ (where $\mathbf{p}=(p,<_p)$ is an \so-pair over $A$). The restrictions of $\p_l$ and $\p_r$ to  $B\supseteq A$  are the only nonforking extensions of $p$ over $B$.\qed
\end{cor}

\begin{dfn}\label{dfn_dependent_elements} Let $ \mathbf{p}=(p,<_p)$ be an \so-pair over $A$ and $\bar b\in\Mon$.
\begin{itemize}
\item $\Dp_A(\bar b)=\{x\in\Mon\mid x\dep_A\bar b\}$ \ and \ $\Dp(\bar b)=\Dp_{\emptyset}(\bar b)$;
\item $\Dp_p(\bar b)=\Dp_A(\bar b)\cap p(\Mon)$; 
\item $\Lp_{\mathbf{p}}(\bar b)=
 \frak p_{l\strok A\bar b}(\Mon)$;
\item $\Rp_{\mathbf{p}}(\bar b) =\frak p_{r\strok A\bar b}(\Mon)$.
\end{itemize}
\end{dfn}

\begin{rmk}\label{rmk_direct_properties} We list some properties of the introduced notions that follow  easily  by Lemmas \ref{lem prop_left_right_basics} and   \ref{lem_forking_iff_non_left_right}. Suppose that $\mathbf p=(p,<_p)$ is an \so-pair over $A$ and $\bar b\in \Mon$.

\smallskip
(a)   $p(\Mon)= \Lp_{\mathbf p}(\bar b)\cup\Rp_{\mathbf p}(\bar b)\cup\Dp_{p}(\bar b)$ \ and   $\Dp_p(\bar b)= (\Lp_{\mathbf p}(\bar b)\cup \Rp_{\mathbf p}(\bar b))^c$ (the relative complement in $p(\Mon)$).  Indeed, by Lemma \ref{lem_forking_iff_non_left_right} the set $\Lp_{\mathbf p}(\bar b)\cup \Rp_{\mathbf p}(\bar b)$ contains all elements of $p(\Mon)$  realizing non-forking extensions of $p$ in $S_1(A\bar b)$, while $\Dp_p(\bar b)$ contains all those realizing forking extensions.

\smallskip
(b)     $\Lp_{\mathbf p}(\bar b)$ is an initial part and $\Rp_{\mathbf p}(\bar b)$ is a final part of $(p(\Mon),<_p)$   by Lemma \ref{lem prop_left_right_basics}(b).

\smallskip
(c) $\Dp_p(\bar b)$ is a convex, bounded (possibly empty) subset of $(p(\Mon),<_p)$  because  it is obtained from $p(\Mon)$ by deleting its initial part $\Lp_{\mathbf p}(\bar b)$ and its  final part $\Rp_{\mathbf p}(\bar b)$.

\smallskip
(d)  The conditions  $a\ind_A\bar b$,  $a\in\Lp_{\mathbf p}(\bar b)\cup \Rp_{\mathbf p}(\bar b)$ and $a\notin\Dp_p(\bar b)$ are mutually equivalent  for all $a\models p$. 

\smallskip
(e)  For $a,a'\models p$ we have:  \  $a\in\Lp_{\mathbf p}(a')$ iff $a'\in\Rp_{\mathbf p}(a)$; \ this is Lemma \ref{lem prop_left_right_basics}(c), 

\smallskip
(f)  Forking over $A$ is a symmetric relation on  $p(\Mon)$. Indeed, it easily follows from    part (e)  that the relation defined by  $x\in\Lp_{\mathbf p}(y)\cup \Rp_{\mathbf p}(y)$ is symmetric on  $p(\Mon)$. 

\smallskip
(g)   For $a,a'\models p$, $a\ind_Aa'$ implies $a\in\Lp_{\mathbf p}(a')\cup\Rp_{\mathbf p}(a')$ so, by part (e), we have 
that one of $(a,a')$ and $(a',a)$ is a Morley sequence in $\frak p_l$, and the other in $\frak p_r$ over $A$; 
hence 
$\{\tp(a,a'/A),\tp(a',a/A)\}= \{\p^2_{l\strok A}, \p^2_{r\strok A}\}$.
\end{rmk}

\begin{lem}\label{Cor_Lpp_north_partition}
Suppose that  $\mathbf{p}=(p,<_p)$  is  an \so-pair over $A$ and $\bar b\in\Mon$.  

\begin{enumerate}[(a)]
\item \  If  $p\wor \tp(\bar b/A)$, then  $\Lp_{\mathbf{p}}(\bar b)= \Rp_{\mathbf{p}}(\bar b)=p(\Mon)$ and $\Dp_p(\bar b)=\emptyset$;
\item \ If $p\nfor \tp(\bar b/A)$, then $\Lp_{\mathbf{p}}(\bar b)<_p \Dp_p(\bar b)<_p \Rp_{\mathbf{p}}(\bar b)$ is a convex partition of $(p(\Mon),<_p)$; 
\item  \ If $p\nwor \tp(\bar b/A)$ and $p\fwor \tp(\bar b/A)$, then 
$\Lp_{\mathbf{p}}(\bar b)<_p \Rp_{\mathbf{p}}(\bar b)$ is  a convex partition of $(p(\Mon),<_p)$.
\end{enumerate}  
\end{lem} 
\begin{proof} We {\it claim} that   
 $p\wor\tp(\bar b/A)$ holds if and only if  
   $\Lp_{\mathbf p}(\bar b)\cap \Rp_{\mathbf p}(\bar b)\neq\emptyset$. Left-to-right direction is clear, so we prove the other one. Assume that $\Lp_{\mathbf p}(\bar b)\cap \Rp_{\mathbf p}(\bar b)\neq\emptyset$. Since $\Lp_{\mathbf p}(\bar b)$ and $\Rp_{\mathbf p}(\bar b)$ are loci of complete types  we have $\Lp_{\mathbf p}(\bar b)= \Rp_{\mathbf p}(\bar b)$. Here, by Remark \ref{rmk_direct_properties}(b), we have equality of an initial and a final part  of $p(\Mon)$, so $\Lp_{\mathbf p}(\bar b)= \Rp_{\mathbf p}(\bar b)=p(\Mon)$. Thus, $\frak p_{l\,\strok A\bar b}$ is the unique extension of $p$ in $S_1(A\bar b)$ and $p\wor \tp(\bar b/A)$ holds. This proves the {\it claim} and also part (a) of the lemma.  
   
To prove (b), it suffices to note that $p\nfor\tp(\bar b/A)$ implies $\Dp_p(\bar b)\neq \emptyset$, so the desired conclusion follows  by the claim and Remark \ref{rmk_direct_properties}(a)-(c). Similarly for part (c).
\end{proof}

\begin{dfn} Suppose that $p \in S_1(A)$ is a stationarily ordered type and  $D\subseteq p(\Mon)$. We say that $D$  is a {\em $\Dp_p$-closed} set if $a\in D$ implies $\Dp_p(a)\subseteq D$.
\end{dfn}

\begin{lem}\label{lem rmk closed for Dp} Suppose that $\mathbf p=(p,<_p)$ is an \so-pair over $A$ and $\bar b\in \Mon$.
\begin{enumerate}[(a)]
\item   $\Lp_{\mathbf p}(\bar b)$, $\Rp_{\mathbf p}(\bar b)$ and $\Dp_p(\bar b)$ are $\Dp_p$-closed subsets  of $p(\Mon)$.
\item For all $a,a'\models p$: \  $a\dep_A a'$ and $a\ind_A\bar b$ imply $a\equiv a'\,(A\bar b)$.
\item For all $a,a'\models p$ the following conditions  are equivalent:\\  (1)  $\Dp_p(a)<_p\Dp_p(a')$,; \ \ (2) $a<_p\Dp_p(a')$; \ \   (3) $\Dp_p(a)<_pa'$; \ \ (4)  $a\in \Lp_{\mathbf p}(a')$; \ \ (5) $a'\in \Rp_{\mathbf p}(a)$;\\ 
(6) $(a,a')$ is a Morley sequence in $\p_r$ over $A$;  \ \ (7)   $(a',a)$ is a Morley sequence in $\p_l$ over $A$.
\end{enumerate}
\end{lem}
\begin{proof}
(a) We will prove that $\Lp_{\mathbf p}(\bar b)$ is a $\D_p$-closed set. Let $a\in \Lp_{\mathbf p}(\bar b)$. For any $c\models p$ the set $\Dp_p(c)$ is non-empty (because $c\in\Dp_p(c)$) and bounded in $(p(\Mon),<_p)$; by saturation of the monster there is $c\models p$ satisfying $\Dp_p(c)<_pa$. Then, since $a$ belongs to $\Lp_{\mathbf p}(\bar b)$, which is an initial part of $p(\Mon)$, we have   $\Dp_p(c)\subseteq\Lp_{\mathbf p}(\bar b)$.
 Now $a,c\in \Lp_{\mathbf p}(\bar b)$ implies $a\equiv c\,(A\bar b)$, so there exists an $f\in\Aut(\Mon/{A\bar b})$ with $f(c)=a$. Then $f(\Lp_{\mathbf p}(\bar b))= \Lp_{\mathbf p}(\bar b)$ and $f(\Dp_p(c))=\Dp_p(a)$, so  $\Dp_p(c)\subseteq\Lp_{\mathbf p}(\bar b)$ implies $\Dp_p(a)\subseteq\Lp_{\mathbf p}(\bar b)$; hence $\Lp_{\mathbf p}(\bar b)$ is a $\Dp_p$-closed subset of $p(\Mon)$. 
 In a similar way one can show that $\Rp_{\mathbf p}(\bar b)$ is $\Dp_p$-closed. Finally, the complement of the $\Dp_p$-closed set $\Lp_{\mathbf p}(\bar b)\cup \Rp_{\mathbf p}(\bar b)$, that is  $\Dp_p(\bar b)$, is $\Dp_p$-closed by  symmetry of $x\in\Dp_p(y)$ on $p(\Mon)$.

\smallskip
(b) Suppose that $a\dep_A a'$ realize $p$ and $a\ind_A \bar b$. The latter, by Remark \ref{rmk_direct_properties}(d), implies 
$a\in \Lp_{\mathbf p}(\bar b)\cup\Rp_{\mathbf p}(\bar b)$. Since both sets in the union are $\Dp_p$-closed and $a'\in \Dp_p(a)$ holds (by symmetry), we deduce that either   $a,a' \in \Lp_{\mathbf p}(\bar b)$ or $a,a'\in\Rp_{\mathbf p}(\bar b)$; in either case we have $a\equiv a'(A\bar b)$. This proves part (b).  The proof of part (c)  is straightforward and left to the reader. 
\end{proof}

 \begin{lem}\label{Lema_ Dpa=Dpa' iff a dep a'}
Suppose that  $p$  is  a stationarily ordered type over $A$. For all $a,a'\models p$ we have:
\begin{center}
$a\dep_A a'$ \ \ if and only if \ \ $\Dp_p(a)=\Dp_p(a')$ \ \ if and only if \ \ $\Dp_p(a)\cap \Dp_p(a')\neq \emptyset$. 
\end{center}
\end{lem} 
\begin{proof} 
To prove the first equivalence assume $a\dep_A a'$. Then $a'\in \Dp_p(a)$  and, since the set $\Dp_p(a)$ is $\Dp_p$-closed by  Lemma \ref{lem rmk closed for Dp}(a), we have $\Dp_p(a')\subseteq \Dp_p(a)$. Since $\dep_A$ is symmetric on $p(\Mon)$  by Remark \ref{rmk_direct_properties}(f) we have $a'\dep_A a$ and arguing as above we get  $\Dp_p(a)\subseteq \Dp_p(a')$; thus $\Dp_p(a')=\Dp_p(a)$. This proves one direction of the first equivalence; the other one is trivial (because   $a\in\Dp_p(a)$).  

The left-to-right implication of the second equivalence is trivial. For the reverse implication, assume $b\in\Dp_p(a)\cap\Dp_p(a')$. Then  $b\dep_A a$ and $b\dep_A a'$ so, by the proven equivalence,  $\Dp_p(a)= \Dp_p(b)=\Dp_p(a')$.
\end{proof}

\begin{dfn}\label{Defn_D_p_eqrelation}
For a stationarily ordered type  $p\in S_1(A)$ define $\D_p=\{(x,y)\in p(\Mon)\times p(\Mon)\mid x\dep_A y\}$.
\end{dfn}

\begin{cor}\label{Cor_on p forking is equivalence}
Suppose that $(p,<_p)$ is an \so-pair over $A$. Then:

(a) \  $\D_p$ is a $<_p$-convex equivalence relation on $(p(\Mon),<_p)$ and its classes are sets $\Dp_p(a)$ ($a\in p(\Mon)$);

(b) \ If $p$ is non-algebraic, then the factor order $(p(\Mon)/\D_p,<_p)$ is a dense linear order without endpoints. 
\end{cor}
\begin{proof}
(a)  Lemma \ref{Lema_ Dpa=Dpa' iff a dep a'}  implies that $\D_p$ is an equivalence relation; its classes are sets $\Dp_p(a)$ ($a\models p$) and they are convex by Remark \ref{rmk_direct_properties}(c).

(b) Assume that $p$ is non-algebraic, $a_1<_p a_2$ realize $p$ and $\Dp_p(a_1)<_p\Dp_p(a_2)$. By Lemma \ref{Lema_ Dpa=Dpa' iff a dep a'}  we have $ a_1\ind_Aa_2$  which, by Lemma \ref{lem rmk closed for Dp}(c) implies that $(a_1,a_2)$ is a Morley sequence in $\p_r$ over $A$. By saturation of $\Mon$ there is $a\in p(\Mon)$ such that $(a_1,a,a_2)$ is a Morley sequence in $\p_r$ over $A$. Then $\Dp_p(a_1)<_p\Dp_p(a)<_p \Dp_p(a_2)$, so $(p(\Mon)/\D_p,<_p)$ is a dense linear order. It remains to notice that for any $a\models p$ the set $\Dp_p(a)$ is bounded in $(p(\Mon),<_p)$; hence  $\Dp_p(a)$ is not an endpoint.  
\end{proof}

\begin{lem}\label{lem_square_implies_disjoit_Ds} Suppose that $p,q\in S_1(A)$ are stationarily ordered types such that $p\nwor q$, and $a,a'\models p$ are such that $a\ind_Aa'$. Then:
\begin{enumerate}[(a)]
\item $\Dp_q(a)\cap \Dp_q(a')=\emptyset$;
\item $\Lp_{\mathbf q}(a)\neq \Lp_{\mathbf q}(a')$ and $\Rp_{\mathbf q}(a)\neq \Rp_{\mathbf q}(a')$, where $\mathbf q=(q,<_q)$ is any \so-pair  over $A$. 
\end{enumerate}
\end{lem}
\begin{proof} 
(a) If $q\fwor p$, then  $\Dp_q(x)=\emptyset$   for all $x\models p$ and the conclusion trivially follows, so let us assume  $q\nfor p$.  Towards a contradiction, suppose that $\Dp_q(a)\cap\Dp_q(a')\neq \emptyset$. By Remark \ref{rmk_direct_properties}(g), for all $c,c'\models p$, $c\ind_Ac'$ implies that one of $(c,c')$ and $(c',c)$ has the same type as $(a,a')$ over $A$. In either case $\Dp_q(c)\cap\Dp_q(c')\neq \emptyset$ follows.

By  Remark \ref{rmk_direct_properties}(c)  the sets  $\Dp_q(a)$ and $\Dp_q(a')$ are non-empty (because $q\nfor p$), convex, bounded subsets of $(q(\Mon),<_q)$. 
Hence  the set $\Dp_q(a)\cup \Dp_q(a')$ is bounded and, by saturation of the monster, there exists $a''\models p$ such that  $\Dp_q(a)\cup \Dp_q(a')<_q \Dp_q(a'')$. By the previous paragraph we get $a''\dep_Aa$ and $a''\dep_Aa'$, which by Corollary \ref{Cor_on p forking is equivalence} imply $a\dep_Aa'$. A contradiction.

\smallskip
(b) We will prove   $\Lp_{\mathbf q}(a)\neq \Lp_{\mathbf q}(a')$; the proof of the second inequality is similar. Towards a  contradiction, assume   $\Lp_{\mathbf q}(a)= \Lp_{\mathbf q}(a')$.
Let  $c\models p$ be arbitrary. Then at least one of $c\ind _A a$ and $c\ind _A a'$ holds (otherwise, by transitivity we would have $a\dep_A a'$). By Remark \ref{rmk_direct_properties}(g) $c\ind _A a$ implies $ac\equiv  aa'(A)$ or $ac\equiv a'a(A)$, and in both cases we have  $\Lp_{\mathbf q}(a)= \Lp_{\mathbf q}(c)$; similarly, $c\ind _A a'$ implies    $\Lp_{\mathbf q}(a)=\Lp_{\mathbf q}(c)$. We have just shown that the set $D=\Lp_{\mathbf q}(c)$ does not depend on the choice of $c\models p$. Clearly, $D\subseteq q(\Mon)$ is a non-empty $A$-invariant set, so $D=q(\Mon)$.   In particular, $\Lp_{\mathbf q}(a)
=q(\Mon)$ and every realization of $q$ realizes $\p_{l\strok Aa}$, so $p\wor q$. A contradiction.
\end{proof}

\begin{lem}\label{Prop_equivalents_of_a dep b}
Suppose that the types $p=\tp(a/A)$ and $q=\tp(b/A)$ are stationarily ordered. Then the following conditions are all equivalent:
\begin{enumerate}[(1)]
\item $a\dep_A b$;
\item $b\dep_A a$;
\item $(\Lp_{\mathbf{r}}(a),\Dp_r(a),\Rp_{\mathbf{r}}(a))=(\Lp_{\mathbf{r}}(b),\Dp_r(b),\Rp_{\mathbf{r}}(b))$  holds for all \so-pairs $\mathbf{r}=(r,<_r)$ over $A$;
\item $\Dp_r(a)=\Dp_r(b)$ holds for all  \so-pairs $\mathbf{r}=(r,<_r)$ over $A$;
\item $\Lp_{\mathbf{r}}(a)=\Lp_{\mathbf{r}}(b)\neq r(\Mon)$ holds for some  \so-pair $\mathbf{r}=(r,<_r)$ over $A$.
\end{enumerate}  
\end{lem}
\begin{proof}
(1)$\Leftrightarrow$(2) \ It suffices to prove one direction  of the equivalence, so suppose that $a\dep_A b$ holds. Then $a\in\Dp_p(b)$. For any $b'\equiv b\,(Aa)$ we have $a\in\Dp_p(b')$, hence  $a\in \Dp_p(b')\cap\Dp_p(b)\neq\emptyset$ and, by Lemma \ref{lem_square_implies_disjoit_Ds}(a), we get   $b'\dep_A b$ and $b'\in \Dp_q(b)$. This shows that the locus  of $\tp(b/Aa)$ is entirely contained in $\Dp_q(b)$, so $\tp(b/Aa)$ is contained   neither in  $\frak q_l$ nor in $\frak q_r$ and, by Lemma \ref{lem_forking_iff_non_left_right},  it is a forking extension of $q$; $b\dep_A a$ follows. This completes the proof of the equivalence.

\smallskip
(1)$\Rightarrow$(3) \ Let $\mathbf{r}=(r,<_r)$ be an \so-pair over $A$. We will prove  $\Lp_{\mathbf{r}}(a)=\Lp_{\mathbf{r}}(b)$; the equality of $\Rp_{\mathbf{r}}(a)$ and $\Rp_{\mathbf{r}}(b)$ is proved similarly and then the equality of $\Dp_r(a)$ and $\Dp_r(b)$  follows by Remark \ref{rmk_direct_properties}(a). Suppose on the contrary that $a\dep_A b$ and   $\Lp_{\mathbf{r}}(a)\neq \Lp_{\mathbf{r}}(b)$. Since these are initial parts of $r(\Mon)$, one of them is strictly contained in the other. 
We will assume  $\Lp_{\mathbf{r}}(b)\subset  \Lp_{\mathbf{r}}(a)$; the other case is proved analogously because $b\dep_Aa$ holds by the above proved symmetry. Choose $c\in\Lp_{\mathbf{r}}(b)$ and $c'\in\Lp_{\mathbf{r}}(a)\smallsetminus\Lp_{\mathbf{r}}(b)$. 
Then $c,c'\in\Lp_{\mathbf r}(a)$ implies $c\equiv c'\,(Aa)$, so there is $b'\models q$ such that $bc\equiv b'c'\,(Aa)$. 
From  $b\equiv b'\,(Aa)$ and $a\dep_A b$ we get  $a\dep_A b'$, so $a\in\Dp_p(b)\cap \Dp_p(b')$, which further implies $b\dep_Ab'$ by Lemma \ref{lem_square_implies_disjoit_Ds}(a). Since $c\in\Lp_{\mathbf r}(b)$ we have $c\ind_A b$, so    $bc\equiv b'c'\,(Aa)$ implies $c'\ind_A b'$. By the above proved symmetry ((1)$\Leftrightarrow$(2)) we have $b'\ind_Ac'$, which, together with $b\dep_A b'$, implies $bc'\equiv b'c'\,(A)$ by Lemma \ref{lem rmk closed for Dp}(b). By our choice of $b'$   we get $bc\equiv bc'\,(A)$, which  is impossible since $c\in\Lp_{\mathbf r}(b)$ and $c'\notin\Lp_{\mathbf r}(b)$.

\smallskip
(3)$\Rightarrow$(4) is trivial and for   (4)$\Rightarrow$(5) it suffices to note that $\Dp_p(a)=\Dp_p(b)\neq\emptyset$ and Lemma \ref{Cor_Lpp_north_partition} imply $\Lp_{\mathbf{p}}(a)=\Lp_{\mathbf{p}}(b)\neq p(\Mon)$.

\smallskip
(5)$\Rightarrow$(2). Suppose that $\mathbf{r}=(r,<_r)$ is is an \so-pair over $A$  and  $\Lp_{\mathbf{r}}(a)=\Lp_{\mathbf{r}}(b)\neq r(\Mon)$. For any $b'$ realizing $\tp(b/Aa)$ we have $\Lp_{\mathbf{r}}(b')=\Lp_{\mathbf{r}}(a)=\Lp_{\mathbf{r}}(b)$. By Lemma \ref{lem_square_implies_disjoit_Ds}(b) we have $b\dep_Ab'$. Hence, any realization of $\tp(b/Aa)$ belongs to $\Dp_q(b)$ and  $b\dep_Aa$ follows.
\end{proof}

As an immediate corollary of the equivalence of conditions (1) and (4)  we have that forking over $A$ is an equivalence relation on the set of all elements realizing a stationarily ordered type over $A$:

\begin{cor}\label{Cor_forking_transitivity} 
 $x\dep_A y$ is an equivalence relation on the set  $\{a\in\Mon\mid \mbox{$\tp(a/A)$ is an \so-type}\}$.  \qed
\end{cor}

\begin{prop}\label{cor_nwor_fwor_transitivity} 
$\nwor$ and $\nfor$ are equivalence relations on the set of all stationarily ordered types over a fixed domain.
\end{prop}
\begin{proof} Clearly, both relations are reflexive. Symmetry and transitivity of $\nfor$ follow from Corollary \ref{Cor_forking_transitivity}. Symmetry of $\nwor$ is easy and to verify transitivity   
suppose that $p,q,r\in S_1(A)$ are stationarily ordered types, $p\nwor q$ and $r\nwor q$. Then for any pair of elements  $(a,b)$ where $a\models p$ and $b\models r$ the sets $\Lp_{\mathbf{q}}(a)$ and $\Lp_{\mathbf{q}}(b)$ are proper initial parts of $q(\Mon)$ (where $\mathbf{q}=(q,<_q)$ is any \so-pair over $A$).
 By saturation, for one such pair  we have $\Lp_{\mathbf{q}}(a')\subset\Lp_{\mathbf{q}}(b')$ and for some other pair $\Lp_{\mathbf{q}}(b'')\subset\Lp_{\mathbf{q}}(a'')$. 
 Clearly,  the types $\tp_{x,y}(a',b'/A)$ and $\tp_{x,y}(a'',b''/A)$ are distinct completions of $p(x)\cup r(y)$. Hence  $p\nwor r$ and $\nwor$ is   transitive.  
\end{proof}


\section{Orientation}  

In this section we introduce the relation $\delta$  of direct non-orthogonality  of \so-pairs; roughly speaking,  $\delta(\mathbf p,\mathbf q)$  describes when it is that the \so-pairs $\mathbf p$ and $\mathbf q$ have the same direction (in particular, every definable function between their loci   is increasing modulo forking-dependence). We will show that $\delta$ is an equivalence relation refining the non-orthogonality relation by splitting  each $\nwor$-class  into two $\delta$-classes, so that $(p,<_p)$ has the same direction with either $(q,<_q)$ or $(q,>_q)$. 
  This will allow us to choose an ordering for each \so-pair so that all non-orthogonal \so-pairs   have the same direction,  and that will significantly  simplify the   description of the relationship between invariants  of non-orthogonal types. Let us emphasize that  in the  analysis of invariants, even if the underlying theory is weakly \oo-minimal (with respect to $<$), we will not stick to \so-pairs $(p,<)$, but for each type $p$ choose $<_p$ which is either $<$ or $>$ so that all non-orthogonal pairs become  directly  non-orthogonal.

Recall that
if $(p,<_p)$ is an \so-pair, then $\D_p=\{(x,y)\in p(\Mon)\times p(\Mon)\mid x\dep_A y\}$  is a convex equivalence relation on $(p(\Mon),<_p)$ and the $\D_p$-class of $a\models p$ is  the set $\D_p(a)=\{x\models p\mid x\dep_A a\}$. By Corollary \ref{Cor_on p forking is equivalence} the factor set  $p(\Mon)/\D_p=\{\Dp_p(a)\mid a\in p(\Mon)\}$ is densely linearly ordered by $<_p$ and has no endpoints. 
  Furthermore, if $(q,<_q)$ is an \so-pair, $p\nfor q$ and $b\models q$, then by Lemma \ref{Prop_equivalents_of_a dep b} $\D_p(b)=\D_p(a)$ for any $a\in\D_p(b)$, so $\D_p(b)$ is an element of $p(\Mon)/\D_p$, too.

\begin{lem}\label{Lema_supinf_LpRp} Suppose that $(p,<_p)$ is an \so-pair, $q\in S_1(A)$ is a stationarily ordered type, $p\nwor q$ and $b\models q$. Then:

(a) If $p\fwor q$, then the set $\Lp_{\mathbf p}(b)/\D_p$ is a right bounded initial part of $(p(\Mon)/\D_p,<_p)$ that has no supremum (and   
 $\Rp_{\mathbf p}(b)/\D_p$ is a left bounded final part that has no infimum);

(b) If $p\nfor q$, then  $\Dp_p(b)=\sup (\Lp_{\mathbf p}(b)/\D_p)=\inf (\Rp_{\mathbf p}(b)/\D_p)$ (in $(p(\Mon)/\D_p,<_p)$).
\end{lem}
\begin{proof}
By Lemma \ref{lem rmk closed for Dp}   $\Lp_{\mathbf p}(b)$ is a $\Dp_p$-closed subset  of $p(\Mon)$, so 
$\Lp_{\mathbf p}(b)/\Dp_p=\{\Dp_p(x)\mid x\in\Lp_p(b)\}\subseteq  p(\Mon)/\Dp_p$. $p\nwor q$ implies that $\Lp_{\mathbf p}(b)$ is a proper initial part of $p(\Mon)$, so 
$\Lp_{\mathbf p}(b)/\Dp_p$ is a proper, initial part of $p(\Mon)/\Dp_p$. 

(a) Assume that $c\models p$ and $\Dp_p(c)=\sup(\Lp_{\mathbf p}(b)/\Dp_p)$. Then the set $\Dp_p(c)$ is fixed (set-wise) by any $Ab$-automorphism and thus  the orbit of $\tp(c/Ab)$ under $Ab$-automorphisms is a bounded subset of $p(\Mon)$. This implies that $\tp(c/Ab)$ forks over $A$ and  hence  $p\nfor q$.

(b)   Assume $p\nfor q$ and choose $a\models p$ such that $b\dep_Aa$.  By Lemma \ref{Prop_equivalents_of_a dep b} we have   $(\Lp_{\mathbf{p}}(a),\Dp_p(a),\Rp_{\mathbf{p}}(a))=(\Lp_{\mathbf{p}}(b),\Dp_p(b),\Rp_{\mathbf{p}}(b))$. Since $\Lp_{\mathbf{p}}(a)<_p\Dp_p(a)<_p\Rp_{\mathbf{p}}(a)$ is a convex partition of $p(\Mon)$ into $\Dp_p$-closed sets, we have $\Dp_p(b)=\sup (\Lp_{\mathbf p}(b)/\D_p)=\inf (\Rp_{\mathbf p}(b)/\D_p)$ completing the proof of the lemma. 
\end{proof}

Stationarily ordered pairs $\mathbf{p}=(p,<_p)$ and $\mathbf{q}=(q,<_q)$ 
over the same domain $A$ are {\it non-orthogonal},  denoted by $\mathbf{p}\nwor\mathbf{q}$, if $p\nwor q$ holds. 
By Proposition \ref{cor_nwor_fwor_transitivity}, non-orthogonality is an equivalence relation on the set of all \so-pairs over $A$.

\begin{lem}\label{Lema discussion direct nonorth}
Suppose that $\mathbf p$ and $\mathbf q$ are \so-pairs over $A$ and $\mathbf{p}\nwor \mathbf{q}$.  Then exactly one of the following two conditions holds:

(1) \ $\Lp_{\mathbf{q}}(a_1)\subseteq \Lp_{\mathbf{q}}(a_2)$   for all $a_1<_pa_2$ realizing $p$ (equivalently: $\Lp_{\mathbf{q}}(a_1)\subset \Lp_{\mathbf{q}}(a_2)$   for all independent $a_1<_pa_2$ realizing $p$);

(2) \  $\Lp_{\mathbf{q}}(a_2)\subseteq \Lp_{\mathbf{q}}(a_1)$   for all $a_1<_pa_2$ realizing $p$ (equivalently: $\Lp_{\mathbf{q}}(a_2)\subset \Lp_{\mathbf{q}}(a_1)$   for all independent $a_1<_pa_2$ realizing $p$). 
\end{lem}
\begin{proof} For   $a_1<_pa_2$ realizing $p$ we have two possibilities: The first is   $a_1\dep_A a_2$, in which case we have $\Lp_{\mathbf{q}}(a_1)=\Lp_{\mathbf{q}}(a_2)$  by the equivalence of conditions (1) and (3) from Lemma \ref{Prop_equivalents_of_a dep b}. The second is $a_1\ind_A a_2$, in which case $\Lp_{\mathbf{q}}(a_1)\neq\Lp_{\mathbf{q}}(a_2)$ follows by the equivalence of conditions (1) and (5) from Lemma \ref{Prop_equivalents_of_a dep b}. Then,  since $\Lp_{\mathbf{q}}(a_1)$ and $\Lp_{\mathbf{q}}(a_2)$ are initial parts, we have that exactly one of $\Lp_{\mathbf{q}}(a_1)\subset\Lp_{\mathbf{q}}(a_2)$ and $\Lp_{\mathbf{q}}(a_2)\subset\Lp_{\mathbf{q}}(a_1)$ holds (for all independent pairs $a_1<_p a_2$). Therefore, exactly one  of the conditions (1) and (2)   is satisfied.
\end{proof}

\begin{dfn}\label{dfn direct}
 {\it So}-pairs $\mathbf{p}=(p,<_p)$ and $\mathbf{q}=(q,<_q)$ over the same domain  are {\em directly} non-orthogonal, denoted by $\delta(\mathbf{p},\mathbf{q})$,  if $p\nwor q$ and $\Lp_{\mathbf{q}}(a_1)\subseteq  \Lp_{\mathbf{q}}(a_2)$   for all $a_1<_p a_2$ realizing $p$. 
\end{dfn}

\begin{lem}\label{Lema_R_q_komplement}
 Suppose that $\mathbf p$ and $\mathbf q$ are   \so-pairs over $A$, $\mathbf p\nwor\mathbf q$  and   $a,a'\models p$. Then the following conditions are equivalent:

(1) \ $\Lp_{\mathbf q}(a)\subset \Lp_{\mathbf q}(a')$;

(2) \ $\Lp_{\mathbf q}(a)\cup \Dp_q(a)\subset \Lp_{\mathbf q}(a')$, i.e.\ $\Rp_{\mathbf q}(a)^c\subset\Lp_{\mathbf q}(a')$ (where the complement is taken in $q(\Mon)$).\\
If, in addition, $p\nfor q$ holds, then one more equivalent condition is:

(3) $\Dp_q(a)<_q\Dp_q(a')$.
\end{lem}
\begin{proof}
Recall that $\Rp_{\mathbf q}(a)^c=\Lp_{\mathbf q}(a)\cup\Dp_q(a)$ by Remark \ref{rmk_direct_properties}(a) and Lemma \ref{Cor_Lpp_north_partition}(b,c) as $p\nwor q$. If $p\fwor q$, then $\Dp_q(a)=\emptyset$, so conditions (1) and (2) are equivalent. 
Assume   $p\nfor q$. (2)$\Rightarrow$(3) and  (3)$\Rightarrow$(1) are easy, so we prove only
(1)$\Rightarrow$(2).  Assume $\Lp_{\mathbf q}(a)\subset \Lp_{\mathbf q}(a')$. By Corollary \ref{Cor_on p forking is equivalence}(b) $(q(\Mon)/\D_q,<_q)$ is a dense linear order and $\Lp_{\mathbf q}(a)/\D_q\subset \Lp_{\mathbf q}(a')/\D_q$ are its initial parts. By Lemma  \ref{Lema_supinf_LpRp}(b) we have $\Dp_q(a)=\sup(\Lp_{\mathbf q}(a)/\D_q)$, so  
 $\Dp_q(a)\subseteq \Lp_{\mathbf q}(a')$ and thus  $\Lp_{\mathbf q}(a)\cup \Dp_q(a)\subseteq \Lp_{\mathbf q}(a')$. It remains to prove $\Lp_{\mathbf q}(a)\cup \Dp_q(a)\neq \Lp_{\mathbf q}(a')$; otherwise, $\Dp_q(a')=\sup(\Lp_{\mathbf q}(a')/\D_q)$ would imply $\Dp_q(a)=\Dp_q(a')$ and  $\Lp_{\mathbf q}(a)=\Lp_{\mathbf q}(a')$; a contradiction. Therefore, condition (2) is satisfied. 
\end{proof}

As a corollary of Lemma \ref{Lema discussion direct nonorth} and Lemma \ref{Lema_R_q_komplement} we have:
\begin{cor}\label{Cor_delta_q_q*} If $\mathbf{p}=(p,<_p)$ and $\mathbf{q}=(q,<_q)$ are non-orthogonal \so-pairs over $A$, then 
exactly one of $\delta(\mathbf{p},\mathbf{q})$ and   $\delta(\mathbf{p},\mathbf{q}^*)$ holds, where $\mathbf q^*=(q,>_q)$. 
\end{cor}
\begin{proof}
Notice that $\q_l^*= \q_r$ and $\q_r^*=\q_l$. Fix independent realizations  $a<_pa'$ of $p$, and assume $\lnot\delta(\mathbf p,\mathbf q)$. By Lemma \ref{Lema discussion direct nonorth} we have $\Lp_{\mathbf q}(a')\subset\Lp_{\mathbf q}(a)$, so by Lemma \ref{Lema_R_q_komplement} we have $\Rp_{\mathbf q}(a')^c\subset\Lp_{\mathbf q}(a)$, i.e.\ $\Rp_{\mathbf q}(a')^c\subset\Rp_{\mathbf q}(a)^c$ since $\Lp_{\mathbf q}(a)\subseteq\Rp_{\mathbf q}(a)^c$ by Lemma \ref{Cor_Lpp_north_partition}. Thus $\Rp_{\mathbf q}(a)\subset\Rp_{\mathbf q}(a')$, i.e.\ $\Lp_{\mathbf q^*}(a)\subset\Lp_{\mathbf q^*}(a')$ as $\q_l^*=\q_r$. Therefore, $\delta(\mathbf p,\mathbf q^*)$ by Lemma \ref{Lema discussion direct nonorth}. By symmetry, $\lnot\delta(\mathbf p,\mathbf q^*)$ implies $\delta(\mathbf p,\mathbf q)$ as $\mathbf q^{**}=\mathbf q$, so exactly one of $\delta(\mathbf p,\mathbf q)$ and $\delta(\mathbf p,\mathbf q^*)$ holds.
\end{proof}

\begin{lem}\label{Rmk dfn delta}  Suppose that $\mathbf{p}=(p,<_p)$ and $\mathbf{q}=(q,<_q)$ are non-orthogonal \so-pairs over $A$. 
\begin{enumerate}[(a)]
\item  If    $p\nfor q$, then 
 $\delta(\mathbf p,\mathbf q)$ if and only if the mapping $\Dp_p(x)\longmapsto \Dp_q(x)$ is an isomorphism between $(p(\Mon)/\Dp_p,<_p)$ and $(q(\Mon)/\Dp_q,<_q)$. 

\item If $\mathbf p'=(p,<)$ is an \so-pair over $A$, then $\delta(\mathbf p,\mathbf p')$ if and only if $<_p$ and $<$ agree on pairs of independent realizations of $p$.
\end{enumerate}
\end{lem}
\begin{proof}
(a) Assume $p\nfor q$.  By Lemma \ref{Prop_equivalents_of_a dep b} for $a, b\in p(\Mon)\cup q(\Mon)$   we have:
\begin{center}
$a\dep_Ab$ \ \ iff \ \ $\Dp_p(a)=\Dp_p(b)$ \ \ iff \ \ $\Dp_q(a)=\Dp_q(b)$.
\end{center} 
This   implies that the given mapping is well-defined; it is onto by saturation of $\Mon$ and it is injective  because $\dep_A$ is symmetric and transitive. Now assume $\delta(\mathbf p,\mathbf q)$ and $\Dp_p(a)<_p\Dp_p(a')$. Then $\Lp_{\mathbf q}(a)\subset \Lp_{\mathbf q}(a')$ by Lemma \ref{Lema discussion direct nonorth}, so by Lemma \ref{Lema_R_q_komplement} we have $\Dp_q(a)<_q \Dp_q(a')$. Hence our mapping is strictly increasing; it is an order isomorphism. This proves one direction of the equivalence in (b); the other is proved similarly.  Part  (b)    follows  easily from (a).
\end{proof}

\begin{rmk}\label{rmk new defn direct}
If  $\mathbf{p}=(p,<_p)$ and  $\mathbf{p'}=(p,<)$ are  \so-pairs  over $A$, then  $p\nfor p$ and by Lemma \ref{Rmk dfn delta}(b)  we have: $\delta(\mathbf{p},\mathbf{p'})$   if and only if  $<$ and $<_p$ agree on independent realizations of $p$.  Hence:

--  \ \ $\delta(\mathbf{p},\mathbf{p'})$ if and only if  
 $\frak p_l=\frak p_l'$ (equivalently, $\frak p_r=\frak p_r'$);
 
--   $\lnot \delta(\mathbf{p},\mathbf{p'})$   if and only if 
 $\frak p_l=\frak p_r'$ (equivalently, $\frak p_r=\frak p_l'$). 
\end{rmk}

Direct non-orthogonality of $\mathbf{p}$ and $\mathbf{q}$, when $p\nwor q$, may be equivalently expressed by:
\begin{center}
 $(\Lp_{\mathbf{q}}(x)\mid x\in p(\Mon))$ is an increasing sequence of initial parts of $(q(\Mon),<_q)$.
\end{center} 
When stated in that way it suggests a connection with certain monotone relations.

\begin{dfn}\label{Def_Spq} For  \so-pairs $\mathbf{p}$ and $\mathbf{q}$ over the same domain let  $S_{\mathbf{p},\mathbf{q}}\subseteq p(\Mon)\times q(\Mon)$ be the relation defined by $x\in\Lp_{\mathbf{p}}(y)$. 
\end{dfn} 

\begin{rmk}\label{Rmk def Spq} Let   $\mathbf{p}=(p,<_p)$ and $\mathbf{q}=(q,<_q)$ be \so-pairs over $A$.

(a) $S_{\mathbf{p},\mathbf{q}}$ is  type-definable over $A$: it is the locus of $\tp(a,b/A)$ where $b\models q$ and $a\models \frak p_{l\strok Ab}$. 

(b) $S_{\mathbf{p},\mathbf{q}}$ is  a $\Dp_p\times\Dp_q$-closed subset of $p(\Mon)\times q(\Mon)$, i.e.\ $(a,b)\in S_{\mathbf{p},\mathbf{q}}$ implies $\Dp_p(a)\times\Dp_q(b)\subseteq S_{\mathbf{p},\mathbf{q}}$. Indeed, if $(a,b)\in S_{\mathbf{p},\mathbf{q}}$, then   $a\ind_Ab$, so if $a\dep_Aa'$ and $b\dep_A b'$ for $a'\models p$ and $b'\models q$, then by applying Lemma \ref{lem rmk closed for Dp}(b) twice we get $ab\equiv a'b\equiv a'b'\,(A)$ and   $(a',b')\in S_{\mathbf{p},\mathbf{q}}$ follows.
\end{rmk}

\begin{lem}\label{Lema_direct_non-orth}
Suppose that $\mathbf{p}=(p,<_p)$ and $\mathbf{q}=(q,<_q)$ are non-orthogonal  \so-pairs over $A$. Then the following conditions are equivalent:

(1) \   $\delta(\mathbf{p},\mathbf{q})$; 

(2) \ For all $a\models p$ and $b\models q$:  \ $a\in\Lp_{\mathbf{p}}(b)$ if and only if $b\in\Rp_{\mathbf{q}}(a)$; 

(3)  \  $S_{\mathbf{p},\mathbf{q}}$ is a monotone relation between $(p(\Mon),<_p)$ and $(q(\Mon),<_q)$;

(4) \ $\delta(\mathbf{q},\mathbf{p})$.
\end{lem}
\begin{proof} 
(1)$\Rightarrow$(2) Notice that each of $x\in\Lp_{\mathbf{p}}(y)$ and $y\in\Rp_{\mathbf{q}}(x)$ determines a complete extension of $p(x)\cup q(y)$. Hence in order to prove the equivalence in (2), it suffices to prove only one of its directions. So suppose that (1) holds and let $a\models p$ and $b\models q$. Assuming  $a\in\Lp_{\mathbf{p}}(b)$ we will prove $b\in\Rp_{\mathbf{q}}(a)$. Let $a'\in\Rp_{\mathbf{p}}(ba)$.   Then $a'\in\Rp_{\mathbf{p}}(b)$ and  $a\in\Lp_{\mathbf{p}}(b)$ imply  $ab\not\equiv a'b\,(A)$. Also, $a'\in\Rp_{\mathbf{p}}(a)$ implies  $a<_pa'$. By direct non-orthogonality, we have  $\Lp_{\mathbf{q}}(a)\subset\Lp_{\mathbf{q}}(a')$, so if $b\in\Lp_{\mathbf{q}}(a)$ were true, we would have $b\in\Lp_{\mathbf{q}}(a')$ and hence $ba\equiv ba'\,(A)$; contradiction. Hence $b\notin\Lp_{\mathbf{q}}(a)$. On the other hand, $a\ind_A b$ implies $b\notin\Dp_q(a)$, so the only remaining possibility is $b\in\Rp_{\mathbf{q}}(a)$. 

\smallskip
(2)$\Rightarrow$(3) \ Suppose that condition  (2) is satisfied. To prove the monotonicity assume that  $a'<_pa$, $b<_qb'$ and $(a,b)\in S_{\mathbf{p},\mathbf{q}}$. The latter means $a\in\Lp_{\mathbf{p}}(b)$ from which, by (2), we derive $b\in\Rp_{\mathbf{q}}(a)$. Then,  since  $\Rp_{\mathbf{q}}(a)$ is a final part and  $b<_qb'$, we have $b'\in \Rp_{\mathbf{q}}(a)$. By applying (2) again we get $a\in \Lp_{\mathbf{p}}(b')$ which,   since $\Lp_{\mathbf{p}}(b')$ is an initial part and $a'<_pa$ holds, implies $a'\in\Lp_{\mathbf{p}}(b')$. Hence $(a',b')\in S_{\mathbf{p},\mathbf{q}}$ and $S_{\mathbf{p},\mathbf{q}}$ is a monotone relation.  

\smallskip
(3)$\Rightarrow$(4) Suppose that $S_{\mathbf{p},\mathbf{q}}$ is a monotone relation. Then $a\in \Lp_{\mathbf{p}}(b)$ and $b<_qb'$ imply $a\in\Lp_{\mathbf{p}}(b')$. Hence  $\Lp_{\mathbf{p}}(b)\subseteq \Lp_{\mathbf{p}}(b')$ holds for all $b<_qb'$ realizing $q$, i.e.\ $\delta(\mathbf{q},\mathbf{p})$ holds.  

\smallskip
(4)$\Rightarrow$(1) By now, we have proven (1)$\Rightarrow$(4) and hence the other direction holds, too.    
\end{proof}

\begin{lem}\label{Lema_delta_implies_S+pq_commute}
Suppose that $\mathbf{p}=(p,<_p)$, $\mathbf{q}=(q,<_q)$ and   $\mathbf{r}=(r,<_r)$ are stationarily ordered pairs  over $A$ such that    $\delta(\mathbf{q},\mathbf{p})$ holds.
Then  $S_{\mathbf{q},\mathbf{r}}\circ S_{\mathbf{p},\mathbf{q}}=S_{\mathbf{p},\mathbf{r}}$\,.
\end{lem}
\begin{proof}
We have to show that for all $a\models p$ and $c\models r$ the following holds:
\begin{center}
 $a\in \Lp_{\mathbf{p}}(c)$ \ if and only if \ there exists $b\in\Lp_{\mathbf{q}}(c)$ such that $a\in\Lp_{\mathbf{p}}(b)$. 
 \end{center} 
For the  left-to-right direction, assume that $a\in\Lp_{\mathbf{p}}(c)$. Choose  $b_0\in\Lp_{\mathbf{q}}(c)$ and $a_0\in\Lp_{\mathbf{p}}(cb_0)$. Then $a_0c\equiv ac(A)$  and any $b$ satisfying $a_0b_0c\equiv abc(A)$ witnesses that the condition on the right hand side is fulfilled. 

Suppose that the other direction does not hold: there exist $a\notin\Lp_{\mathbf{p}}(c)$ and $b\models q$  such that $b\in\Lp_{\mathbf{q}}(c)$ and $a\in\Lp_{\mathbf{p}}(b)$. Then, since $\Lp_{\mathbf{p}}(c)$ and $\Lp_{\mathbf{p}}(b)$ are initial parts, we have $\Lp_{\mathbf{p}}(c)\subset \Lp_{\mathbf{p}}(b)$. By saturation there is $b'$ realizing $q$ such that $\Lp_{\mathbf{p}}(b')\subset\Lp_{\mathbf{p}}(c)$. Then $\delta(\mathbf{q},\mathbf{p})$ implies $b'<_qb$. Since $\Lp_{\mathbf{q}}(c)$ is an initial part containing $b$, we have $b'\in\Lp_{\mathbf{q}}(c)$ and  hence $bc\equiv b'c(A)$. But this is in contradiction with $\Lp_{\mathbf{p}}(b')\subset \Lp_{\mathbf{p}}(c)\subset\Lp_{\mathbf{p}}(b)$.
\end{proof}

\begin{prop} 
$\delta$ is an equivalence relation on   the set of all \so-pairs over a fixed domain. It refines the   $\nwor$-equivalence   and  splits   each $\nwor$-class into  two $\delta$-classes.  
 \end{prop}
\begin{proof}
Reflexivity is clear and symmetry follows by Lemma \ref{Lema_direct_non-orth}.  For transitivity, assume $\delta(\mathbf p,\mathbf q)$ and $\delta(\mathbf q,\mathbf r)$. Then by Lemma \ref{Lema_delta_implies_S+pq_commute} $S_{\mathbf{q},\mathbf{r}}$ and  $S_{\mathbf{p},\mathbf{q}}$ are monotone relations and by Lemma \ref{Lema_delta_implies_S+pq_commute} $S_{\mathbf{q},\mathbf{r}}\circ S_{\mathbf{p},\mathbf{q}}=S_{\mathbf{p},\mathbf{r}}$. Since the compositions of monotone relations is a monotone relation, $S_{\mathbf{p},\mathbf{r}}$ is a monotone relation, so $\delta(\mathbf{p},\mathbf{r})$ follows by 
Lemma \ref{Lema_direct_non-orth}. Finally, Corollary \ref{Cor_delta_q_q*} implies that the $\nwor$-class of  $\mathbf p$ is split into  two $\delta$-classes: the one containing $\mathbf p$ and the other $\mathbf p^*$. 
\end{proof}
 
\begin{cor}\label{cor_direct_choice}  Let $T$ be a complete first order theory and $\mathcal F\subseteq S_1(\emptyset)$ any  set of pairwise non-weakly orthogonal, stationarily ordered types. Then  there is a choice of  definable  orderings  such that the \so-pairs $\{(p,<_p)\mid p\in\mathcal F\}$ are pairwise directly non-orthogonal. \qed
\end{cor}

\section{Dp-minimality and stationarily ordered theories}\label{s dp}

In this section we prove that any theory which is stationarily ordered in a strong sense has to be  dp-minimal  and satisfy tp=Lstp; the latter property motivated us to choose the name:  stationarily ordered types. It would be interesting to know if the converse  is true or not:

\begin{question}
Is every dp-minimal ordered theory satisfying tp=Lstp stationarily ordered?
\end{question}

We will not recall the  original definition of dp-minimality, but rather a characterisation due to Kaplan and Simon which will be used in the proof.

\begin{fact}[\cite{KS}, Corollary 1.7]\label{dp fact char} A theory $T$ is dp-minimal if and only if for all   $A$, all $A$-mutually indiscernible sequences of {\it singletons} $I$ and $J$, and all singletons $c$, at least one of $I$ and $J$ is $Ac$-indiscernible.
\end{fact} 

 \begin{lem}\label{dp lem fin sat} Suppose that  $\mathbf p=(p,<_p)$ is an so-pair over $A$. Then:
\begin{enumerate}[(a)]
\item If $B\supseteq A$ and $q\in S_1(B)$ is finitely satisfiable in $A$ and extends $p$, then $q=\p_{l\strok B}$ or $q=\p_{r\strok B}$.
\item If $(a_i)_{i\in\omega}$ is a decreasing (increasing) sequence of realizations of $p$ such that $\tp(a_i/Aa_{<i})$ is finitely satisfiable in $A$ for all $i\in\omega$, then $(a_i)_{i\in\omega}$ is a Morley sequence in $\p_l$ ($\p_r$) over $M$. 
\end{enumerate}
\end{lem}
\begin{proof}
If $q\in S_1(B)$ extends $p$ and is finitely satisfiable in $A$, then $q$ does not fork over $A$ so, by Corollary \ref{Cor_existence of nf extensions of so types}, we have $q=\p_{l\strok B}$ or $q=\p_{r\strok B}$.
This proves part (a). Part (b) is an easy consequence of part (a).
\end{proof}

\begin{thm}\label{thm_stat_implies_dp_min} If all complete 1-types over small sets are stationarily ordered, then $T$ is dp-minimal and Lstp=tp.
\end{thm}
\begin{proof} The second claim is easy: if $d_1,d_2$ realize the same   stationarily ordered type  $p\in S_1(A)$, then we can choose a Morley sequence $I$ in  $\frak p_r$ over $Ad_1d_2$ so that $d_1\,^\frown I$ 
and $d_2\,^\frown I$ are Morley sequences in $\frak p_r$ over $A$. Hence Lstp($d_1/A$)=Lstp($d_2/A$).

To prove the other claim, assume that $T$ is not dp-minimal.
By Fact \ref{dp fact char} there exist some $A$, $A$-mutually indiscernible sequences of singletons $I=(a_i\mid i\in\omega)$ and $J=(b_j\mid j\in\omega)$ and a singleton $c$ such that neither $I$ nor $J$ is $Ac$-indiscernible. After some standard modifications of $A$, $I$ and $J$, we may assume that the sequences $(\tp(a_i/Ac)\mid i<\omega)$ and $(\tp(b_j/Ac)\mid j<\omega)$ are non-constant. Without loss of generality assume that $a_0\not\equiv a_1\,(Ac)$ and $b_0\not\equiv b_1\,(Ac)$. Witness that by formulas $\phi(x,z),\psi(y,z)\in L(A)$ satisfying:   
\begin{equation}\models\phi(a_0,c)\wedge \neg\phi(a_1,c)\wedge \psi(b_0,c)\wedge \neg\psi(b_1,c).
\end{equation}
If we extend the original sequences to (mutually indiscernible) sequences of order-type $\bm{\omega}+\bm{\eta}$ and then absorb  both $\bm{\eta}$-parts into $A$, then we get:  
\begin{center}
\mbox{ $\tp(a_i/AJa_{<i})$ and $\tp(b_j/AIb_{<j})$ are finitely satisfied in $A$ for all  $i,j\in\omega$.}
\end{center}
To complete the proof of the theorem it suffices to show that at least one of the types $p=\tp(a_0/A), q=\tp(b_0/A)$  and $r=\tp(c/A)$ is not stationarily ordered. 
Towards a contradiction, suppose that $p,q$ and $r$ are stationarily ordered types. By (1) we have $p\nwor r$ and $q\nwor r$, so by Corollary \ref{cor_direct_choice} there are $A$-definable orderings $<_p, <_q$ and $<_r$ such that the \so-pairs $\mathbf{p}=(p,<_p)$, $\mathbf q=(q,<_q)$ and $\mathbf r=(r,<_r)$ are pairwise directly non-orthogonal. Without loss of generality, assume that the sequence $I$ is $<_p$-increasing. Since $\tp(a_i/AJa_{<i})$ is finitely satisfiable in $A$, Lemma \ref{dp lem fin sat} applies and  $I$ is a Morley sequence in $\frak p_r$. Then $\delta(\mathbf p,\mathbf r)$ implies $\Lp_{\mathbf r}(a_0)\subset    \Lp_{\mathbf r}(a_1) \subset \Lp_{\mathbf r}(a_2)\subset   \ldots$. Further, by Lemma  \ref{Lema_R_q_komplement} we have $\Rp_{\mathbf r}(a_i)^c=\Lp_{\mathbf r}(a_i)\cup \Dp_r(a_i)\subset \Lp_{\mathbf r}(a_{i+1})$ (where the complement is taken within $r(\Mon)$). Summing up, we have:
\begin{equation}\Lp_{\mathbf r}(a_0)\subseteq \Rp_{\mathbf r}(a_0)^c\subset \Lp_{\mathbf r}(a_1)\subseteq \Rp_{\mathbf r}(a_1)^c\subset \Lp_{\mathbf r}(a_2)\subseteq \Rp_{\mathbf r}(a_2)^c\subset \ldots
\end{equation}
Then $\models \phi(a_0,c)\land\lnot\phi(a_1,c)$ implies $\Lp_{\mathbf r}(a_0)<_r c$: otherwise, we would have $c\in \Lp_{\mathbf r}(a_0)\subset \Lp_{\mathbf r}(a_1)$ and hence $a_0c\equiv a_1c\,(A)$, which is impossible. Similarly  we get $c<_r \Rp_{\mathbf r}(a_1)$. Hence $c\in D_0$, where  $D_i=\Rp_{\mathbf r}(a_{i+1})^c\smallsetminus \Lp_{\mathbf r}(a_i)$ is a convex subset of $r(\Mon)$ for all $i\in\omega$.
Now we apply similar reasoning to the sequence $J$. We have two   cases to consider:

Case 1. \  $J$ is $<_q$-increasing. In this case, by arguing as above, we conclude that (2) holds with each $a_i$ replaced by $b_i$. Then $\models \psi(b_0,c)\land\lnot \psi(b_1,c)$ implies $c\in E_0$, where $E_j=\Rp_{\mathbf r}(b_{j+1})^c\smallsetminus \Lp_{\mathbf r}(b_j)$. \ Hence $c\in D_0\cap E_0\neq \emptyset$.

Case 2. \  $J$ is $<_q$-increasing. In this case (by duality) we have: 
\begin{center}$\Rp_{\mathbf r}(b_0)\subseteq \Lp_{\mathbf r}(b_0)^c\subset \Rp_{\mathbf r}(b_1)\subseteq \Lp_{\mathbf r}(b_1)^c\subset \Rp_{\mathbf r}(b_2)\subseteq \Lp_{\mathbf r}(b_2)^c\subset \ldots$
\end{center}
Then $\models \psi(b_0,c)\land\lnot \psi(b_1,c)$
implies $c\in E_0$  where   $E_j=\Lp_{\mathbf r}(b_{j+1})^c\smallsetminus \Rp_{\mathbf r}(b_j)$. Hence $c\in D_0\cap E_0\neq \emptyset$. 

In both cases we concluded that there is a convex set $E_0$ which is invariant under $Ab_0b_1$-automorphisms and meets $D_0$ (which is also convex and $Aa_0a_1$-invariant). 
Since the sequences are  mutually indiscernible  $D_i\cap E_j\neq\emptyset$ holds for all $i,j\in\omega$.  

Now we have three convex sets $D_1<_rD_3<_rD_5$ and a convex set   $E_0$ meeting each of them; $E_0$ has to contain the middle one: $D_3\subseteq E_0$. By mutual indiscernibility  $D_i\subseteq E_j$ holds for all $i,j$. By changing roles of $D$'s and $E$'s we get $E_j\subseteq D_i$ and thus $E_i=D_j$. In particular, $D_1=D_3$. Contradiction. 
\end{proof} 

\section{Invariants}\label{s inv and nonorth}

In this section we introduce invariants $\Inv_{\mathbf{p}}(M)$ and show that they behave very well with respect to non-orthogonality. 
$T$ is still an arbitrary first-order theory, we do {\em not} assume even that it is ordered.

\begin{notat}  For   a stationarily ordered type  $p\in S_1(A)$  and a model $M\supseteq A$ define:
\begin{itemize}
\item $\Dp_p=\{(a,b)\in p(\Mon)\times p(\Mon)\mid a\dep_A b\}$;

\item $\Dp_p^M=\Dp_p\cap (p(M)\times p(M))$.
\end{itemize}
\end{notat}  
Note that if $\mathbf p=(p,<_p)$ is an \so-pair over $A$ and $M\supseteq A$ realizes $p$, then the relations 
$\Dp_p$ and $\Dp_p^M$ are $<_p$-convex equivalence relation on $p(\Mon)$ and $p(M)$  respectively.

\begin{dfn}\label{dfn Inv} For a stationarily ordered pair $\mathbf{p}=(p,<_p)$ over $A$ and a model $M\supseteq A$, the $\mathbf{p}$-invariant of $M$, denoted by $\Inv_{\mathbf{p}}(M)$, is the order-type of $(p(M)/\Dp_p^M,<_p)$ if $p(M)\neq\emptyset$ and   $\Inv_{\mathbf{p}}(M)=\mathbf 0$ if $p$ is omitted in $M$. 
\end{dfn}

\begin{rmk}\label{rmk def invariants} Suppose that $p\in S_1(A)$,  $M\supseteq A$ is a model and that $\mathbf p=(p,<_p)$ and  $\mathbf p'=(p,<)$ are \so-pairs over $A$. 

-- If $\delta(\mathbf p,\mathbf p')$ holds, then by  Lemma \ref{Rmk dfn delta}(b),   $<$ and $<_p$ agree on pairs of independent realizations of $p$, hence   $(p(M)/\Dp_p^M,<_p)=(p(M)/\Dp_p^M,<)$   and 
 $\Inv_{\mathbf p}(M)= \Inv_{\mathbf p'}(M)$. 
 
 -- If $\lnot\delta(\mathbf p,\mathbf p')$ holds, then   $(p(M)/\Dp_p^M,<_p)=(p(M)/\Dp_p^M,>)$  and 
 $\Inv_{\mathbf p}(M)$ is the reverse of  $\Inv_{\mathbf p'}(M)$.  \\
Therefore, for a fixed model $M\supseteq A$ and  a stationarily ordered type $p\in S_1(A)$, there are at most two possibilities for $\Inv_{\mathbf p}(M)$ where $\mathbf p=(p,<)$ is an \so-pair over $A$.    
\end{rmk}

We will describe the impact of direct non-orthogonality of \so-pairs  $\mathbf{p}$ and $\mathbf{q}$ on the relationship between the  orders $(p(\Mon)/\Dp_p,<_p)$ and $(q(\Mon)/\Dp_q,<_q)$. We will prove that the  orders are   canonically isomorphic if $p\nfor q$  and canonically shuffled if $p\fwor q$ holds. If in addition $p,q$ are  convex types, then we prove that the the same holds for the orders $(p(M)/\Dp_p^M,<_p)$ and $(q(M)/\Dp_q^M,<_q)$ (where $M\supseteq A$ is a model and $p(M),q(M)\neq\emptyset$); 
we will give an example  showing that 
 the convexity assumption  cannot be omitted. 
 First  we  deal with the (easier) case  $p\nfor q$, which will be resolved in Proposition \ref{prop_bounded case}.

  Recall that a type $p\in S_1(A)$ is convex if there is an $A$-definable linear order $(D_p,<_p)$ such that $p(\Mon)$ is a convex subset of $D_p$; $(D_p,<_p)$ witnesses the convexity of $p$.
In the next lemma we show that the convexity of an \so-type is implied by a significantly weaker convexity condition; we say that a subset $D$ of a $\leqslant$-partially ordered set is convex if $x,x'\in D$ and $x\leqslant y\leqslant x'$ imply $y\in D$.  
   
\begin{lem}\label{Lema_convex_so_pair_witness}
Suppose that $\mathbf{p}=(p,<_p)$ is an \so-pair over $A$ witnessed by $(D_p,<_p)$ and that $\leqslant$ is an $A$-definable {\em partial} order on $\Mon$ such that every 
Morley sequence in $\p_r$ is strictly $\leqslant$-increasing and $p(\Mon)$ is a $\leqslant$-convex subset of $\Mon$. Then   there exists a formula $\phi(x)\in p$ implying $x\in D_p$ such that $(\phi(\Mon),<_p)$ witnesses the convexity of $p$. In particular, $p$ is a convex type.
\end{lem}
\begin{proof}
Fix $a\models p$. Since Morley sequences in $\p_r$ are $\leqslant$-increasing  we have  $(a\leqslant x)\in \p_{r\strok Aa}$.
Let $\Sigma_0(a,x)$ be the partial type consisting of all formulas \ $\forall y (\theta(y)\Rightarrow y<_px)$ \
where  $\theta(y)$ is a consistent, strongly $\mathbf p$-bounded formula  over $Aa$ implying $y\in D_p$. By Lemma \ref{lem_forking_iff_non_left_right} any forking extension $q(y)\in S_1(Aa)$ contains  a strongly $\mathbf p$-bounded formula  over $Aa$, so $\Sigma_0(a,x)$ expresses $\Dp_p(a)<_p x$. Clearly, $\Sigma_0(a,x)\cup p(x)\vdash  \p_{r\,\strok Aa}(x)$ and hence $\Sigma_0(a,x)\cup p(x)\vdash a\leqslant x$. By compactness there is $\phi_0(x)\in p$ such that  $\Sigma_0(a,x)\cup \{\phi_0(x)\}\vdash a\leqslant x$; in other words, $(\Dp_p(a)<_px)\cup \{\phi_0(x)\}\vdash a\leqslant x$.

By Lemma \ref{lem prop_left_right_basics}(c), $(a,b)$ is a Morley sequence in $\p_r$ over $A$ iff $(b,a)$ is a Morley sequence in $\p_l$ over $A$, so every Morley sequence in $\p_l$ over $A$ is strictly $\leqslant$-decreasing. Therefore, arguing similarly as in the previous paragraph we can find $\phi_1(x)\in p$ with  $(x<_p\Dp_p(a)) \cup \{\phi_1(x)\}\vdash x\leqslant a$. 

Let\, $\phi(x):=(x\in D_p)\land\phi_0(x)\land\phi_1(x)$.\,   We claim that $p(\Mon)$ is   $<_p$-convex in $\phi(\Mon)$. Suppose that $b,b'\models p$, $c\in \phi(\Mon)$ and $b<_pc<_pb'$. Choose $b_0,b_1\in\Mon$ such that $(b_0,b)$ and $(b',b_1)$ are Morley sequences in $\p_r$ over $A$; then $\Dp_p(b_0)<_pb<_pc<_pb'<_p\Dp_p(b_1)$. Now  $c\in\phi_0(\Mon)$ and $\Dp_p(b_0)<_pc$ imply $b_0\leqslant c$, and  $c\in\phi_1(\Mon)$ and $c<_p\Dp_p(b_0)$ imply $c\leqslant b_1$. Since $p(\Mon)$ is  $\leqslant$-convex  and since  $b_0\leqslant c\leqslant b_1$  we have $c\models p$. Therefore, $p(\Mon)$ is  a $<_p$-convex subset of $\phi(\Mon)$; in particular, $p$ is a convex type. 
\end{proof}

\begin{dfn}
If $\mathbf p=(p,<_p)$ is an \so-pair  over $A$ witnessed by $(D_p,<_p)$ and   $p(\Mon)$ is   $<_p$-convex in $D_p$, then we say that  $\mathbf p$ is a convex \so-pair (witnessed by  $(D_p,<_p)$).
\end{dfn} 

In the next lemma we show that for a convex type $p$ any linear order witnessing that it is stationarily ordered, nearly witnesses that $p$ is convex, too.  
 
 \begin{lem}\label{Lema_convex_so_pair_witness1} Suppose that $(D_p,<_p)$ witnesses that $\mathbf p=(p,<_p)$ is an \so-pair over $A$ and that $p$ is a convex type. Then there is an $A$-definable set   $D\subseteq D_p$ containing $p(\Mon)$ such that $(D,<_p^D)$ (where $<_p^D$ is the restriction of $<_p$ to $D$) witnesses  that $(p,<_p^D)$ is a convex \so-pair. 
\end{lem}
\begin{proof} Since $p$ is convex there is an $A$-definable linear order $(D_p',<_p')$ such that $p(\Mon)$ is a $<_p'$-convex subset of $D_p'$. 
 By Lemma \ref{lem prop_left_right_basics}(d)   for the \so-pair  $\mathbf p'=(p,<_p')$ we have that either $\p_r'=\p_r$ or $\p_l'=\p_r$ holds. First suppose  $\p_r'=\p_r$ and
define:  $x\leqslant y$ iff $x=y\lor (x\in D_p'\land y\in D_p'\land x<_p'y)$.  \ Then $\leqslant$ is   
an $A$-definable partial order on $\Mon$ and every Morley sequence in $\p_r$ is strictly $\leqslant$-increasing. By Lemma \ref{Lema_convex_so_pair_witness} there is $\phi(x)\in p$ implying $x\in D_p$  such that $p(\Mon)$  is $<_p$-convex in $D=\phi_p(\Mon)$; this settles the first case. The second case $\p_l'=\p_r$ is handled similarly
by   defining: $x\leqslant y$ iff $x=y\lor (x\in D_p'\land y\in D_p'\land y<_p'x)$. 
 \end{proof}
 
\begin{lem}\label{Lema a dep b witness}
Suppose that $\mathbf p= (p,<_p)$ and $\mathbf q=(q,<_q)$ are convex  \so-pairs over $A$ witnessed by $(D_p,<_p)$ and $(D_q,<_q)$, $a\models p$, $b\models q$ and $a\dep_A b$. Then there is a formula $\theta(x,y)\in\tp(a,b/A)$ such that $\theta(a,\Mon)\subseteq \Dp_{q}(a)$ and $\theta(\Mon,b)\subseteq\Dp_{p}(b)$;  in particular, $\theta(a,y)\vdash q(y)$ and $\theta(x,b)\vdash p(x)$.
\end{lem}
\begin{proof} 
By Lemma \ref{lem_forking_iff_non_left_right}, there is a formula   $\theta_1(x,y)\in \tp(a,b/A)$ with $\theta_1 (a,y)$   strongly $\mathbf{q}$-bounded; since $q(\Mon)$ is a convex subset of $D_q$, we have $\theta_1(a,\Mon)\subseteq \Dp_{q}(a)$. Similarly, there is $\theta_2(x,y)\in\tp(a,b/A)$ with $\theta_2( \Mon,b)\subseteq \Dp_{p}(b)$. Clearly, $\theta(x,y):=\theta_1(x,y)\land \theta_2(x,y)$ works. 
\end{proof}

 \begin{prop}\label{prop_bounded case} Let $\mathbf p= (p,<_p)$ and $\mathbf q=(q,<_q)$ be \so-pairs over $A$ such that $\delta(\mathbf p,\mathbf q)$ and $p\nfor q$.
\begin{enumerate}[(a)]
\item  The mapping $\Dp_p(x)\mapsto\Dp_q(x)$ defines  an isomorphism between $(p(\Mon)/\Dp_p,<_p)$ and $(q(\Mon)/\Dp_q,<_q)$.
\item If $p$ and $q$ are convex types and $M\supseteq A$ is a model, then  the mapping $\Dp^M_p(x)\mapsto\Dp^M_q(x)$ is   an isomorphism between $(p(M)/\Dp_p^M,<_p)$ and $(q(M)/\Dp_q^M,<_q)$, and $\Inv_{\mathbf p}(M)=\Inv_{\mathbf q}(M)$.
\end{enumerate}
\end{prop}
\begin{proof}
Part (a) is a reformulation of Lemma \ref{Rmk dfn delta}(a), so we prove only part 
(b).  By Lemma \ref{Lema_convex_so_pair_witness1}, after replacing $D_p$ by an appropriate $A$-definable subset, we may assume that   $(D_p,<_p)$ witnesses the convexity of $\mathbf p$; similarly, assume  that $(D_q,<_q)$ witnesses the convexity of $\mathbf q$. Since the given mapping is a restriction of the one from (a), we only need to show that it is well-defined and onto. 
For the first task it suffices for $a\in p(M)$ to find $b\in q(M)$ such that $a\dep_Ab$. So  let $a\in p(M)$  and choose $b'\in q(\Mon)$ such that $b'\dep_Aa$; this is possible by $p\nfor q$. By Lemma \ref{Lema a dep b witness}  there is a formula $\theta(x,y)\in \tp(a,b'/A)$ such that $\theta(a,\Mon)\subseteq \Dp_q(a)$ and $\theta(\Mon,b')\subseteq \Dp_p(b')$.
Since $\theta(a,y)$ is consistent, it is satisfied by some $b\in M$. Consequently $b\models q$ and $b\dep_Aa$, so the mapping is well defined. 
Similarly,  $\exists x\, \theta(x,y)\in q(y)$ implies that for every $b\in q(M)$ there exists $a\in p(M)$ with $a\dep_Ab$; the mapping is onto.
\end{proof}

\begin{exm}\label{Exm_nonconvex_distinct_Invs}
The assumed convexity of types in Proposition \ref{prop_bounded case}(b) cannot be omitted.\\
 Consider the structure  $(\mathbb Q,<,E,D_n)_{n\in\omega}$ where $(\mathbb Q,<,D_n)_{n\in\omega}$
is a dense linear order (completely) coloured by the dense colors $D_n$  (from Example  \ref{exm dense colors}) and $E$ is a convex equivalence relation   such that $(\mathbb Q/E, <)$ is a dense linear order without endpoints. The theory of this structure eliminates quantifiers. Let  $p_n\in S_1(\emptyset)$ be the type of an element of color $D_n$; $\mathbf p_n=(p_n,<)$ is a convex \so-pair witnessed by $(D_n,<)$ and $(p_n(\mathbb Q),<)$ has order type $\bm{\eta}$. Let $q$ be the type of a colorless element;  it is not convex and   is omitted in our structure.
If $a\in D_n$, then $q(x)\cup\{E(x,a)\}$ is consistent, has bounded locus in $q(\Mon)$ and thus forks over $\emptyset$; hence $p_n\nfor q$. Summing up: $p_n\nfor q$ are stationarily ordered types, $p_n$ is  convex, $q$ is not  convex, and $\bm{\eta}=\Inv_{\mathbf p_n}(\mathbb Q)\neq \Inv_{\mathbf q}(\mathbb Q)=\mathbf 0$. 
\end{exm} 

We now open the harder case, which will be completed in Propositions   \ref{prop_shuffled case general}   and \ref{prop_shuffled case local}.

\begin{dfn}
Complete types $p,q\in S_1(A)$ are {\it shuffled}, denoted by $\mathcal S(p,q)$, if they are stationarily ordered, $p\nwor q$ and $p\fwor q$; \so-pairs  $\mathbf{p}$ and $\mathbf{q}$ are {\it directly shuffled}, denoted by   $\mathcal S(\mathbf{p},\mathbf{q})$, if they are directly non-orthogonal and their types are shuffled.
\end{dfn}

Recall that by Lemma \ref{Lema_direct_non-orth}   $S_{\mathbf{p},\mathbf{q}}=\{(x,y)\in p(\Mon)\times q(\Mon)\mid x\in \Lp_{\mathbf p}(y)\}$   is a monotone relation between $(p(\Mon),<_p)$ and $(q(\Mon),<_q)$ provided that $\delta(\mathbf p, \mathbf q)$; also,
by Remark \ref{Rmk def Spq} $S_{\mathbf{p},\mathbf{q}}$ is a $\Dp_p\times\Dp_q$-closed relation: $(x,y)\in S_{\mathbf{p},\mathbf{q}}$ implies $\Dp_p(x)\times \Dp_q(y)\subseteq S_{\mathbf{p},\mathbf{q}}$.

\begin{lem}\label{rmk def shuffled}
Suppose that the  \so-pairs $\mathbf{p}=(p,<_p)$ and $\mathbf{q}=(q,<_q)$ are  shuffled. Then:

(a) The type $p(x)\cup q(y)$ has exactly two completions in $S_2(A)$: the first one determined by $x\in 
\Lp_{\mathbf p}(y)$ and the second one by $x\in \Rp_{\mathbf p}(y)$.

(b) If $\delta(\mathbf p,\mathbf q)$, then any formula $\sigma(x,y)$ belonging to the first  but not to the second  completion of  $p(x)\cup q(y)$ relatively defines the relation $S_{\mathbf{p},\mathbf{q}}$ within $p(\Mon)\times q(\Mon)$. 
\end{lem}
\begin{proof}
(a) Let $a\models p$ and $b\models q$. 
Then $p\fwor q$ implies $a\ind _A b$, so by Remark \ref{rmk_direct_properties}(d) either $a\in \Lp_{\mathbf{p}}(b)$ (i.e $a\models \p_{l\strok Ab}$)  or $a\in \Rp_{\mathbf{p}}(b)$ ($a\models \p_{r\strok Ab}$);  since  $p\nwor q$ implies $\p_{l\strok Ab}\neq \p_{r\strok Ab}$, we have exactly two possibilities for $\tp(a,b/A)$.   

(b) Assume $\delta(\mathbf p,\mathbf q)$. For all $a\models p$, $b\models q$ and $\sigma(x,b)\in \p_{l\strok Ab}(x)\smallsetminus \p_{r\strok Ab}(x)$, we have: 
  \ $\models\sigma(a,b)$    if and only if   $a\in \Lp_{\mathbf p}(b)$.
This means that $\sigma(x,y)$ relatively defines  $S_{\mathbf{p},\mathbf{q}}$ within $p(\Mon)\times q(\Mon)$. 
\end{proof}

\begin{lem}\label{lem_shuff rel def} Suppose that    $\mathbf p= (p,<_p)$ and $\mathbf q=(q,<_q)$ are \so-pairs over $A$  witnessed by $(D_p,<_p)$ and $(D_q,<_q)$   and such that $\delta(\mathbf p,\mathbf q)$.  
Then there exists a formula $\sigma(x,y)\in L(A)$ defining a monotone relation between $(D_p,<_p)$ and $(D_q,<_q)$ and relatively defining 
 the relation  $S_{\mathbf p,\mathbf q}$  within $p(\Mon)\times q(\Mon)$.

\end{lem}
\begin{proof}
(a) By  Lemma \ref{rmk def shuffled}(b) there is a formula $\sigma'(x,y)$ relatively defining $S_{\mathbf p,\mathbf q}$ in $p(\Mon)\times q(\Mon)$.
Since $S_{\mathbf p,\mathbf q}$ is a monotone relation, by Lemma \ref{Fact_rel definable monotone relation} there are $A$-definable sets $D_p'\subseteq D_p$ and $D_q'\subseteq D_q$ such that $\sigma'(x,y)$ relatively defines a monotone relation between $(D_p',<_p)$ and $(D_q',<_q)$. Then the formula:
$$\sigma(x,y):=x\in D_p\wedge y\in D_p\wedge (\exists x'\in D_p')(\exists y'\in D_q')(x\leq_p x'\wedge y'\leq_q y\wedge\sigma'(x',y'))\,$$
defines a monotone relation between $(D_p,<_p)$ and $(D_q,<_q)$.
\end{proof}

\begin{prop}\label{prop_shuffled case general} If $\mathbf p= (p,<_p)$ and $\mathbf q=(q,<_q)$ are directly shuffled \so-pairs over $A$, then
 \[\hat S_{\mathbf p,\mathbf q}=\{(\Dp_p(a),\Dp_q(b))\mid (a,b)\in  S_{\mathbf p,\mathbf q}\}\] 
 is a   shuffling relation between $(p(\Mon)/\Dp_p,<_p)$ and $(q(\Mon)/\Dp_q,<_q)$.
\end{prop}
\begin{proof}
The relation $\hat S_{\mathbf p,\mathbf q}$ is obtained from a $\Dp_p\times \Dp_q$-closed monotone  relation $S_{\mathbf p,\mathbf q}\subseteq p(\Mon)\times q(\Mon)$ by factoring out convex equivalence relations: $\Dp_p$ in the first coordinate and $\Dp_q$ in the second.
Because the quotient by a convex equivalence relation respects  the ordering it will suffice to verify the following:
\begin{enumerate}[(i)]
\item For all $b,b'\models q$: \ $\Dp_q(b)<_q\Dp_q(b')$ \ implies \ $S_{\mathbf p,\mathbf q}(p(\Mon),b)\subset S_{\mathbf p,\mathbf q}(p(\Mon),b')$;

\item For all $a,a'\models p$: \ $\Dp_p(a)<_p\Dp_p(a')$ \ implies \ $S_{\mathbf p,\mathbf q}(a',q(\Mon))\subset S_{\mathbf p,\mathbf q}(a,q(\Mon))$; and

\item For all $b\models q$ the set $S_{\mathbf p,\mathbf q}(p(\Mon),b)$ does not have a supremum $\Dp_p$-class; equivalently, neither $S_{\mathbf p,\mathbf q}(p(\Mon),b)$ contains a maximal $\Dp_p$-class nor its complement in $p(\Mon)$ contains a minimal $\Dp_p$-class.
\end{enumerate}

To prove (i) assume that  $b, b'\models q$ satisfy $\Dp_q(b)<\Dp_q(b')$. Then  $\delta(\mathbf{p},\mathbf{q})$  by Lemma \ref{Lema discussion direct nonorth} implies   $\Lp_{\mathbf p}(b)\subset  \Lp_{\mathbf p}(b')$. By definition of $S_{\mathbf p,\mathbf q}$  we have $S_{\mathbf p,\mathbf q}(p(\Mon),b)=\Lp_{\mathbf p}(b)$ and $S_{\mathbf p,\mathbf q}(p(\Mon),b')=\Lp_{\mathbf p}(b')$, so  the proof of part (i) is complete. 
The proof of (ii) is similar, after noticing that for  $a\models p$ we have:
\[S_{\mathbf p,\mathbf q}(a,q(\Mon))=\{b\in q(\Mon)\mid a\in \Lp_{\mathbf p}(b)\}=\{b\in q(\Mon)\mid b\in \Rp_{\mathbf q}(a)\}
 =\Rp_{\mathbf q}(a),\]
where the second equality holds by Lemma \ref{Lema_direct_non-orth}. To prove (iii) notice that any two elements of $S_{\mathbf p,\mathbf q}(p(\Mon),b)=\Lp_{\mathbf p}(b)$ have the same type over $Ab$, so there is an $Ab$-automorphism mapping one element onto the other. Hence for any pair of $\Dp_p$-classes contained in $\Lp_{\mathbf p}(b)$  there is an $Ab$-automorphism mapping one onto the other, so $\Lp_{\mathbf p}(b)$ does not contain a maximal $\Dp_p$-class. Finally, since $p$ and $q$ are shuffled, the complement of $S_{\mathbf p,\mathbf q}(p(\Mon),b)$ in $p(\Mon)$ is equal to $\Rp_{\mathbf p}(b)$, so the same argument shows that this complement does not contain a minimal $\Dp_p$-class.
\end{proof}

The next lemma treats a special case of Proposition \ref{prop_shuffled case local}; the proof of the  general case will be reduced to this one.
We will say that a stationarily ordered type $p\in S_1(A)$ is {\em trivial} if   $\D_p(a)=\{a\}$ for all $a\in p(\Mon)$.

\begin{lem}\label{lem_shuff rel def1}
Suppose that    $\mathbf p= (p,<_p)$ and $\mathbf q=(q,<_q)$ are convex \so-pairs over $A$  witnessed by $(D_p,<_p)$ and $(D_q,<_q)$,  $p,q\in S_1(A)$ are trivial types,    $\delta(\mathbf p,\mathbf q)$, and $M\supseteq A$ is a model.   Then:

(a) If $|q(M)|\geqslant 2$, then   $S_{\mathbf p,\mathbf q}$ is a shuffling relation between (infinite) orders $(p(\Mon),<_p)$ and $(q(\Mon),<_q)$.

(b) If $M$ realizes both $p$ and $q$, then  $S_{\mathbf p,\mathbf q}$ is a shuffling relation between $(p(\Mon),<_p)$ and $(q(\Mon),<_q)$. 
\end{lem}
\begin{proof}
Let  $\sigma(x,y)$ be a formula  given by  Lemma \ref{lem_shuff rel def}: $\sigma(x,y)$  defines a monotone relation between $(D_p,<_p)$ and $(D_q,<_q)$ and relatively defines
 the  relation $S_{\mathbf p,\mathbf q}$   within $p(\Mon)\times q(\Mon)$.
Since $p$ and $q$ are trivial types, we have $\Dp_p(x)=\{x\}$ for $x\models p$ and $\Dp_q(y)=\{y\}$ for $y\models q$. In this case Proposition \ref{prop_shuffled case general} guarantees that $S_{\mathbf p,\mathbf q}$ is a shuffling relation between $(p(\Mon),<_p)$ and $(q(\Mon),<_q)$; in particular, they are dense linear orders and $(\sigma(\Mon,y)\mid y\in q(\Mon))$ is a strictly increasing sequence of initial parts of $D_p$.

(a) Assume $|q(M)|\geqslant 2$. If $b<_q b'$ and  $b,b'\in q(M)$, then  by convexity of $\mathbf q$ we have $b<_q y<_q b'\vdash q(y)$, so $(q(M),<_q)$ is  a dense linear order. Now $M\prec \Mon$ implies that  $(\sigma(M,y)\mid y\in q(M))$ is a strictly increasing sequence of initial parts of $D_p^M=D_p\cap M$. Since $p(\Mon)$ is  $<_p$-convex in  $D_p$, for $b<_qb'$ and $b,b'\models q$: \begin{center}
$\emptyset\neq \sigma(\Mon,b')\smallsetminus \sigma(\Mon,b)=\Lp_{\mathbf p}(b')\smallsetminus \Lp_{\mathbf p}(b)\subseteq p(\Mon)$, 
\end{center}
so $(\sigma(M,y)\cap p(M)\mid y\in q(M))$ is a strictly increasing sequence of initial parts of $p(M)$; in particular, $|p(M)|\geqslant 2$. By symmetry, $(p(M),<_p)$ is an infinite dense linear order and $(\sigma(x,M)\cap q(M)\mid x\in p(M))$ is a strictly decreasing sequence of final parts of $q(M)$.  
It remains to show that for any $b\in q(M)$  the set $\sigma(M,b)\cap p(M)$ has no supremum in $(p(M),<_p)$. Suppose, on the contrary that $a\in p(M)$ and $a=\sup_{p(M)} (\sigma(M,b)\cap p(M))$.  
If $\sigma(M,b)\cap p(M)\neq \emptyset$, then   $a=\sup_{D_p^M}(\sigma(M,b))$, so $a\in \dcl(Ab)\smallsetminus \dcl(A)$ implying  $a\dep_Ab$ and $p\nfor q$; a contradiction. 
The remaining case is when $\sigma(M,b)\cap p(M)=\emptyset$. In this case  $\sigma(M,b)=\{x\in D_p^M\mid x<_p p(M)\}$, so $a=\min (p(M))$ (recall that we have defined $\sup(\emptyset)$ to be the minimal element of the order if such an element exists).
Having that $(\sigma(M,b')\cap p(M)\mid b'\in q(M))$ is a strictly increasing sequence of initial parts of $p(M)$, we further conclude that $b=\min(q(M))$, and for every $b'\in q(M)$ distinct from $b$ we have $a\in\sigma(M,b')$. Thus $\sigma(a,M)$  is a final part of $D_q^M$  containing  all the elements of $q(M)$ except $b$; by convexity of $\mathbf q$, $b=\max (D_q^M\smallsetminus\sigma(a,M))$ follows. Then $b=\max (D_q\smallsetminus\sigma(a,\Mon))$, hence $b\dep_A a$ and $p\nfor q$; a contradiction.

(b) Suppose that $a\in p(M)$ and $b\in q(M)$. Since $\sigma(a,\Mon)$ is a final part of $D_q$ relatively defining $\Rp_{\mathbf q}(a)$, we have that $\sigma(a,\Mon)$ does not meet $q(\Mon)^-=\{y\in D_q\mid y<_qq(\Mon)\}$ and $D_q\smallsetminus\sigma(a,\Mon)$ does not meet $q(\Mon)^+=\{y\in D_q\mid q(\Mon)<_qy\}$. So if $|q(M)|=1$, since $\mathbf q$ is convex, we have either $b=\max(D_q^M\smallsetminus \sigma(a,M))$ or $b=\min(\sigma(a,M))$, depending on whether $\models\lnot\sigma(a,b)$ or $\models\sigma(a,b)$. By elementarity, $b=\max(D_q\smallsetminus \sigma(a,\Mon))$ or $b=\min(\sigma(a,\Mon))$; in either case $b\dep_Aa$ and $p\nfor q$. Since this is not possible, we conclude $|q(M)|\geq 2$, and the desired conclusion follows by part (a) of the lemma.
\end{proof}

\begin{dfn} A stationarily ordered type $p\in S_1(A)$ is {\em simple} if $\D_p=\{(a,b)\mid a,b\models p \mbox{ and } a\dep_A b\}$ is a relatively $A$-definable relation on $p(\Mon)$.
\end{dfn}

We will show in Lemma \ref{Lema_equiv_convsimpl_reg} that in the context of binary theories the notion of simplicity   is consistent with the notion of simplicity for  asymmetric regular types introduced in  \cite{MT}; however, it significantly differs from the one used in the \oo-minimal context by Mayer in \cite{May} and  by Rast and Sahota in \cite{Rast2}.

\begin{rmk}\label{Rmk simple type}
Suppose that  $p\in S_1(A)$ is a simple, stationarily ordered type witnessed by  $(D_p,<_p)$.

(a)  $\Dp_p$ is a  relatively definable convex equivalence on $p(\Mon)$ so, by  compactness,  there is an $A$-definable,  convex equivalence relation $E_p$ on  $D_p$ agreeing with $\Dp_p$ on $p(\Mon)$. Note that  $E_p(x,\Mon)<_p y$ is first-order expressible on $D_p$ and that  any formula defining it relatively defines $\D_p(x)<_p y$ on $p(\Mon)\times p(\Mon)$.

(b) If in addition $p(\Mon)$ is a $<_p$-convex subset  of $D_p$, for $a\models p$ we have $E_p(a,\Mon)=\D_p(a)$, so $\Dp_p(a)$ is an $Aa$-definable set for $a\models p$; hence  $\D_p(a)<_p x$ and $x<\D_p(a)$ are first-order expressible in $(D_p,<_p)$. 
\end{rmk}

\begin{lem}\label{Lema ainvar of simple}
Suppose that $\mathbf p=(p,<_p)$ is a convex \so-pair over $A$ witnessed by $(D_p,<_p)$ and that $p$ is simple.
Then for any model $M\supseteq A$, either $\Inv_{\mathbf p}(M)$ is   a dense order-type or $ \Inv_{\mathbf p}(M)\in\{\mathbf{0},\mathbf{1}\}$. 
\end{lem}
\begin{proof}
Suppose that $a,b\in p(M)$ are independent over $A$, say $\D_p(a)<_p\D_p(b)$; then $(a,b)$ is a Morley sequence in $\p_r$ over $A$, so by saturation of $\Mon$ there is $c\in \Mon$ such that $(a,c,b)$ is a Morley sequence in $\p_r$, too. By Remark \ref{Rmk simple type} $\D_p(a)<_px<_p\D_p(b)$ is first-order expressible and  $c$ witnesses the existential quantifier in $\models \exists x\,(\D_p(a)<_px<_p\D_p(b))$. Hence for some $c'\in M$ we have $\D_p(a)<_pc'<_p\D_p(b)$. Then $\D_p(a)<_p\D_p(c')<_p\D_p(b)$  and $\Inv_{\mathbf p}(M)$ is a dense order-type. 
\end{proof}

\begin{lem}\label{Lema convex implies simple} Suppose that $p\in S_1(A)$ is a convex, stationarily ordered type. If there exists a convex, stationarily ordered type $q\in S_1(A)$ such that $\mathcal S(p,q)$, then $p$ (and by symmetry also $q$) is  simple.  
\end{lem}
\begin{proof}
Suppose that $(D_p,<_p)$ and $(D_q,<_q)$ are $A$-definable linear orders witnessing that $\mathbf p= (p,<_p)$ and $\mathbf q=(q,<_q)$ are convex, directly shuffled \so-pairs over $A$. Let $\sigma(x,y)$ be a formula defining a monotone relation between $(D_p,<_p)$ and $(D_q,<_q)$ and relatively defining $S_{\mathbf{p},\mathbf{q}}$; it exists by Lemma \ref{lem_shuff rel def}. It is easy to see that for $b,b'\models q$ the set $\Lp_{\mathbf p}(b)\smallsetminus\Lp_{\mathbf p}(b')$ is  defined by $\sigma(x,b)\wedge\neg\sigma(x,b')$.  In particular,  $\Lp_{\mathbf p}(y)=\Lp_{\mathbf p}(y')$ is a relatively definable relation within $q(\Mon)\times q(\Mon)$.  For $b,b'\models q$ by Lemma \ref{Prop_equivalents_of_a dep b} we have:   $b\dep_Ab'$ if and only if $\Lp_{\mathbf p}(b)=\Lp_{\mathbf p}(b')$. Hence $y\dep_A y'$ is a relatively definable relation on $q(\Mon)$ and  $q$ is a simple type. By symmetry,  $p$ is simple, too.
\end{proof}

\begin{prop}\label{prop_shuffled case local} Suppose that $\mathbf p= (p,<_p)$ and $\mathbf q=(q,<_q)$ are directly shuffled, convex \so-pairs over $A$, $M\supseteq A$ is a model and $\hat S_{\mathbf p,\mathbf q}^M=\{(\Dp_p^M(a),\Dp_q^M(b))\mid a\in p(M), \ b \in q(M), (a,b)\in S_{\mathbf p,\mathbf q}\}$.
\begin{enumerate}[(a)]
\item If $|q(M)/\Dp_q^M|\geq 2$, then $\hat S_{\mathbf p,\mathbf q}^M$ shuffles the (infinite) orders $(p(M)/\Dp_p^M,<_p)$ and $(q(M)/\Dp_q^M,<_q)$.
\item If $|q(M)/\Dp_q^M|= 1$, then $p$ is omitted in $M$.
\end{enumerate}
\end{prop}
\begin{proof} Without loss of generality let $A=\emptyset$. Suppose that the orders $(D_p,<_p)$ and $(D_q,<_q)$ witness the convexity of the \so-pairs $\mathbf p$ and $\mathbf q$. For the rest of the proof we will operate in $\Mon^{eq}$ (in fact we will need to add only two new sorts to $\Mon$); note that passing to $T^{eq}$ does not affect the assumptions:  $\mathbf p$ and $\mathbf q$ are convex, directly shuffled \so-pairs. 
By Lemma \ref{Lema convex implies simple} $p$ and $q$ are simple types so, by Remark \ref{Rmk simple type}, there is a  definable convex  equivalence relation $E_p$ on $D_p$   agreeing with
$\Dp_p$ on $p(\Mon)$; then $(D_p/E_p,<_p)$ is a definable linear order (in $\Mon^{eq}$) witnessing that $\mathbf p_E=(p_E,<_p)$ is a convex \so-pair, where $p_E=\tp([a]_{E_p})$ for $a\models p$.  
Similarly, there is a definable convex  equivalence relation $E_q$ on $D_q$   agreeing with
$\Dp_q$ on $q(\Mon)$ such that $(D_q/E_q,<_q)$ witnesses  that $\mathbf q_E=(q_E,<_q)$ is a convex \so-pair. The \so-pairs $\mathbf p_E$ and $\mathbf q_E$ are directly shuffled: $p\nfor p_E$, $q\nfor q_E$ and $\mathcal S(\mathbf p,\mathbf q)$  by transitivity of $\nwor$ and $\nfor$ imply $\mathcal S(\mathbf p_E,\mathbf q_E)$, while $\delta(\mathbf p, \mathbf q)$, $\delta(\mathbf p,\mathbf p_E)$ and  $\delta(\mathbf q,\mathbf q_E)$ imply $\delta(\mathbf p_E,\mathbf q_E)$ by transitivity of $\delta$. Note that $p_E$ and $q_E$ are trivial types and $\hat S_{\mathbf p,\mathbf q}=  S_{\mathbf p_E,\mathbf q_E}$. The rest of the proof is an application of  Lemma \ref{lem_shuff rel def1}.

(a) If $|q(M)/\Dp_q^M|\geqslant 2$, then $M$ contains two independent realizations of $q$, so $|q_E(M)|\geqslant 2$. Hence $\mathbf p_E$ and $\mathbf q_E$ satisfy all the assumptions of Lemma \ref{lem_shuff rel def1}(a), so $(p_E(M),<_p)$ and  $(q_E(M),<_q)$ are infinite linear orders shuffled by $S_{\mathbf p_E,\mathbf q_E}^M$. Now $\hat S_{\mathbf p,\mathbf q}=  S_{\mathbf p_E,\mathbf q_E}$ implies the desired conclusion.

(b) If $|q(M)/\Dp_q^M|=1$ and $p$ is realized in $M$, then both $p_E$ and $q_E$ are realized in $M$. Hence $p_E$ and $q_E$ satisfy all the assumptions of Lemma \ref{lem_shuff rel def1}(b), so  $(q_E(M),<_q)$ is an infinite order; a contradiction. Therefore, $M$ omits $p$.
\end{proof} 

\begin{dfn} Let $p\in S_1(A)$ be a stationarily ordered type and $M\supseteq A$ a model. 
 A {\it $\Dp_p^M$-transversal} is any set of representatives of $\Dp_p^M$-classes, i.e.\  a   maximal, pairwise independent subset of $p(M)$.    
\end{dfn} 

For $I_p(M)\subseteq M$  a $\Dp_p^M$-transversal, the mapping $x\longmapsto \Dp_p(x)$   is a  natural isomorphism between $(I_p(M),<_p)$ and $(p(M)/\Dp_p^M,<_p)$; in particular, the order type of $(I_p(M),<_p)$   does not depend on the particular choice of the transversal. Using this observation
we can reformulate the previous proposition.

\begin{cor}\label{cor shuffled representatives} Suppose that $\mathbf p= (p,<_p)$ and $\mathbf q=(q,<_q)$ are directly shuffled, convex  \so-pairs over $A$, $A\subseteq M$,  $I_p(M)$ is a   $\Dp_p^M$-transversal and $I_q(M)$ is a  $\Dp_q^M$-transversal. Then:
\begin{enumerate}[(a)]
\item If $|I_q(M)|\geq 2$, then the restriction of $S_{\mathbf p,\mathbf q}$ to $I_p(M)\times I_q(M)$ is a shuffling relation between (infinite) orders $(I_p(M),<_p)$ and $(I_q(M),<_q)$;
\item If $|I_q(M)|=1$, then $p$ is omitted in $M$.\qed
\end{enumerate} 
\end{cor}

\begin{exm} The conclusions   of the previous corollary may fail if one of the types  is  not convex.\\ 
Take the dense linear order without endpoints coloured by $\omega$ dense colors $\{D_n\mid n\in\omega\}$ from Example \ref{exm dense colors}; note that forking is the equality relation. Let $p_0\in S_1(\emptyset)$ be the type of an element of  color $D_0$  and let $q\in S_1(\emptyset)$  be the type of a colorless element; $p_0$ is convex, while $q$ is not convex.   Then $(p_0,<)$ and $(q,<)$ are directly shuffled \so-pairs; $(p_0(\Mon),<)$ and $(q(\Mon),<)$ are shuffled by $<$. On the other hand, if $M$ is the prime model,  then $M$ omits $q$ while $(p_0(M),<)$ is an infinite  dense order. 
\end{exm}

\begin{prop}\label{Prop commutativity in M}
Suppose that   $(\mathbf{p}_n=(p_n,<_n)\mid n\in\alpha)$ is a sequence of pairwise directly shuffled, convex \so-pairs over $A$. Let $M\supseteq A$ be a model, let $I_n(M)$ be a $\Dp_{p_n}^M$-transversal and let $S_{n,m}^M=S_{\mathbf{p}_n,\mathbf{p}_m}\cap (I_n(M)\times I_m(M))$ ($n<m<\alpha$). Then either \ 

(1) \ \   $|\bigcup_{n\in\alpha}I_n(M)|\leqslant 1$, \ or  \ 

(2) \ \   $((I_n(M),<_n)\mid n<\alpha)$ are dense linear orders shuffled by $(S_{n,m}^M\mid n<m<\alpha)$. 
\end{prop}
\begin{proof}Suppose that $|\bigcup_{n\in\alpha}I_n(M)|\geqslant 2$ and let $n_0<\alpha$ be such that $I_{n_0}(M)\neq \emptyset$. Then $|I_{n_0}(M)|=1$ is impossible: otherwise, by Corollary \ref{cor shuffled representatives}(b) we would have each $p_m$ for $m\neq n_0$ omitted in $M$ and thus $|\bigcup_{n\in\alpha}I_n(M)|=1$. Hence $|I_{n_0}(M)|\geqslant 2$ and by  applying Corollary \ref{cor shuffled representatives}(a) we deduce: the orders $(I_n(M),<_n)$ and $(I_m(M),<_m)$ are shuffled by $S_{n,m}^M$ for all $n<m<\alpha$. 

To complete the proof of (2),  it remains to verify that the sequence of  relations $S_{n,m}^M$ is coherent. 
By Lemma \ref{lem_shuff rel def} for  all $n<m<\alpha$ there is a formula $\sigma_{n,m}(x,y)$ defining a monotone relation between the domains of $<_n$ and $<_m$ and relatively defining $S_{\mathbf p_n,\mathbf p_m}$ within $p_n(\Mon)\times p_m(\Mon)$. 
Fix $n<m<k<\alpha$, let $a_n\in I_n(M)$ and $a_k\in I_k(M)$. By Lemma \ref{Lema_delta_implies_S+pq_commute} the sequence of   $(S_{\mathbf{p}_n,\mathbf{p}_m}\mid n<m<\alpha)$ is coherent, so:
\begin{center}
 $(a_n,a_k)\in S_{n,k}^M$ \ if and only if \ $\models \exists y (\sigma_{n,m}(a_n,y)\land \sigma_{m,k}(y,a_k)\land \bigwedge p_m(y))$. 
\end{center}
Since $p_m$ is convex, $\sigma_{n,m}(a_n,\Mon)$ is a   $\mathbf{p}_m$-left bounded final part and  $\sigma_{m,k}(\Mon,a_k)$ is a   $\mathbf{p}_m$-right bounded initial part of the domain of $<_m$, we have  $\sigma_{n,m}(a_n,y)\land \sigma_{m,k}(y,a_k)\vdash p_m(y)$. Hence:
\begin{center}
 $(a_n,a_k)\in S_{n,k}^M$ \ if and only if \ $\models \exists y (\sigma_{n,m}(a_n,y)\land \sigma_{m,k}(y,a_k))$. 
\end{center} 
The right hand side of this equivalence holds in $\Mon$ if and only if it holds in $M$. Hence $S^M_{n,k}=S^M_{m,k}\circ S^M_{n,m}$\,,  completing the proof of the proposition.
\end{proof}

 \section{Stationarily ordered types in binary theories}\label{s so in binary general}

In this section we prove several  technical results related to stationarily ordered types in binary theories. 
\begin{rmk}
An equivalent way of stating that a theory $T$ is  binary is: the type of any tuple of elements is forced by the types of pairs of its elements:
\begin{center}
$\bigcup_{1\leqslant i<j\leqslant n}\,\tp_{x_i,x_j}(a_i,a_j)\forces_T \tp_{x_1,\ldots,x_n}(a_1,\ldots,a_n)$
\end{center}
holds for all elements $a_1,\ldots,a_n$.  
Yet another way of expressing that is: 
\begin{center}$\bigcup_{b\in B}\,\tp_x(a/b)\vdash_T \tp_x(a/B)$ holds for all $a,B$.
\end{center}
These characterizations are consequences of compactness and will be freely used below. 
\end{rmk}

\begin{rmk}\label{rmk_morely seq in binary} Suppose that $T$ is binary and that $\mathbf p=(p,<_p)$ is an \so-pair over $A$. 

(a) \ Morley sequences in $\p_l$ ($\p_r$) over $A$ have a simple description:
  they  are   decreasing (increasing) sequences of $\Dp_p$-representatives.
Indeed, if $(I,<)$ is a linear order, then by binarity: \begin{center}
$(a_i\mid i\in I)$ is a Morley sequence in $\p_r$ over $A$ \ if and only if  \  each $(a_i,a_j)$   for $i<j$ is so.
\end{center} By Lemma \ref{lem rmk closed for Dp}(c) the latter is equivalent to $\Dp_p(a_i)<_p\Dp_p(a_j)$. In other words, the sequence $(\Dp_p(a_i)\mid i\in I)$ is  $<_p$-increasing.

(b) In fact, the type of a pairwise independent tuple $\bar a=(a_0,\ldots,a_n)$ of realizations of $p$ is determined by its $\{<_p\}$-type: if $\bar b=(b_0,\ldots,b_n)$ is another pairwise independent tuple of realizations of $p$ that has the same $\{<_p\}$-type as $\bar a$ does, then $\tp(\bar a)=\tp (\bar b)$. 
\end{rmk}

\begin{prop}\label{Prop decomposition}
Suppose that $T$ is a small, binary, stationarily ordered theory. Let  $(\alpha_i\mid i<\kappa)$ be an enumeration of all the $\nwor$-classes of $S_1(\emptyset)$ and for each $X\subseteq \Mon$ let $\alpha_i(X)=\{a\in X\mid \tp(a)\in\alpha_i\}$. 

(a) $(\alpha_i(\Mon)\mid i<\kappa)$ is an orthogonal decomposition of $\Mon$: \ $\tp(\alpha_i(\Mon))\wor \tp(\Mon\smallsetminus \alpha_i(\Mon))$  ($i<\kappa$).

(b) If $ M_i\prec \Mon$ for all $i<\kappa$, then $N=\bigcup_{i<\kappa}\alpha_i(M_i)\models T$.
\end{prop}
\begin{proof}
(a) Suppose that $\bar a,\bar a'\in \alpha_i(\Mon)$,  $\bar b\in\Mon\smallsetminus \alpha_i(\Mon)$ and $\tp(\bar a)=\tp(\bar a')$. For each $a_j\in \bar a$ and $b_i\in\bar b$ we have 
  $\tp(a_j)\wor \tp(b_i)$, so $\tp(a_j,b_i)=\tp(a_j',b_i)$ and, by binarity of $T$, we have $\tp(\bar a,\bar b)=\tp(\bar a',\bar b)$. The conclusion follows.

 (b) We will prove that every consistent formula $\phi(x,\bar c)$   with parameters from $N$ is satisfied by an element of $N$; fix such a formula $\phi(x,\bar c)$. Since $T$ is small isolated types are dense in $S_1(\bar c)$, so there is $d\in \phi(\Mon,\bar c)$   such that $\tp(d/\bar c)$ is isolated. Write 
$\bar c=\bar c_0,...,\bar c_n$ where  $\bar c_k\in\alpha_{i_k}(M_{i_k})$ and the $i_k$  are pairwise distinct. We have two cases to consider.

Case 1.  $d\in \alpha_{i_k}(\Mon)$ for some $k\leqslant n$. In this case, by part (a) we have $\tp(d/\bar c_k)\vdash \tp(d/\bar c)$, so $\tp(d/\bar c_k)$ is isolated and forces $\phi(x,\bar c)$. Choose $d'\in\alpha_{i_k}(M_{i_k})$ realizing $\tp(d/\bar c_k)$; clearly, $d'\in N$   realizes $\phi(x,\bar c)$.

Case 2. $d\notin \bigcup_{k\leqslant n}\alpha_{i_k}(\Mon)$. In this case, by part (a),  we have $\tp(d)\vdash\tp(d/\bar c)$, so $\tp(d)$ is isolated. If $\tp(d)\in\alpha_j$ and $d''\in M_j$ realizes $\tp(d)$, then $d''\in \alpha_j(M_j)$, so $d''\in N$ is a realization of $\phi(x,\bar c)$. 

In either of cases we  found a realization of $\phi(x,\bar c)$ in $N$; hence $N\prec \Mon$.
\end{proof}

Now we recall the notion of  regularity for global invariant types from \cite{PT} and \cite{MT}: a global non-algebraic type $\p$ is {\em weakly regular over $A$} if it is $A$-invariant and: 
\begin{center}
for all $X\subset\p_{\strok A}(\Mon)$ and $a\models\p_{\strok A}$: either $a\models\p_{\strok AX}$ or $\p_{\strok AX}(x)\forces\p_{\strok AXa}(x)$ holds.
\end{center} 
Alternatively, $\p$ is weakly regular over $A$ if it is $A$-invariant and the operator $\cl_\p^A$ defined by:
\begin{center}
 $\cl_\p^A(X)=  \p_{\strok A}(\Mon)\smallsetminus\p_{\strok AX}(\Mon)$ \  for \  $X\subset\p_{\strok A}(\Mon)$,
 \end{center} is a closure operator (satisfies  monotonicity, finite character and idempotency) on $\p_{\strok A}(\Mon)$. 

An $A$-invariant global type $\p$ is {\em $A$-asymmetric} if $\tp(a_0,a_1/A)\neq \tp(a_1,a_0/A)$ for $(a_0,a_1)$ realizing a Morley sequence in $\p$ over $A$. 

\begin{fact}\label{Thm_wregular}[\cite{MT}, Theorem 2.4]
Suppose that $\p$ is weakly regular over $A$ and $A$-asymmetric. 
\begin{enumerate}[(a)]
\item There is an $A$-definable partial order $\leqslant$  on $\Mon$ such that every Morley sequence in $\p$ over $A$ is strictly increasing; in this case we say that $\leqslant$ witnesses the $A$-asymmetry of $\p$.

\item For any model $M\supseteq A$ the order type of any maximal Morley\footnote{Recall that we allow  Morley sequences   indexed by any linear order.} sequence (in $\p$ over $A$) consisting of elements of $M$ does not depend on the particular choice of the sequence; this order type is denoted (see 2.5 in \cite{MT})   by $\Inv_{\p,A}(M)$. 
\end{enumerate}
\end{fact} 

\begin{prop}\label{prop_regular} Suppose that $T$ is binary and that $\mathbf p=(p,<_p)$ is an \so-pair over $A$. Then:

(a)  $\p_r$ and $\p_l$ are weakly regular over $A$  and  $A$-asymmetric.

(b)  $\Inv_{\p_r,A}(M)=\Inv_{\mathbf p}(M)$ for any model $M\supseteq A$.
\end{prop}
\begin{proof} 
(a) Clearly, $\p_r$ is $A$-asymmetric. For 
 $X\subseteq p(\Mon)$   and $a\models p$ we have:   $a\models\p_{r\strok AX}$ iff  $a\in\Rp_{\mathbf p}(X)$.
The regularity condition translates to: \
  either   $a\in \Rp_{\mathbf{p}}(X)$  or  $\Rp_{\mathbf{p}}(X)= \Rp_{\mathbf{p}}(Xa)$.
\ 

In order to prove it,  we will (have to) use binarity. Assuming $a\notin\Rp_{\mathbf{p}}(X)$
it suffices to show that 
 for all  $b\in\Rp_{\mathbf{p}}(X)$ and $c\in\Rp_{\mathbf p}(Xa)$ we have $b\equiv c\,(AXa)$.
For such $b,c$ we   have $b\equiv c\,(AX)$  so, by binarity, it remains to prove $b\equiv c\,(a)$. 
From $a\notin \Rp_{\mathbf{p}}(X)$ and $b\in\Rp_{\mathbf{p}}(X)$, keeping in mind that $\Rp_{\mathbf{p}}(X)$ is a $\Dp_{p}$-closed final part of $p(\Mon)$, we deduce   $\Dp_p(a)<b$ and hence $b\in \Rp_{\mathbf p}(a)$. 
On the other hand, $c\in\Rp_{\mathbf p}(Xa)$ directly implies $c\in\Rp_{\mathbf p}(a)$. Therefore, $b\equiv c\,(a)$.

(b) Note that the extension $\{(x,y)\in\Mon\times\Mon\mid x=y\mbox{ or }x<_py\}$ of $<_p$ witnesses the $A$-asymmetry of $\p_r$, so the desired conclusion follows from the description of Morley sequences in $\p_r$  from Remark \ref{rmk_morely seq in binary}(a). 
\end{proof}

Recall from \cite{MT} the notions of simplicity and convexity for weakly regular, $A$-asymmetric types $\p$:

-- $\p$ is {\em convex over $A$} (see 2.10 in \cite{MT}) if there is an $A$-definable {\em partial} ordering $\leqslant$ witnessing the $A$-asymmetry of $\p$ such that $\p_{\strok A}(\Mon)$ is a $\leqslant$-convex subset of $\Mon$ ($a,a'\models \p_{\strok A}$ and $a\leqslant c\leqslant a'$ imply $c\models \p_{\strok A}$).

-- $\p$ is {\em simple over $A$} if  $\cl^A_{\p}(x)=\cl^A_{\p}(y)$ is a relatively definable relation on $\p_{\strok A}(\Mon)$ (4.1 and 2.5 in \cite{MT}).

\begin{lem}\label{Lema_equiv_convsimpl_reg}
Let $\mathbf p=(p,<_p)$ be an \so-pair over $A$ and assume that $\p_r$ is weakly regular over $A$. 

(a) The type $\p_r$ is simple over $A$ if and only if $p$ is a simple type.

(b) The type $\p_r$ is convex over $A$ if and only if $p$ is a convex type.
\end{lem}
\begin{proof}
(a) For $a,b\models p$  we have $\cl^A_{\p_r}(a)=p(\Mon)\smallsetminus \Rp_{\mathbf p}(a)=\Lp_{\mathbf p}(a)\cup\Dp_p(a)$. Hence 
$\cl^A_{\p_r}(a)=\cl^A_{\p_r}(b)$   if and only if  $\Dp_p(a)=\Dp_p(b)$, so   $\p_r$ is simple over $A$ if and only if $p$ is a simple  type.

(b) Fix $(D_p,<_p)$ witnessing that $\mathbf p$ is an \so-pair. 
First suppose that $\p_r$ is convex over $A$ witnessed by an $A$-definable partial order $\leqslant$ on $\Mon$: $p(\Mon)$ is a $\leqslant$-convex subset of $\Mon$ and Morley sequences in $\p_r$ over $A$ increase. Then the assumptions of Lemma \ref{Lema_convex_so_pair_witness} are satisfied, so $p$ is a convex type. This proves one direction of the equivalence. For the other, it suffices to note that if  $(D_p,<_p)$ witnesses the convexity of $p$, then   the partial order $\{(x,y)\in \Mon\times\Mon\mid x=y \mbox{ or } x<_p y\}$ witnesses the $A$-asymmetry and convexity of $\p_r$ over $A$.  
\end{proof}

\begin{fact}\label{Fact_MTcor46}[\cite{MT}, Corollary 4.6]
Suppose that $T$ is countable, $A$ is finite and $\p$ is weakly regular over $A$ and $A$-asymmetric. If $I(\aleph_0,T)<2^{\aleph_0}$, then $\p$ is both  simple and convex over $A$. 
\end{fact}  

\begin{prop}\label{Prop weakly_reg_convexsimple}
If $T$ is a complete, countable, binary theory and $I(\aleph_0,T)<2^{\aleph_0}$, then every stationarily ordered type over a finite domain is convex and simple. 
\end{prop}
\begin{proof} Let $A$ be finite and let $\mathbf{p}=(p,<_p)$ be an \so-pair over $A$. By Proposition \ref{prop_regular} the type $\p_r$ is weakly regular over $A$ and $A$-asymmetric. By Fact \ref{Fact_MTcor46} $\p_r$ is simple  and  convex  over $A$. By Lemma \ref{Lema_equiv_convsimpl_reg}   $p$ is simple and convex.
\end{proof}

\begin{lem}\label{Lema_symmetry of isolation}(Symmetry of isolation)
Suppose that $p,q\in S_1(A)$ are stationarily ordered, convex types, $a\models p$,  $b\models q$ and $a\dep_A b$. 
Then  $\tp(a/Ab)$ is isolated if and only if $\tp(b/Aa)$ is isolated.
\end{lem}
\begin{proof}
Choose orders  witnessing that the types are stationarily ordered and convex. Assume that $\tp(a/Ab)$ is isolated by $\phi(x,b)$ and $a\dep_A b$.  By Lemma \ref{Lema a dep b witness}  there is a formula $\psi(x,y)\in\tp(a,b/A)$ such that   $\psi(a,y)\forces q(y)$. Let $\theta(x,y):= \phi(x,y)\wedge\psi(x,y)$. Clearly, $\theta(a,y)\in\tp(b/Aa)$,  $\theta(x,b)\forces\tp(a/Ab)$ and $\theta(a,y)\forces q(y)$. We show that $\theta(a,y)$ isolates $\tp(b/Aa)$. Let $b'$ be such that $\models\theta(a,b')$. Then $\models \psi(a,b')$ implies $b\equiv b'\,(A)$, so there exists $a'$ such that $ab\equiv a'b'\,(A)$. Now, $\theta(x,b)\forces\tp(a/Ab)$ implies $\theta(x,b')\forces\tp(a'/Ab')$, so $\models\theta(a,b')$ implies $a\equiv a'\,(Ab')$. Thus $ab'\equiv a'b'\equiv ab\,(A)$, hence $b'\equiv b\,(Aa)$ and $\theta(a,y)$ isolates  $\tp(b/Aa)$. 
\end{proof}

\begin{lem}\label{Lema_isolated_wor_nonisolated} 
If $q\in S_1(A)$ is a stationarily ordered   convex type,  $p\in S_n(A)$ is  isolated  and  $p\nwor q$, then $q$ is  isolated, too. In particular, (non-)isolation of stationarily ordered, convex types is preserved under $\nwor$.
\end{lem}
\begin{proof}Suppose that $p$ is isolated by $\psi(\bar x)$. 
Fix $\bar a\models p$ and choose an order  $(D_q,<_q)$ witnessing that $\mathbf q=(q,<_q)$ is a convex \so-pair. The types $\q_{l\strok A\bar a}$ and $\q_{r\strok A\bar a}$ are distinct because $p\nwor q$, so by Remark \ref{Rmk notin p_l} there is a strongly $\mathbf{q}$-left-bounded formula $\phi_l(\bar a,y)\in \q_{r\strok A\bar a}\smallsetminus \q_{l\strok A\bar a}$. Similarly, there is a strongly $\mathbf{q}$-right-bounded formula $\phi_r(\bar a,y)\in \q_{l\strok A\bar a}\smallsetminus \q_{r\strok A\bar a}$.
 We {claim} that the formula
$$\theta(y):= \exists \bar x' \,(\psi(\bar x')\land\phi_l(\bar x',y))\land \exists \bar x'' \,(\psi(\bar x'')\land\phi_r(\bar x'',y))$$
isolates $q$. Clearly, $\theta(y)\in q(y)$. 
Assume that $\models\theta(b)$, and let $\bar a'$ and $\bar a''$ witness the existential quantifiers. Then $\models\psi(\bar a')\wedge\psi(\bar a'')$ implies $\bar a',\bar a''\models p$, so $\phi_l(\bar a',\Mon)$ is strongly $\mathbf{q}$-left-bounded,  by $c_l\models q$ say. Similarly,  $\phi_r(\bar a'',\Mon)$ is strongly $\mathbf{q}$-right-bounded,   by $c_r\models q$. Now $\models\phi_l(\bar a',b)\wedge\phi_r(\bar a'',b)$ implies $c_l<_qb<_qc_r$ and, by convexity of $q$, we get $b\models q$.  We have just shown that every realization of $\theta(y)$ realizes $q$; $q$ is  isolated. 
\end{proof}

\begin{cor}\label{Cor_isolated forces over non-isolated}($T$ binary)
If $\tp(\bar a)$ is isolated and  $\tp(b)$  is stationarily ordered, convex and non-isolated  for all $b\in B$, then $\tp(\bar a)\vdash \tp(\bar a/B)$ and  $\tp(\bar a/B)$ is isolated.
\end{cor}
\begin{proof}
It suffices to show that $\bar a'\equiv \bar a$ implies $\bar a'B\equiv\bar aB$, i.e.\ that any pair of elements $c,d\in \bar aB$ has the same type as the corresponding pair $c',d'\in \bar a'B$. Clearly, this   holds if $c,d\in B$ or  $c,d\in \bar a$.  If $c\in\bar a$ and $d\in B$, then by Lemma \ref{Lema_isolated_wor_nonisolated} we have $\tp(c)\wor \tp(d)$  and hence $c\equiv c'\,(d)$. 
\end{proof}

The next proposition is probably well known, but since we couldn't find a reference, a short proof is included.

\begin{prop}\label{Prop infmany nwor implies continuum}
If $T$ is a countable, complete, binary theory having an infinite family of pairwise orthogonal, non-isolated  types in $S_1(\emptyset)$, then $I(\aleph_0,T)=2^{\aleph_0}$.  
\end{prop}
\begin{proof}
We may assume that $T$ is small, since otherwise $I(\aleph_0,T)=2^{\aleph_0}$ follows. Suppose that  the types $\{p_n\mid n\in\omega\}\subseteq S_1(\emptyset)$ are non-isolated and pairwise orthogonal. For each $J\subseteq \omega$ we will construct a countable model $M_J$ satisfying: $M_J$ realizes $p_n$ if and only if $n\in J$. Clearly, for distinct $J$'s the corresponding models are not isomorphic and $I(\aleph_0,T)=2^{\aleph_0}$.

Let $J\subseteq \omega$,  let $a_j$ realize $p_j$ and let $A_J=\{a_j\mid j\in J\}$. Since $T$ is binary, for each $n\notin J$ we have $p_n\wor \tp(A_J)$, so $p_n$ has a unique extension over $A_J$; since $p_n$ is non-isolated, the extension is non-isolated, as well. By the Omitting Types Theorem, there exists a countable model $M_J\supseteq A_J$ omitting each $p_n$ for $n\notin J$. Clearly, $M_J$ realizes $p_j$ for $j\in J$ and for distinct $J$  the corresponding models are non-isomorphic. 
\end{proof}

\subsection{Definability and invariants}

\begin{dfn} A convex \so-pair $\mathbf p=(p,<_p)$ is left-definable (right-definable) if   $\p_l$ ($\p_r$) is a definable type. $\mathbf p$ is  definable if it is left-definable or right-definable; otherwise, it is non-definable. 
\end{dfn}

Note that in the \oo-minimal context terms non-cut  (or rational cut) for $p\in S_1(\emptyset)$ when $\mathbf p$ is left(right)-definable, and (irrational) cut when $\mathbf p$ is non-definable  are widely used in the literature.

\begin{lem}\label{Lema left definability}
Suppose that $\mathbf p=(p,<_p)$ is a convex \so-pair over $A$ witnessed by $(D_p,<_p)$. The following conditions are equivalent:

(1) \ $\mathbf p$  is left-definable;

(2) \  $p$ is not finitely satisfiable in $p(\Mon)^-=\{x\in D_p \mid x<_pp(\Mon)\}$;

(3) \ There exists an $A$-definable $D\subseteq D_p$  such that $p(\Mon)$ is an initial part of $(D,<_p)$;

(4) \ There exists a  $\Mon$-definable $D\subseteq D_p$  such that $p(\Mon)$ is an initial part of $(D,<_p)$;

(5)  \ $\p_l$ is not finitely satisfiable in $p(\Mon)^-$.\\
Similarly for    $\mathbf p$  right-definable.
\end{lem}
\begin{proof}
(1)$\Rightarrow$(2) Suppose that $\p_l$ is definable;  being $A$-invariant it  is definable over $A$.
In particular, the set $D=\{a\in D_p\mid (x<_pa)\in\p_l\}$ is $A$-definable. If $b\models p$, then  $x<_pb$ relatively defines a left eventual subset of $(p(\Mon),<_p)$, so $(x<_pb)\in\p_l$ and $b\in D$. Thus $p(\Mon)\subseteq D$ and $(x\in D)\in p$. For   $c\in p(\Mon)^-$ the set  $\{x<_p c\}\cup p(x)$ is inconsistent, so $(x<_pc)\notin \p_l$ and $c\notin D$; hence $x\in D$ is not satisfied in $p(\Mon)^-$.
 
(2)$\Rightarrow$(3).  If  $\theta(x)\in p$ is not satisfied  in $p(\Mon)^-$, then   $p(\Mon)$ is an initial part of $(\theta(\Mon)\cap D_p,<_p)$.

(3)$\Rightarrow$(4) is trivial. 

(4)$\Rightarrow$(5). If $D\subseteq D_p$ satisfies (4), then the formula $x\in D$ defines a left eventual subset of $(p(\Mon),<_p)$ and hence belongs to $\p_l$. Clearly $x\in D$ is not satisfied in $p(\Mon)^-$.

(5)$\Rightarrow$(4). Suppose that $\phi(x)\in\p_l$ is not satisfied in $p(\Mon)^-$. Let $\psi(x):=x\in D_p\land \exists y(\phi(y)\land y\in D_p\land y\leqslant_p x)$. Then $\psi(x)\in \p_l$ defines a final part of $D_p$ and   $p(\Mon)$ is an initial part of $(\psi(\Mon),<_p)$. 

(4)$\Rightarrow$(1) Suppose that $D\subseteq D_p$ is $\Mon$-definable and   $p(\Mon)$  is an initial part of $(D,<_p)$. For any   $\psi(x;\bar y)$:   \begin{center}
for all $\bar c\in \Mon$: \  $\psi(x,\bar c)\in \p_l$ \  if and only if \   $\psi(\Mon,\bar c)$ contains an initial part of $(D,<_p)$. 
\end{center}Since $(D,<_p)$ is $\Mon$-definable,   the right-hand side of the equivalence is a  $\Mon$-definable property of $\bar c$, so $\p_l$ is definable; $\mathbf p$ is left-definable.
\end{proof}

\begin{rmk}\label{Rmk left def non-isolated}
Let $\mathbf p=(p,<_p)$ be a convex \so-pair over $A$ witnessed by $(D_p,<_p)$. By Lemma \ref{Lema_fs_vs_nonisolation} we have: $p$ is non-isolated if and only it is finitely satisfied in $D_p\smallsetminus p(\Mon)=p(\Mon)^-\cup p(\Mon)^+$. \ Combining with Lemma  \ref{Lema left definability} we have the following options:

--  $p$ is finitely satisfied in $p(\Mon)^-$ but not in $p(\Mon)^+$ \ (iff $p$ is non-isolated and  $\mathbf p$ is right-definable);

--   $p$ is finitely satisfied in $p(\Mon)^+$ but not in $p(\Mon)^-$  \ (iff $p$ is non-isolated and  $\mathbf p$ is left-definable);

--  $p$ is finitely satisfied in both $p(\Mon)^-$ and $p(\Mon)^+$ \ (iff   $\mathbf p$ is non-definable)

--  $p$ is not finitely satisfied in  $p(\Mon)^-\cup p(\Mon)^+$ \ (iff $p$ is  isolated,   $\mathbf p$ is both left- and right-definable).
\end{rmk}

\begin{lem}\label{Lema_nonisolated_frakp}
Suppose that $\mathbf p=(p,<_p)$ is a convex \so-pair over $A$ witnessed by $(D_p,<_p)$. Then:

(a) If $p$ is right-definable and non-isolated, then $\frak p_{l\strok B}$ is non-isolated for all $B\supseteq A$;

(b) If $p$ is left-definable and non-isolated, then $\frak p_{r\strok B}$ is non-isolated for all $B\supseteq A$;

(c) If $p$ is non-definable, then both $\frak p_{r\strok B}$ and $\frak p_{l\strok B}$ are non-isolated for all $B\supseteq A$.
\end{lem} 
\begin{proof}
(a) Suppose that $\mathbf p$ is right-definable and $p$ is non-isolated. By Remark \ref{Rmk left def non-isolated} $p$ is finitely satisfied in $p(\Mon)^-$, so the equivalence of conditions (2) and (5) in Lemma \ref{Lema left definability} implies that $\p_l$ is finitely satisfied in  $p(\Mon)^-$. In particular,  $\p_{l\strok B}$ is finitely satisfied in $p(\Mon)^-$. On the other hand, since $\p_{l\strok B}(\Mon)$ is an initial part of $p(\Mon)$, we have $\p_{l\strok B}(\Mon)^-=p(\Mon)^-$, so $\p_{l\strok B}$ is finitely satisfied in   $\p_{l\strok B}(\Mon)^-$; by Lemma \ref{Lema_fs_vs_nonisolation} $\p_{l\strok B}$ is non-isolated. This proves part (a); parts (b) and (c) are proved analogously.
\end{proof}  
 
\begin{lem}\label{Lema left definable nwor}
If $p,q\in S_1(\emptyset)$  are  non-isolated  and  the  convex \so-pairs $\mathbf p=(p,<_p)$ and $\mathbf q=(q,<_q)$  are directly non-orthogonal, then $\mathbf p$ and $\mathbf q$ are simultaneously (left-,\,right-)\, non-definable. 
\end{lem}
\begin{proof} Suppose that $\mathbf p$ is left-definable. Choose  $(D_p,<_p)$ and $(D_q,<_q)$ witnessing the convexity of $\mathbf p$ and $\mathbf q$; moreover, by   Lemma \ref{Lema left definability} we may assume that $p(\Mon)$ is an initial part of $(D_p,<_p)$;. 

\smallskip
Case 1.  $\mathcal S(p,q)$ holds. \ Let $\sigma(x,y)$ be an $L$-formula defining a monotone relation between  $(D_p,<_p)$ and $(D_q,<_q)$ and relatively defining $S_{\mathbf p,\mathbf q}$ (i.e.\ $x\in \Lp_{\mathbf p}(y)$); it exists by Lemma \ref{lem_shuff rel def}. Then $(\sigma(\Mon,b)\mid b\in D_q)$ is an increasing sequence of initial parts of $D_p$ and $\emptyset\neq \sigma(\Mon,b)\subset p(\Mon)$  for all $b\in q(\Mon)$. For each $a\in p(\Mon)$ and $b\in \Lp_ {\mathbf q}(a)$, by direct non-orthogonality and Lemma \ref{Lema_direct_non-orth},  we have $a\in\Rp_{\mathbf p}(b)$ and hence $a\notin \sigma(\Mon,b)$. The latter implies $a\notin\bigcap_{b\in q(\Mon)}\sigma(\Mon,b)$ and thus $\bigcap_{b\in q(\Mon)}\sigma(\Mon,b)=\emptyset$. By monotonicity: $\sigma(\Mon,y)=\emptyset$ if and only if $y<_q q(\Mon)$. Hence the formula  $y\in D_q\land\exists x\,\sigma(x,y)$ belongs to $q(y)$ and is not satisfied in $q(\Mon)^-$, so   condition (2) of Lemma \ref{Lema left definability} is satisfied and $\mathbf q$ is left-definable. 

\smallskip
Case 2.  $p\nfor q$.  \ Choose $a\models p$ and $b\models q$ be such that $a\dep b$. 
Since the types in question are convex, by Lemma \ref{Lema a dep b witness} there is a formula $\phi(x,y)\in\tp(a,b)$ witnessing the dependence such that $\phi(a, \Mon)\subseteq \Dp_{q}(a)$ and $\phi(\Mon,b)\subseteq \Dp_{p}(b)$.    
Put $\sigma'(x,y):=y\in D_q\land  y<_q \phi(x,\Mon)$. Clearly, $\sigma'(a,\Mon)<_q \Rp_{\mathbf q}(b)$   and, using $\delta(\mathbf p,\mathbf q)$ and the fact that $p(\Mon)$ is an initial part of $(D_p,<_p)$,  we get  $\bigcap\{\sigma'(x,\Mon)\mid x\in D_p\land x<_pa\}=q(\Mon)^-$. 
Hence $q(\Mon)^-$ is $a$-definable and, being $\Aut(\Mon)$-invariant, we get that $q(\Mon)^-$ is a definable, initial part of $D_q$.  Then the set  $D=D_q\smallsetminus q(\Mon)^-$ satisfies condition (3) of Lemma \ref{Lema left definability}, so    
$\mathbf q$ is left-definable. 
\end{proof}

\begin{lem}\label{Lema invariants prime model}
Suppose that $T$ is small,  $\mathbf p=(p,<_p)$ is a   convex \so-pair over $\emptyset$ witnessed by $(D_p,<_p)$, $p$ is simple and non-isolated,  and  $M$ is a countable model of $T$. 

(a) If $\mathbf p$ is left-definable, then $\Inv_{\mathbf p}(M)\in\{\mathbf{0},\bm{\eta},\bm{\eta}+\mathbf 1\}$. Moreover, if $M$ is prime over $a$, then $\Inv_{\mathbf p}(M)=\bm{\eta}+\mathbf 1$ and 
$\Inv_{\mathbf q}(M)= \bm{\eta}$ for any  convex \so-pair $\mathbf q$ directly shuffled with $\mathbf p$.

(b) If $\mathbf p$ is right-definable, then $\Inv_{\mathbf p}(M)\in\{\mathbf{0},\bm{\eta},\mathbf{1}+\bm{\eta}\}$. Moreover, if $M$ is prime over $a$, then $\Inv_{\mathbf p}(M)=\mathbf{1}+\bm{\eta}$ and 
$\Inv_{\mathbf q}(M)= \bm{\eta}$ for any  convex \so-pair $\mathbf q$ directly shuffled with $\mathbf p$.

(c) If $\mathbf p$ is non-definable and $M$ is prime over $a\models p$, then $\Inv_{\mathbf p}(M) =\mathbf{1} $  and 
$\Inv_{\mathbf q}(M)= \mathbf{0}$ for any  convex \so-pair $\mathbf q$ directly shuffled with $\mathbf p$.
\end{lem}
\begin{proof} (a) Suppose that $\mathbf p$ is left-definable. By Lemma \ref{Lema left definability}, after modifying $D_p$,  we may assume that $p(\Mon)$ is an initial part of $(D_p,<_p)$. Suppose that $M$ is countable and $b\in p(M)$.   Since $p$ is simple  $\exists y \,(y<_p\Dp_p(b)\land y\in D_p)$ is first-order expressible and $\Mon\models \exists y \,(y<_p\Dp_p(b)\land y\in D_p)$ (since by Corollary \ref{Cor_on p forking is equivalence}  $(p(\Mon)/\Dp_p,<_p)$ is a dense order without endpoints); the same   holds in $M$, so $\Inv_{\mathbf p}(M)$ does not have minimum. Since  $\Inv_{\mathbf p}(M)$ is a dense order type by Lema \ref{Lema ainvar of simple}, we conclude $\Inv_{\mathbf p}(M)\in\{\bm{\eta},\bm{\eta}+\mathbf 1\}$.

Assume that $M$ is prime over $a$ and let $p_r=\p_{r\,\strok a}$. Since $p$ is non-isolated and $\mathbf p$ is left-definable, by  Lemma \ref{Lema_nonisolated_frakp}(b) $p_r\in S_1(a)$ is  non-isolated. In particular, $p_r$ is omitted in   $M$, so   $\Dp_p^M(a)$ is the maximum of $(p(M)/D_p^M,<_p)$ and $\Inv_{\mathbf p}(M)=\bm{\eta}+\mathbf 1$. 
If $\delta(\mathbf q,\mathbf p)$, then $\mathbf q$ is left-definable by Lemma  \ref{Lema left definable nwor},  so $\Inv_{\mathbf q}(M)\in\{\bm{\eta},\bm{\eta}+\mathbf 1\}$. Since $\Inv_{\mathbf q}(M)$ and $\Inv_{\mathbf p}(M)$ are shuffled, at most one of them has a maximum, so $\Inv_{\mathbf q}(M)=\bm{\eta}$. This completes the proof of part (a) of the lemma; part (b) is proved similarly. 

(c) Assume that $\mathbf p$ is non-definable, $M$ is prime over $a\models p$ and $\delta(\mathbf p, \mathbf q)$. By Lemma \ref{Lema_nonisolated_frakp}(c)  the types $\p_{r\,\strok a}\in S_1(a)$ and $\p_{l\strok a}$ are non-isolated and hence omitted in $M$. Hence $\Inv_{\mathbf p}(M)=\mathbf 1$ and $\Inv_{\mathbf q}(M)=\mathbf 0$ by Lemma   
\ref{prop_shuffled case local}(b). 
\end{proof}

\subsection{Binary stationarily ordered theories with few countable models}\label{s binary so with few}

Here we continue the analysis of models of binary stationarily ordered theories   and work toward proving Theorems  \ref{conclusion} and \ref{Thm_intro_main}.  

\smallskip\noindent  {\bf Assumption.}\ Throughout this subsection we assume that $T$ is a  binary, stationarily ordered  theory satisfying none of the conditions (C1)-(C3): $T$ is small and  every type $p\in S_1(\emptyset)$ is convex and simple.

\begin{notat} Recall that (by \ref{dfn_dependent_elements}) $\Dp(a)=\{x\in \Mon\mid x\dep a\}$ and   define: $\Dpf(A)=\bigcup_{a\in A}\Dp(a)$.
  \end{notat}

\begin{lem}\label{Lema_Bi subset eps A_i}
Suppose that   $(\Dpf(A_i)\mid i\in I)$ are  pairwise disjoint sets and  $B_i\subseteq \Dpf(A_i)$  for all $i\in I$. Denote: $A_I=\bigcup_{j\in I}A_j$, $B_I=\bigcup_{j\in I}B_j$ and $B_{I-i}=\bigcup_{j\in I\smallsetminus\{i\}}B_j$.
Then:
\begin{enumerate}[(a)]
\item If  $\tp(B_i/A_i)=\tp(B_i'/A_i)$   for all $i\in I$, then $\tp(B_I)=\tp(B_I')$;
\item For all $i\in I$ we have $\tp(B_i/A_i)\vdash \tp(B_i/B_{I-i}A_I)$.
\item If $B_i$ is atomic over $A_i$ for all $i\in I$, then $B_I$ is atomic over $A_I$. 
\end{enumerate}
\end{lem}
\begin{proof} Without loss of generality  we will assume that $B_i\supseteq  A_i$   for all $i\in I$.

(a)  In order to  prove $\tp(B_I)=\tp(B_I')$, since $T$ is binary, it suffices to show that any pair $(c,d)$ of elements of $B_I$ has the same type  as the corresponding pair $(c',d')$ of elements of $B_I'$. Fix  such pairs and indices $i,j\in I$ satisfying $c\in B_i,c'\in B_i'$ and $d\in B_j,d'\in B_j'$.   Since $c$ and $c'$ have the same type over  $A_i$ and belong to $\Dpf(A_i)=\bigcup_{a\in A_i}\Dp(a)$, there is  $a_c\in A_i$ such that  $c,c'\in\Dp(a_c)$; then  $c\dep a_c$ and  $c'\dep a_c$ imply $c\dep c'$.
Similarly, there exists  $a_d\in A_j$ such that $d,d'\in \Dp(a_d)$. 
We have two cases to consider.  The first
is when  $c\ind d$ holds. In this case $c\dep c'$, $d\dep d'$ and $c\ind d$, by applying Lemma \ref{lem rmk closed for Dp}(b) twice, imply that $\tp(c,d)=\tp(c',d')$. 
In the second case  we have $c\dep d$. Then by transitivity  $a_c\dep a_d$ holds and hence $i=j$ (because $(\Dpf(A_i)\mid i\in I)$ are  pairwise disjoint sets).  Thus $c,d\in B_i,c',d'\in B_i'$, so   $\tp(B_i)=\tp(B_i')$ implies $\tp(c,d)=\tp(c',d')$. The proof of part (a) is complete.

\smallskip
(b) Fix $i\in I$ and suppose  that  $B_i'$ satisfies  $\tp(B_i'/A_i)=\tp(B_i/A_i)$. For $j\in I\smallsetminus \{i\}$ define $B_j'=B_j$; then  $B_{I-i}'=B_{I-i}$. Note that the premise of part (a) of the lemma is satisfied, so we have $\tp(B_I)=\tp(B_I')$ and thus $\tp(B_i/B_{I-i}A_I)=\tp(B_i'/B_{I-i}A_I)$. We have just proved that every realization of $\tp(B_i/A_i)$ has the same type over $B_{I-i}A_I$ as $B_i$ does; the conclusion follows.

\smallskip
(c) By binarity it suffices to prove that $\tp(b',b''/A_I)$ is isolated for $b',b''\in B_I$. If $b',b''$ are in the same $B_i$, then by the assumption $\tp(b',b''/A_i)$ is isolated and by (b) we have $\tp(b',b''/A_i)\forces \tp(b',b''/A_I)$, so $\tp(b',b''/A_I)$ is isolated. If $b'\in B_i$, $b''\in B_j$ and $i\neq j$, then both $\tp(b'/A_i)$ and $\tp(b''/A_j)$ are isolated. By part (b) we have $\tp(b'/A_i)\forces\tp(b'/A_I)$, but also $\tp(b''/A_j)\forces\tp(b''/A_Ib')$ because $b'\in B_{I-j}$. By transitivity of isolation $\tp(b',b''/A_I)$ is isolated.
\end{proof}

\begin{prop}\label{Lema_D(a) atomic over a}
If $T$ satisfies condition (C4), then $I(\aleph_0,T)=2^{\aleph_0}$. 
\end{prop}
\begin{proof}
Witness (C4) by $b,c\in\Mon$:  $b\dep c$ (i.e.\ $b\in \Dp(c)$) and $\tp(b/c)$ is non-isolated.   Let $p=\tp(c)$ and $q=\tp(c,b)$.  Then $q\in S_2(\emptyset)$ is a non-isolated type and $(a',b')\models q$ implies  $b'\in\Dp (a')$. Choose  an order $(D_p,<_p)$ witnessing that $(p,<_p)$ is a convex \so-pair. Since $T$ is small  there is a countable  saturated  $M\models T$.  Let $A\subseteq M$ be a transversal of all  $\dep$-equivalence classes of $M$ chosen such that whenever some class contains a realization of $p$, then the representative realizes $p$. By saturation, the set $A_p=A\cap p(\Mon)$ is ordered by $<_p$ in the order-type $\bm{\eta}$:  
 $A_p=\{a_i\mid i\in \mathbb Q\}$. For  any $I\subseteq\mathbb Q$, we will find a countable  $M_I\models T$ with $A\subseteq M_I\subseteq M$  and: \setcounter{equation}{0}
\begin{equation}
\mbox{ for all $a\in p(M_I)$:\  $q$ is realized in $\Dp(a)\cap M_I$ \ if and only if \ $a\dep a_i$ for some $i\in I$.}
\end{equation}
Then  the set   $\{\Dp_p(a)\mid a\in p(M_I)  \mbox{ and $q$ is realized in $\Dp(a)\cap M_I$}\}$ ordered by $<_p$ would be isomorphic to $(I,<_{\mathbb Q})$  and for $I,J\subseteq \mathbb Q$ having distinct order-types the corresponding models $M_I$ and $M_J$ would be non-isomorphic. Since there are continuum many non-isomorphic suborders of the rationales the proof of the lemma  will  be complete. 
Fix $I\subset \mathbb Q$ and let $A_I=\{a_i\mid i\in I\}$. Choose a sequence $\{\bar b_a\mid a\in A\}$   in the following way:

-- If $a\in A_I$ and $a=a_i$  choose   $b_i\in M$ with $(a_i,b_i)\models q$ and put $\bar b_a=a_ib_i$;  

-- If $a\notin A_I$, then put $\bar b_a=a$.
\\
Since $T$ is small for each $a\in A$ there is a model $M_a\subseteq M$ which is prime over $\bar b_a$; let $B_a=M_a\cap \Dp(a)$ and put $M_I=\bigcup_{a\in A}B_a$.

\smallskip
{\it Claim 1.} \ \ $M_I$ is an elementary  submodel of $M$.

\smallskip {\em Proof of Claim 1.} \ The proof is similar to that of  Proposition \ref{Prop decomposition}(b).
Let  $\phi(x,\bar c)$  be a consistent formula with parameters from $M_I$  and let $d\in \phi(M,\bar c)$. After adding dummy parameters we may assume that 
$\bar c=\bar c_0,...,\bar c_n$ where each $\bar c_k$ is a tuple of elements from the same $\dep$-class of $M$:  $\bar c_k\in B_{a^k}$, $a^k\in A$ and $\bar b_{a^k}\subseteq \bar c_k$; also, we may assume $d\in\Dpf(\bar c_n)$.   
We have:  
\begin{center}
$d\bar c_n\in\Dpf(\bar c_n)$ \  and \
$\Dpf(\bar c_n)=\Dpf(\bar b_{a^n})=\Dp(a^n)$ is disjoint from  $\Dpf(\bar c_0...\bar c_{n-1})=\bigcup_{k< n}\Dp(a^k)$.
\end{center}
In this situation  Lemma \ref{Lema_Bi subset eps A_i}(b) applies and we get  $\tp(d/\bar c_n)\vdash \tp(d/\bar c)$.  
 In particular, $\tp(d/\bar c_n)\vdash \phi(x,\bar c)$ holds and, by compactness, there is a formula $\psi(x,\bar c_n)\in\tp(d/\bar c_n)$ implying $\phi(x,\bar c)$. By Lemma \ref{Lema a dep b witness} such  a formula can be chosen so that $\psi(x,\bar c_n)$ forces $x\in \Dp(a^n)$.
Since   $\psi(x,\bar c_n)$ is consistent and  $\bar c_n\in M_{a^n}$, it has a realization    $d'\in M_{a^n}$. Then    $d'\in \Dp(a^n)$  implies   $d'\in B_{a^n}\subseteq M_I$. Now, $d'\in M_I$ is a realization of $\psi(x,\bar c_n)$ and $\psi(x,\bar c_n)\vdash \phi(x,\bar c)$ imply that $d'\in M_I$ satisfies $\phi(x,\bar c)$;
$M_I$ is an elementary submodel of $M$. \  $\Box_{Claim}$

\smallskip
{\it Claim 2.} \ If $(a',b')\in q(M_I)$, then   $a'\dep a_i$ for some $i\in I$. 

\smallskip
{\it Proof of Claim 2.} \ Suppose, on the contrary, that $(a',b')\in q(M_I)$ and   $a'\ind a_i$ for all $i\in I$. By construction of $M_I$  there is $j\in\mathbb Q\smallsetminus I$ such that $a',b'\in B_{a_j}=M_{a_j}\cap \Dp(a_j)$.  Then $j\notin I$ implies that $M_{a_j}$ is atomic over $a_j$, so $\tp(a',b'/a_j)$ is isolated and hence $\tp(b'/a'a_j)$ and $\tp(a'/a_j)$ are isolated. By symmetry of isolation  proved in Lemma   \ref{Lema_symmetry of isolation}   we conclude 
that $\tp(a_j/a')$ is isolated. Since  $\tp(b'/a'a_j)$ is isolated, by transitivity of isolation we have that $\tp(b',a_j/a')$ is isolated and hence 
$\tp(b'/a')$ is  isolated; a contradiction.  \  $\Box_{Claim}$

\smallskip

  By construction of $M_I$,  $q$ is realized in  each $B_{a_i}$ for $i\in I$; by Claim 2, these are the only  $\dep$-classes of $M_I$ that realize $q$. Therefore, $M_I$ satisfies condition (1) completing the proof of the lemma.  
\end{proof}

\section{Proofs of Theorems 1 and 2}

\noindent {\bf Assumption and notation.} Throughout this section we will assume that $T$ is a countable, binary, stationarily ordered theory. Later on, further assumptions will be added. Let $S_1^{ni}(\emptyset)$ be  the set of all non-isolated complete 1-types, let $w_T=|S_1^{ni}(\emptyset)/\nwor|$ and fix:
\begin{itemize}
\item  $\mathcal F_T$ -- a set of representatives of all $\nfor$-classes of $S_1^{ni}(\emptyset)$;
\item $(\alpha_i\mid i<w_T)$ -- an enumeration of all $\nwor$-classes of $S_1^{ni}(\emptyset)$;
\item $\alpha_i^{\mathcal F}=\alpha_i\cap \mathcal F_T$ \ for $i<w_T$; \ note that the types from $\alpha_i^{\mathcal F}$ are pairwise shuffled.  
\end{itemize}
Also, fix orderings $(<_p\mid p\in S_1(\emptyset))$ such that:
\begin{itemize}
\item $\mathbf p = (p ,<_p)$ is an \so-pair witnessed by $(D_p,<_p)$  for each $p\in S_1(\emptyset)$. If  $p$ is convex, then $p(\Mon)$ is convex in $D_p$;

\item   $p \nwor q$  implies  $\delta(\mathbf p,\mathbf q)$ for all  $p,q\in S_1(\emptyset)$ \ (this is possible by Corollary \ref{cor_direct_choice}).
\end{itemize}

Then Theorem 1 translates to:
\begin{enumerate}[(a)]
\item  $I(\aleph_0,T)=2^{\aleph_0}$ if and only if at least one of the following conditions holds:
\begin{enumerate}[(C1)]
\item $T$ is not small;
\item there is a non-convex type $p\in S_1(\emptyset)$;
\item there is a non-simple type $p\in S_1(\emptyset)$;
\item there is a non-isolated forking extension of some $p\in S_1(\emptyset)$ over an 1-element domain;
\item there are infinitely many $\nwor$-classes of non-isolated types in $S_1(\emptyset)$ (i.e.\ $w_T\geqslant \aleph_0$).
\end{enumerate}
\item $I(\aleph_0,T)=\aleph_0$ if and only if none of (C1)--(C5) hold and $ |\mathcal F_T|=\aleph_0$ ; 
\item $I(\aleph_0,T)<\aleph_0$ if and only if none of (C1)--(C5) hold and $ |\mathcal F_T|<\aleph_0$.
\end{enumerate}
 
\subsection{Proof of Theorem 2}\label{Subsection 8.1} In this subsection we will prove Theorem 2:\\ If $T$ is a complete, countable, binary, stationarily ordered theory satisfying none of the conditions (C1)--(C4), then for all countable models $M$ and $N$:
\begin{center}
$M\cong N$ \   if and only if \ 
$(\Inv_{\mathbf{p}}(M)\mid p\in \mathcal F_T)=(\Inv_{\mathbf{p} }(N)\mid p\in \mathcal F_T)$.
\end{center}

\noindent {\bf Assumption.} Throughout this subsection  assume  that none of the (C1)--(C4) holds. Hence:
\begin{itemize} 
\item $T$  is  small  and every type $p\in S_1(\emptyset)$ is convex and simple (the failure off (C1)-(C3));
\item If $a\dep b$, then $\tp(a/b)$ is isolated (the failure of (C4)).
\end{itemize}
Under these assumptions we will prove that $T$ satisfies the conclusion of Theorem \ref{Thm_intro_main}. We start by sorting out a few technicalities.

\begin{lem}\label{Prop almost aleph0 cat}  
 $T$ is  almost $\aleph_0$-categorical: the type $\bigcup_{i<n} \,p_i(x_i)$ has finitely many completions in $S_n(\emptyset)$ for all integers $n>0$ and types $p_i\in S_1(\emptyset)$ ($i<n$). 
\end{lem}
\begin{proof}
 Suppose that $p_0,p_1\in S_1(\emptyset)$. Let $a\models p_0$ and  $F=\{q\in S_1(a)\mid   \mbox{$q$ is a forking extension of $p_1$}\}$.   If $p_0\fwor p_1$, then $F=\emptyset$. 
 Otherwise, there is $b\models p_1$ such that  $a\dep b$. Since $p_1$ is convex and simple, by Remark \ref{Rmk simple type} the set $\Dp_{p_1}(b)$ is  $b$-definable; it is also $a$-definable (by $\psi(x,a)$ say) because $a\dep b$ and the transitivity of $\dep$ imply the invariance of $\Dp_{p_1}(b)$ under $a$-automorphisms.  
Now, by  transitivity of $\dep$ again, $F$ consists of all types from $S_1(a)$   that are consistent with  $\psi(x,a)$. In either case $F$ is a closed subset of $S_1(a)$. By assumptions on $T$,  each member of $F$ is an isolated type so, by compactness, $F$ is a finite set. Since $p_1$ has at most two non-forking extensions in $S_1(a)$,     the overall number of its  extensions   in $S_1(a)$ is finite. This proves the case $n=2$. The general case follows by binarity. 
\end{proof}

\begin{dfn}The decomposition of a model $M$ is a partition: $M=\at(M)\ \dot\cup\ \dot\bigcup_{i<w_T}\alpha_i(M) $, where:

-- \  $\at(M)=\{a\in M\mid \mbox{$\tp(a)$ is isolated}\}$ \ is called the atomic part of $M$;

-- \ $\alpha_i(M)=\{a\in M\mid \tp(a)\in\alpha_i\}$ \ is the $\alpha_i$-part of $M$  ($i<w_T$).
\end{dfn}

\begin{lem}\label{Lema decomposition parts}
 (a) The atomic part of a model $M$ is its (maximal) atomic submodel.

(b) The decomposition parts are orthogonal: for each non-empty part $P\in \{\at(M)\}\cup\{\alpha_i(M)\mid i<w_T\}$ we have  $\tp(P)\wor\tp(M\smallsetminus P)$.

(c)  $M$ is atomic over $M\smallsetminus \at(M)$.
\end{lem}
\begin{proof}
(a) This is an easy consequence of almost $\aleph_0$-categoricity: If $\bar a=(a_1,...,a_n)\in\at(M)$, then each   $\tp_{x_i}(a_i)$ is isolated. Since there are finitely many completions of $\bigcup_{i\leqslant n}\tp_{x_i}(a_i)$ in $S_n(\emptyset)$, each of them is an isolated type  and so is $\tp(\bar a)$. Further, if $\phi(x,\bar a)$ is consistent, then it has a realization $b\in M$ such that $\tp(b/\bar a)$ is isolated. By transitivity of isolation $\tp(b,\bar a)$ is isolated and $b\in\at(M)$; $\at(M)$ is a model.

(b) If  $P=\alpha_i(M)$, then by Proposition \ref{Prop decomposition}(a) we have $\tp(P)\wor \tp(\Mon\smallsetminus \alpha_i(\Mon))$ and thus $\tp(P)\wor\tp(M\smallsetminus P)$. Now, assume $P=\at(M)$ and let $\bar c\subseteq P$. Then $\tp(\bar c)$ is isolated by part (a) and, since $\tp(b)$ is convex and non-isolated for all $b\in M\smallsetminus P$, Corollary \ref{Cor_isolated forces over non-isolated} applies: $\tp(\bar c)\vdash \tp(\bar c/M\smallsetminus P)$, i.e.\ $\tp(\bar c)\wor \tp(M\smallsetminus P)$.
  
(c) It suffices to show that $\at(M)$ is atomic over $M\smallsetminus \at(M)$. Let $\bar c\subseteq \at(M)$. Then $\tp(\bar c)$ is isolated by part (a) and $\tp(\bar c)\wor \tp(M\smallsetminus \at(M))$ by part (b); hence $\tp(\bar c/M\smallsetminus \at(M))$ is isolated.
\end{proof}

The previous lemma guarantees that the decomposition parts of a model are mutually  independent. In the next lemma we will prove that the independence is quite strong;  e.g. by replacing one part of a model by the corresponding part of some other model, we still have a model.
 
\begin{lem}\label{Lema model of alfas}
If $(M_i\mid i< w_T)$ is a sequence of models ($\prec \Mon$) and  $M_{at}$ is an atomic model,  then $M=M_{at}\cup\bigcup_{i<w_T}\alpha_i(M_i)$ is a model of $T$. Moreover,  $\alpha_i(M)=\alpha_i(M_i)$ and $\at(M)=M_{at}$.
\end{lem}
\begin{proof}This is a consequence of  Proposition \ref{Prop decomposition}(b). Here we have an enumeration $(\alpha_i\mid i<w_T)$ of  $\nwor$-classes of non-isolated types, while in \ref{Prop decomposition} all the $\nwor$-classes are enumerated. So, choose an enumeration $(\alpha_j\mid w_T\leqslant j<\kappa)$ of all the $\nwor$-classes of isolated types from $S_1(\emptyset)$; then $(\alpha_i\mid i<\kappa)$ enumerates all the $\nwor$-classes of $S_1(\emptyset)$. Extend our sequence of models by defining $M_j=M_{at}$ for  $w_T\leqslant j<\kappa$. Then the sequence  $(M_i\mid i< \kappa)$ satisfies the assumption of Proposition \ref{Prop decomposition}(b),  so $\bigcup_{i<\kappa}\alpha_i(M_i)$ is a model of $T$. Now, $M_{at}=\bigcup_{w_T\leqslant j <\kappa}\alpha_j(M_j)$ implies $\bigcup_{i<\kappa}\alpha_i(M_i)=M_{at}\cup\bigcup_{i<w_T}\alpha_i(M_i)=M$;  the conclusion follows. 
\end{proof}

\begin{notat}
For each countable model $M\models T$:
\begin{itemize}
\item Let $I_{p}(M)$ be a fixed $\dep$-transversal of $p(M)$ for each $p\in\mathcal F_T$;
\item  $I_{\alpha_i}(M)=\bigcup_{p\in \alpha_i^{\mathcal F}}I_p(M)$ for $i<w_T$; \ \  $I(M)=\bigcup_{i<w_T} I_{\alpha_i}(M)$. 
\end{itemize}
Note that $(I_p(M),<_p)$ has order type $\Inv_{\mathbf p}(M)$ for each $p\in\mathcal F_T$.
\end{notat}

\begin{lem}\label{Lema M prime over I(M)}
If  $M$ is countable, then   $\alpha_i(M)$ is atomic over $I_{\alpha_i}(M)$ ($i<w_T$) and $M$ is prime over $I(M)$.
\end{lem}
\begin{proof}
Suppose that $b\in M\smallsetminus\at(M)$ and let $p\in \mathcal F_T$ be such that $p\nfor \tp(b)$.   Then by  Proposition \ref{prop_bounded case}(b)  $\emptyset\neq \Dp_p^M(b)$ implies that there is $a\in I_p(M)$ with $a\dep b$; in particular,  $\bigcup_{a\in I(M)}\Dp^M(a)=M\smallsetminus \at(M)$. Next we show that $\Dp^M(b)=\Dp^M(a)$ is atomic over $a$:   if $\bar c\subseteq \Dp^M(a)$ and $c_i\in \bar c$, then $c_i\dep a$ implies that 
 $\tp(c_i/a)$ is isolated (by $\psi_i(x_i,a)$ say);  by almost $\aleph_0$-categoricity of $T$, the type $\{\psi_i(x_i,a)\mid  i<|\bar c|\}$ has finitely many completions in $S_{|\bar c|}(a)$,  so all of them are isolated; hence $\tp(\bar c/a)$ is isolated.
We have a pairwise disjoint family $(\Dp^M(a)\mid a\in I(M))$ with each $\Dp^M(a)$ atomic over $a$; in this situation Lemma \ref{Lema_Bi subset eps A_i}(c) applies. First we apply it to $\bigcup_{a\in I_{\alpha_i}(M)}\Dp^M(a)=\alpha_i(M)$ and conclude that $\alpha_i(M)$ is atomic over $I_{\alpha_i}(M)$. Next we apply it to   
$\bigcup_{i<w_T}\alpha_i(M)=M\smallsetminus \at(M)$ and conclude that it is atomic over $I(M)$. By Lemma \ref{Lema decomposition parts}  $M$ is atomic over $M\smallsetminus \at(M)$, so by transitivity of atomicity $M$ is atomic over $I(M)$; being countable, $M$ is prime over $I(M)$.
\end{proof} 

 \begin{proof}[{\it Proof of Theorem 2}]   Assuming  that $M,N\models T$ are countable and $\Inv_{\mathbf p}(M)=\Inv_{\mathbf p}(N)$  for all $p\in \mathcal F_T$, we will prove $M\cong N$. 

{\it Claim.} \   $I_{\alpha_i}(M)$ and $I_{\alpha_i}(N)$ have the same elementary type for all $i<w_T$. 

{\it Proof.} \ Fix $i<w_T$  and for simplicity assume that $\alpha_i^{\mathcal F}=\{p_n\mid n<\beta\}$ where $\beta\leqslant\aleph_0$.  
The case   $| I_{\alpha_i}(M)|=1$ is easy: in this case there is   $p\in\alpha_i$ with $I_{\alpha_i}(M)=I_p(M)=\{c\}$, so the equality of invariants of $M$ and $N$ implies $I_{\alpha_i}(N)=I_p(N)=\{c'\}$, and $\tp(c)=\tp(c')=p$ proves the claim. Assume now $|I_{\alpha_i}(M)|\geqslant 2$.  Recall that types from $\alpha_i^{\mathcal F}$ are pairwise shuffled, so 
  Proposition  \ref{Prop commutativity in M} applies  to $(\mathbf{p}_n=(p_n,<_{p_n})\mid n<\beta)$ and  $M$: for $S_{n,m}^M=S_{\mathbf{p}_n,\mathbf{p}_m}\cap(I_{p_n}(M)\times I_{p_m}(M))$ we have:
\begin{center}
$((I_{p_n}(M),<_{p_n})\mid n<\beta)$ are countable dense  orders shuffled by the family $(S_{n,m}^M\mid n<m<\beta)$.
\end{center}
The same holds for $N$: the orders  $((I_{p_n}(N),<_{p_n})\mid n<\beta)$ are  shuffled by  $(S_{n,m}^N\mid n<m<\beta)$.
 Since the  orders $(I_{p_n}(M),<_{p_n})$ and $(I_{p_n}(N),<_{p_n})$ are isomorphic for all  $n<\beta$, condition (3) of  Proposition \ref{Prop projekcije izomorfne iff} is satisfied and hence condition (2) is satisfied, too: there are order-isomorphisms $f_n:I_{p_n}(M)\to I_{p_n}(N)$ such that for all $n<m<\beta$, $a\in I_{p_n}(M)$ and $a'\in I_{p_m}(M)$ we have: 
\begin{center}
$(a,a')\in S_{n,m}^M$ \ if and only if \ $(f_n(a),f_m(a'))\in S_{n,m}^N$. 
\end{center}
Now it is not hard to verify that $f_{\alpha_i}=\bigcup_{n<\beta}f_n: I_{\alpha_i}(M)\to I_{\alpha_i}(N)$ is an elementary mapping; since $T$ is binary  it suffices to show that for any pair  of  elements $a,a'\in I_{\alpha_i}(M)$ we have $\tp(a,a')=\tp(f_{\alpha_i}(a),f_{\alpha_i}(a'))$.
 If $a,a'\in I_{p_n}(M)$  and   $a<_{p_n}a'$,  then $\tp(a,a')=\tp(f_{\alpha_i}(a),f_{\alpha_i}(a'))$ holds because $f_n$ is an order isomorphism, so both $(a,a')$ and $(f_{\alpha_i}(a),f_{\alpha_i}(a'))$ are Morley sequences in the right globalization of $\mathbf p_n$. If $a\in I_{p_n}(M)$, $a'\in I_{p_m}(M)$ and $n<m<\beta$,  then 
by Lemma \ref{rmk def shuffled}(a) the type  $\tp(a,a')$ is determined by  $p_n(x)\cup p_m(x')$ plus the information whether $ (a,a')\in S_{n,m}^M$ (i.e.\ $a\in\Lp_{\mathbf p_n}(a')$) or
$ (a,a')\notin S_{n,m}^M$ (i.e.\ $a\in\Rp_{\mathbf p_n}(a')$) holds.   Since $f_{\alpha_i}$ induces a bijection of  $S_{n,m}^M$ and $S_{n,m}^N$, the pair $(f_{\alpha_i}(a),f_{\alpha_i}(a'))$ satisfies the same  type as $(a,a')$ does, so $\tp(a,a')=\tp(f_{\alpha_i}(a),f_{\alpha_i}(a'))$. This completes the proof of $\tp(I_{\alpha_i}(M))=\tp(I_{\alpha_i}(N))$. \ $\Box_{Claim}$

To finish the proof of the theorem, it remains to note that  for any pair of elements $a,a'\in I(M)$ we have $\tp(a,a')=\tp(f(a),f(a'))$, where $f=\bigcup\{f_{\alpha_i}\mid i<w_T\}$: if $a,a'$ are from the same $I_{\alpha_i}(M)$ that was shown above and if $a,a'$ are from  distinct $I_{\alpha_i}$ we have $\tp(a)\wor\tp(a')$ and  $\tp(a,a')=\tp(f(a),f(a'))$ holds because $f$ preserves types of elements. This proves that the bijection $f:I(M)\to I(N)$ is (partial)  elementary. By Lemma \ref{Lema M prime over I(M)} $M$ is prime over $I(M)$ and $N$ is prime over $I(N)$, so the uniqueness of prime models implies that $M$ and $N$ are isomorphic. 
\end{proof}
 
\subsection{The number of countable models} 
In this subsection we will keep the assumption  from the previous subsection: $T$ satisfies none of (C1)-(C4). 
In this situation,  Theorem \ref{Thm_intro_main} applies: the isomorphism type of a  countable model $M$ is determined by   $(\Inv_{\mathbf{p}}(M)\mid p\in\mathcal F_T)$.  Define: 
 \begin{center}
 $\mathcal I_{\alpha_i}(M)=(\Inv_{\mathbf{p}}(M)\mid p\in\alpha_i^{\mathcal F})$  \ \ and \ \  $\kappa_i=|\{\mathcal I_{\alpha_i}(M)\mid \mbox{$M\models T$ countable}\}|$  \ \   ($i<w_T$).
 \end{center}

\begin{lem}\label{Lema_no_ofctble1}
$I(\aleph_0,T)=\Pi_{i<w_T}\kappa_i$ and  if $w_T=\aleph_0$ then $I(\aleph_0,T)=2^{\aleph_0}$.
\end{lem}
\begin{proof} 
By Lemma \ref{Lema M prime over I(M)} each $\alpha_i(M)$ ($i<w_T$) is atomic over $I_{\alpha_i}(M)$ and by (the claim from the proof of) Theorem 2, the elementary type of $I_{\alpha_i}(M)$ is determined by $\mathcal I_{\alpha_i}(M)$; hence the isomorphism type of $\alpha_i(M)$ is determined by $\mathcal I_{\alpha_i}(M)$.  
 By Lemma \ref{Lema model of alfas}  we may  assemble a countable submodel of $\Mon$ by taking for each $i<w_T$  the $\alpha_i$-part of some countable  model ($\kappa_i$ possibilities) and then add  to the union some atomic submodel; we have \  $I(\aleph_0,T)=\Pi_{i<w_T}\kappa_i$. Since $\alpha_i(M)=\emptyset$ and $\alpha_i(M)\neq \emptyset$ are possible we have $\kappa_i \geqslant 2$; hence $w_T=\aleph_0$ implies $I(\aleph_0,T)=2^{\aleph_0}$.
\end{proof}

In order to compute  $\kappa_i$ in terms of $|\alpha_i^{\mathcal F}|$ we have to distinguish two kinds of classes $\alpha_i$. We   say that $\alpha_i$ is a definable $\nwor$-class if for some (equivalently all by Lemma \ref{Lema left definable nwor})  $p\in\alpha_i$  the \so-pair $\mathbf p$ is (left-)\,right-definable; otherwise, we say that $\alpha_i$ is a non-definable class.

Since the \so-pairs corresponding to   distinct members of $\alpha_i^{\mathcal F}$ are pairwise  shuffled, by Proposition \ref{prop_shuffled case local} we have the following possibilities for the order-types in $\mathcal I_{\alpha_i}(M)$ when $M$ is countable:

 (P1) \  All of them are $\mathbf 0$;

(P2)  \  Exactly one of them is $\mathbf 1$ and all the others are $\mathbf 0$ \  ($|\alpha_i^{\mathcal F}|$ possibilities);

(P3) \  $(\Inv_{\mathbf{p}}(M)\mid p\in\alpha_i^{\mathcal F})$ is a sequence of dense, pairwise shuffled order-types.

\smallskip
\noindent By Remark \ref{Remark posssible order-types} we have at most  $(|\alpha_i^{\mathcal F}|+1)^2$ possibilities in (P3), so $\kappa_i\leqslant |\alpha_i^{\mathcal F}|+3 |\alpha_i^{\mathcal F}|+ 2$.

\begin{lem}\label{Lema_no_ofctble2}
  If $\alpha_i$ is a definable class,  then $\kappa_i=|\alpha_i^{\mathcal F}|+2$.   
\end{lem}
\begin{proof}   Without loss   assume that $\mathbf p$ is left-definable for all $p\in \alpha_i$. By Lemma \ref{Lema invariants prime model}(a)  we have $\Inv_{\mathbf p}(M)\in\{\mathbf 0,\bm{\eta}+\mathbf 1,\bm{\eta}\}$; this reduces the possibilities in (P1)-(P3) to: 
 
 (1) \  All of them are $\mathbf 0$;

(2)  \  Exactly one of them is $\bm{\eta}+\mathbf 1$ and all the others are $\bm{\eta}$;

(3) \ All of them are $\bm{\eta}$.\\
Clearly, (1) is realized in the prime model and (3) in a countable, saturated model. By Lemma \ref{Lema invariants prime model}(a), for each $p\in\alpha_i^{\mathcal F}$ the prime model over a realization of $p$ realizes   (2) with $\Inv_{\mathbf p}(M)=\bm{\eta}+\mathbf 1$; altogether we have $|\alpha_i^{\mathcal F}|+2$ possibilities.
\end{proof}

\begin{lem}\label{Lema_no_ofctble3}
If $i<w_T$ and $\alpha_i$ is a non-definable class, then $\kappa_i=|\alpha_i^{\mathcal F}|^2+3|\alpha_i^{\mathcal F}|+2$.
\end{lem} 
\begin{proof} We will prove that each of the   (P1)-(P3) is realized in some countable model. 
Clearly, (P1) is realized in the prime model and  Lemma \ref{Lema invariants prime model}(c) implies that each of the $|\alpha_i^{\mathcal F}|$ (P2)-possibilities is realized, too. It remains to show that each of the (P3)-possibilities is realized.  
Let  $p,q\in\alpha_i^{\mathcal F}$ (possibly $p=q$) and choose $a_p\models p$ and $b_q\models q$ with $a_p\in\Lp_{\mathbf p}(b_q)$. Since $T$ is small, there is a model  $M_{p,q}$   prime over $a_pb_q$. 

\smallskip
{\it Claim 1.} \ \ If $r\in\alpha_i$, then $\r_{r\,\strok b_q}$ and $\r_{l\,\strok a_p}$ are omitted in $M_{p,q}$. 

 {\it Proof.} \ Fix $r\in \alpha_i$ and let $c\in \Mon$ realize $\r_{r\,\strok b_q}$; then $c\in \Rp_{\mathbf r}(b_q)$ and $\delta(\mathbf q,\mathbf r)$, by Lemma \ref{Lema_direct_non-orth}(b), implies $b_q\in\Lp_{\mathbf q}(c)$; thus
$(a_p,b_q)\in S_{\mathbf p,\mathbf q}$ and $(b_q,c)\in S_{\mathbf q,\mathbf r}$. By Lemma \ref{Lema_delta_implies_S+pq_commute} the direct non-orthogonality of $\mathbf p,\mathbf q$ and $\mathbf r$ implies $S_{\mathbf q,\mathbf r}\circ S_{\mathbf p,\mathbf q}=S_{\mathbf p,\mathbf r}$, 
so $(a_p,c)\in S_{\mathbf p,\mathbf r}$ and $a_p\in\Lp_{\mathbf p}(c)$;   by Lemma \ref{Lema_direct_non-orth} we have $c\in\Rp_{\mathbf r}(a_p)$, so $c\models \r_{r\,\strok a_p}$. Since $T$ is binary: $c\models \r_{r\,\strok a_p}$ and $c\models \r_{r\,\strok b_q}$ imply $c\models \r_{r\,\strok a_pb_q}$. Since $\mathbf r$ is non-definable   Lemma \ref{Lema_nonisolated_frakp}(c) applies, so 
   $\r_{r\,\strok a_pb_q}\in S_1(a_pb_q)$ is   non-isolated and hence omitted in $M_{p,q}$; thus $c\notin M_{p,q}$ and $\r_{r\,\strok b_q}$ is omitted in $M_{p,q}$. Similarly, $\r_{l\,\strok a_p}$ is omitted in $M_{p,q}$. $\Box_{Claim}$ 

\smallskip
{\it Claim 2.} \ If  $p\neq q$, then  $\Inv_\mathbf p(M_{p,q})=\mathbf{1}+\bm{\eta}$, $\Inv_\mathbf q(M_{p,q})=\bm{\eta}+\mathbf{1}$ and $\Inv_\mathbf r(M_{p,q})=\bm{\eta}$ for all $r\in\alpha_i$ with $r\fwor p$ and $r\fwor q$; in particular, $\Inv_\mathbf r(M_{p,q})=\bm{\eta}$ for all $r\in\alpha_i^{\mathcal F}$ with $r\neq p,q$.

{\it Proof.} \ Since $p,q\in \alpha_i^{\mathcal F}$ $p\neq q$ implies $p\fwor q$. By Claim 1 the type $\p_{l\,\strok a_p}$ is omitted in $M_{p,q}$, so $\Dp_p^{M_{p,q}}(a_p)=\min (p(M_{p,q})/\D_p^{M_{p,q}},<_p)$; hence $\Inv_\mathbf p(M_{p,q})$ has minimum. Similarly 
$\Inv_\mathbf q(M_{p,q})$ has a maximum. By Proposition \ref{prop_shuffled case local}  $p\fwor q$ implies that
$(p(M_{p,a})/\D_p^{M_{p,q}},<_p)$ and $(q(M_{p,q})/\D_q^{M_{p,q}},<_q)$ are shuffled, so at most one of them has minimum (maximum). Hence  $\Inv_\mathbf p(M_{p,q})= \mathbf{1}+\bm{\eta}$ and $\Inv_\mathbf q(M_{p,q})=\bm{\eta}+\mathbf{1}$. Similarly, $\Inv_\mathbf r(M_{p,q})=\bm{\eta}$ for any $r\in\alpha_i$ with $r\fwor p$ and $r\fwor q$. $\Box_{Claim}$

\smallskip
{\it Claim 3.} \  $\Inv_\mathbf p(M_{p,p})=\mathbf{1}+\bm{\eta}+\mathbf{1}$  and $\Inv_\mathbf r(M_{p,p})=\bm{\eta}$ for all $r\in\alpha_i$ with $r\fwor p$; in particular, $\Inv_\mathbf r(M_{p,p})=\bm{\eta}$ for all $r\in\alpha_i^{\mathcal F}$ with $r\neq p$.

{\it Proof.} \ Analogous to the proof of Claim 2. $\Box_{Claim}$

\smallskip
{\it Claim 4.} \  $N_p=M_{p,p}\smallsetminus \D(b_p)$ is a model of $T$,  $\Inv_\mathbf p(N_p)=\mathbf{1}+\bm{\eta}$  and $\Inv_\mathbf r(N_p)=\bm{\eta}$ for any $r\in\alpha_i$ with $r\fwor p$; in particular, $\Inv_\mathbf r(N_p)=\bm{\eta}$ for all $r\in\alpha_i^{\mathcal F}$ with $r\neq p$.

{\it Proof.} \  First we prove $N_p\models T$. Suppose that  $\phi(x,\bar c)$ is consistent and $\bar c\subset N_p$.  Since $T$ is small  isolated types are dense in $S_1(\bar c)$, so   there is $d\in M_{p,p}$ with $\tp(d/\bar c)$  isolated; it suffices to prove $d\in N_p$. Let $c_j\in \bar c$ and $s=\tp(c_j)$. First we {\em claim} that $c_j\models \mathfrak{s}_{l\,\strok b_p}$. 
If $s\wor p$ (i.e.\ $s\notin\alpha_i$) the claim trivially holds, so assume $s\nwor p$. Note that $c_j\in N_p$ implies  $c_j\ind b_p$, so either  $c_j\models \mathfrak{s}_{r\,\strok b_p}$ or $c_j\models \mathfrak{s}_{l\,\strok b_p}$. The first option is impossible by Claim 1, so $c_j\models \mathfrak{s}_{l\,\strok b_p}$, proving the {\em claim}.  
Now, let $b'\in \D(b_p)$ and $t=\tp(b')$; then $t\in\alpha_i$. By Lemma \ref{Prop_equivalents_of_a dep b} $b_p\dep b'$ implies $\Lp_{\mathbf s}(b_p)=\Lp_{\mathbf s}(b')$, so $c_j\in \Lp_{\mathbf s}(b')$. 
By  Lemma \ref{Lema_direct_non-orth} 
$\delta(\mathbf{t},\mathbf s)$  implies  $b'\in \Rp_{\mathbf t}(c_j)$, so $b'\models \mathfrak{t}_{r\strok c_j}$; since this holds for all $c_j\in \bar c$, by binarity, we have $b'\models \mathfrak{t}_{r\strok\bar c}$. Since $\alpha_i$ is a non-definable class the type $t\in \alpha_i$ is non-definable  and  $\tp(b'/\bar c)$ is non-isolated by Lemma \ref{Lema_nonisolated_frakp}(c). Hence $d\neq b'$ and $d\in N_p$; $N_p$ is a model of $T$. 
Since $p(N_p)$ is obtained from $p(M_{p,p})$ by removing the maximal  $\D_p^{M_{p,p}}$-class $\Dp^{M_{p,p}}_p(b_p)$ and $\Inv_\mathbf p(M_{p,p})=\mathbf{1}+\bm{\eta}+\mathbf 1 $ by Claim 3, 
we have
$\Inv_\mathbf p(N_p)=\mathbf{1}+\bm{\eta}$. For any $r\in\alpha_i$ with $r\fwor p$  we have $r(M_{p,p})=r(N_p)$, so  $\Inv_\mathbf r(N_p)=\Inv_\mathbf r(M_{p,p})=\bm{\eta}$.
$\Box_{Claim}$ 

\smallskip
{\it Claim 5.} \  $N_p'=M_{p,p}\smallsetminus \D(a_p)$ is a model of $T$,  $\Inv_\mathbf p(N_p')= \bm{\eta}+\mathbf{1}$  and $\Inv_\mathbf r(N_p')=\bm{\eta}$ for any $r\in\alpha_i$ with $r\fwor p$; in particular, $\Inv_\mathbf r(N_p')=\bm{\eta}$ for all $r\in\alpha_i^{\mathcal F}$ with $r\neq p$.

{\it Proof.} \  Analogous to the proof of Claim 4. $\Box_{Claim}$

\smallskip 
 If $M\models T$ is countable and saturated, then $\Inv_\mathbf p(M)=\bm{\eta}$ for all $p\in\alpha_i^{\mathcal F}$. All the other (P3)-possibilities  are realized  by Claims 2-5. Therefore, $\kappa_i=|\alpha_i^{\mathcal F}|^2+3|\alpha_i^{\mathcal F}|+2$.  
\end{proof}

 \begin{thm}\label{Cor exact number of ctbles} If $T$ satisfies none of (C1)-(C4) and $u\subseteq w_T$ is the set of all indices $i<w_T$ corresponding to non-definable classes $\alpha_i$, then
\ 
 $I(\aleph_0,T)=\Pi_{i\in w_T\smallsetminus u}(|\alpha_i^{\mathcal F}|+2)\cdot  \Pi_{j\in u}(|\alpha_j^{\mathcal F}|^2+3|\alpha_j^{\mathcal F}|+2)$.
\end{thm}

\subsection{Proof of Theorem \ref{conclusion}} 
If at least one of  conditions (C1)--(C5) is fulfilled, then   $I(\aleph_0,T)=2^{\aleph_0}$: for condition (C1) by Fact \ref{Fact_smalltheory}, for (C2) and (C3)  by Proposition \ref{Prop weakly_reg_convexsimple}, for (C4) by Proposition \ref{Lema_D(a) atomic over a} and for (C5) by Proposition \ref{Prop infmany nwor implies continuum}. This proves the $(\Leftarrow)$ direction of part (a). For the other direction. assume that none of (C1)-(C5) holds. Then by Theorem \ref{Cor exact number of ctbles} $I(\aleph_0,T)=\Pi_{i\in w_T\smallsetminus u}(|\alpha_i^{\mathcal F}|+2)\cdot  \Pi_{j\in u}(|\alpha_j^{\mathcal F}|^2+3|\alpha_j^{\mathcal F}|+2)$; the failure of (C5) implies that $w_T$ is finite, so   $|\alpha_i^{\mathcal F}|\leqslant \aleph_0$ implies $I(\aleph_0,T)\leqslant \aleph_0$. This proves part (a) of Theorem 1. 

Now, assume that none of (C1)-(C5) holds. Then $w_T$ is finite and
$I(\aleph_0,T)=\Pi_{i\in w_T\smallsetminus u}(|\alpha_i^{\mathcal F}|+2)\cdot  \Pi_{j\in u}(|\alpha_j^{\mathcal F}|^2+3|\alpha_j^{\mathcal F}|+2)\in \omega\cup\{\aleph_0\}$; note that $I(\aleph_0,T)$ is finite if and only if each $\alpha_i^{\mathcal F}=\alpha_i\cap \mathcal F_T$ is a finite set if and only if $\mathcal F_T$ is finite. This completes the proof of Theorem 1.\qed

\begin{cor}\label{Thm_VC for stat theories}
Vaught's conjecture holds for   binary,  stationarily ordered  theories. \qed
\end{cor}

\begin{cor}\label{Cor nwor=nfor}A binary, stationarily ordered theory $T$ in which $\nwor$  and $\nfor$ agree on $S_1(\emptyset)$  has either $2^{\aleph_0}$ or $3^n\cdot 6^m$ countable models for some $m,n\in\omega$.
\end{cor}
\begin{proof}
If $\nwor$ and $\nfor$ agree on $S_1(\emptyset)$, then $|\alpha_i^{\mathcal F}|=1$ for all $i<w_T$, so $I(\aleph_0,T)=3^{|w_T|-|u|}\cdot 6^{|u|}$ by Theorem \ref{Cor exact number of ctbles}. 
\end{proof}

\section{Weakly quasi-o-minimal theories and Rubin's theorem}\label{s9}

   Kuda\u\i bergenov in   \cite{Kud} introduced the notion of a {\em weakly quasi-\oo-minimal theory},   generalizing   both weakly \oo-minimal theories and quasi-\oo-minimal theories.  A complete theory $T$ of  linearly  ordered structures is weakly quasi-\oo-minimal if   every parametrically definable subset of $\Mon$ is a finite Boolean combination of $\emptyset$-definable and convex sets. 
So, if $p\in S_1(A)$, then   every relatively $\Mon$-definable subset of   $p(\Mon)$ is  a finite Boolean combination of convex sets. If $C$ is such a combination, then either $C$ or  $p(\Mon)\smallsetminus C$ is left eventual in $p(\Mon)$; similarly for right eventual sets.
Thus, weakly quasi \oo-minimal theories are stationarily ordered in a strong sense.
\begin{lem}\label{qwom fact stat}
Any  complete 1-type  in a weakly quasi-\oo-minimal theory is stationarily ordered.
\end{lem}

As an immediate  corollary of the previous lemma and Corollary \ref{Thm_VC for stat theories} we have:

\begin{thm}\label{thm vaught bin wqom} Vaught's conjecture holds for binary, weakly quasi-\oo-minimal theories.
\end{thm}

In \cite{May} Laura Mayer proved that an  \oo-minimal theory has either continuum or $3^n\cdot 6^m$ countable models; this was recently reproved and slightly  generalized  by Rast and Sahota in \cite{Rast2}. It follows from either of these proofs that an \oo-minimal theory $T$ with  $I(\aleph_0,T)<2^{\aleph_0}$ is binary. Notice that in the \oo-minimal context $\wor=\fwor$ holds, i.e.\ there is no pair of shuffled types. Using the notation from Section 8, by Corollary \ref{Cor nwor=nfor}  we have
 $I(\aleph_0,T)= 3^{w_T-|u|}\cdot 6^{|u|}$.

  One could ask whether small, weakly quasi-\oo-minimal theories are also binary. Unfortunately, this is not true, since there are $\aleph_0$-categorical weakly \oo-minimal theories which are not binary. Examples of these can be found in \cite{Her}. Further distinction with the \oo-minimal case is the existence of a weakly \oo-minimal theory with exactly $\aleph_0$ countable models; it was originally found by Alibek and Baizhanov in \cite{AB}. We will sketch a variant of that example.

\subsection{A binary, weakly \oo-minimal theory with $\aleph_0$ countable models} \label{example} 
Let the language $L$ has binary predicate symbols $\{<\}\cup \{S_{ij}\mid i<j\in\omega\}$ and unary predicate symbols $\{O_i\mid i\in\omega\}\cup\{C_{i,n}\mid i,n\in\omega\}$. Consider countable $L$-structures $(M,<,O_i(M),C_{i,j}(M),S_{i,j}^M)_{i,j\in\omega}$ satisfying:
\begin{enumerate} 
\item $(M,<)$ is a dense linear order without endpoints;
\item for each $i\in\omega$ the set $O_i(M)$ is an open convex subset of $M$,  $O_0(M)<O_1(M)<\ldots$ and $O(M)=\bigcup_{i\in\omega}O_i(M)$ is an initial part of $M$ (we emphasize $O\notin L$);
\item for all $i,n\in\omega$, the $C_{i,n}(M)$ are open convex subsets of $O_i(M)$ (called the $n$-th color of elements of $O_i(M)$)  such that $C_{i,0}(M)<C_{i,1}(M)<\ldots$ and $C_i(M)=\bigcup_{n\in\omega}C_{i,n}(M)$ is an initial part of $O_i(M)$ (so that the coloured order $(O_i(M),<,C_{i,n}(M))_{n\in\omega}$ is a model of the theory $T_0$ with 3 countable models  from Example \ref{exm ehr}) (we emphasize $C_i\notin L$) ;
\item $S_{i,j}^M\subseteq O_i(M)\times O_j(M)$  for all $i,j\in \omega$;
\item the family $\mathcal S=(S_{i,j}^M\mid i<j\in\omega)$ shuffles the  sequence $((O_i(M),<,C_{i,n}(M))_{n\in\omega}\mid i\in\omega)$ of colored orders, by which we mean:
\begin{enumerate}
\item   $\mathcal S$ shuffles the sequence of orders  $((O_i(M),<)\mid i\in\omega)$; and 
\item  for each $n\in\omega$ the sequence   $(S_{i,j}^M\cap (C_{i,n}(M)\times C_{j,n}(M))\mid i<j\in\omega) $ shuffles the sequence of $n$-th colors  $((C_{i,n}(M),<)\mid i\in \omega)$.
\end{enumerate}
\end{enumerate}
The above properties are first-order expressible; let $T$ be the corresponding theory. We will show that $T$ is complete, weakly \oo-minimal,   eliminates quantifiers and $I(\aleph_0,T)=\aleph_0$. To describe the isomorphism types of countable models of $T$,  first we show that the $L$-substructure $C(M)=\bigcup_{i\in\omega}C_i(M)$ of a countable $M\models T$ has a unique isomorphism type. Let $N\models T$ be countable. Fix $i\in\omega$. By (5)(b) we have that 
$(S_{i,j}^M\cap (C_{i,n}(M)\times C_{j,n}(M))\mid i<j\in\omega) $ shuffles the sequence   $((C_{i,n}(M),<)\mid i\in \omega)$ and $(S_{i,j}^N\cap (C_{i,n}(N)\times C_{j,n}(N))\mid i<j\in\omega) $ shuffles the sequence   $((C_{i,n}(N),<)\mid i\in \omega)$. Since for each $n\in\omega$ the orders  $(C_{i,n}(M),<)$ and $(C_{i,n}(N),<)$ have the order type $\bm{\eta}$, 
condition (3) of  Proposition \ref{Prop projekcije izomorfne iff} is satisfied. Hence the other two conditions are satisfied   and,  as in the proof of  Theorem \ref{Thm_intro_main}, we may conclude that $C_i(M)$ and $C_i(N)$ are isomorphic substructures;  let $f_i$ be an isomorphism. Then $\bigcup_{i\in\omega}f_i:C(M)\to C(N)$ is an isomorphism of substructures; $C(M)$ is unique, up to isomorphism.  
Now define:
\begin{center}
 $p_i(x)= \{O_i(x)\}\cup \{\neg C_{i,n}(x)\mid n\in\omega\}$ \ ($i\in\omega$)
\ and  \  
$q(x)= \{\neg O_i(x)\mid i\in\omega\}$.
\end{center}
Note that  $M=C(M)\dot{\cup}q(M)\dot{\bigcup}_{i\in\omega}p_i(M)$. If $p_i(M)\neq\emptyset$ for some $i\in\omega$, then $(p_i(M),<)$ is a final part of   $(O_i(M),<)$ having the order-type $\bm{\eta}$ or $\mathbf{1}+\bm{\eta}$ and  by (5) we have: 
 
\smallskip
\ \ (6)  the sequence   $(S_{i,j}^M\cap (p_i(M)\times p_j(M))\mid i<j\in\omega) $ shuffles the sequence
$((p_i(M),<)\mid i\in\omega)$.

\smallskip\noindent 
Note that the only unary predicate with a non-void interpretation on the substructure $p_i(M)$ is $O_i$ (and $p_i(M)\subset O_i(M)$). Also, the substructure  $q(M)$ is a pure linear order  because for any unary predicate $X\in L$ the set $X(M)$ is convex and  $X(M)<q(M)$ (provided that $q(M)\neq \emptyset$); this is the reason why we gave distinct names $C_{i,n}$ and $C_{j,n}$ to $n$-th colors of distinct  $O_i(M)$ and $O_j(M)$. Arguing similarly as in the proof of $C(M)\cong C(N)$ one proves that $M\cong N$ if and only if:  

\ \ (7)  \ $(q(M),<)\cong (q(N),<)$ and $(p_i(M),<)\cong (p_i(N),<)$ for all $i\in\omega$.
 
\noindent 
To sketch the proof, assume that (7) holds and choose  isomorphisms $f:C(M)\to C(N)$ and $g:q(M)\to q(N)$. Using condition (6) and arguing as in the proof of $C(M)\cong C(N)$, we find an isomorphism $h:{\bigcup}_{i\in\omega}p_i(M)\to {\bigcup}_{i\in\omega}p_i(N)$ and conclude that $f\cup g\cup h:M\to N$ is an isomorphism.

The rest of the proof is similar to the one in Example \ref{Exm_shuffled _elquant}. 
If  $M'\models T$ is countable, then it is fairly standard to construct a countable $M^*\succ M'$ satisfying: $(q(M^*),<)$ and each $(p_i(M^*),<)$ have order-type $\bm{\eta}$; by (7) the model $M^*$ is unique up to   isomorphism, so $T$ is complete, small and $M^*$ is its countable, saturated model. 
By a   back-and-forth construction one shows that 
any pair of tuples from $M^*$ with the same quantifier-free type are conjugated by an automorphism of $M^*$; the conclusion is that 
 $T$ eliminates quantifiers.  It follows that  $T$ is binary, weakly o-minimal with  trivial forking ($a\dep b$ implies $a=b$) and with
two $\nwor$-classes of non-isolated types ($w_T=2$): the first one is $\{q\}$ and the second one consists of pairwise shuffled types  $\{p_i\mid i\in\omega\}$. $T$ satisfies none of (C1)-(C5) and by   Theorem \ref{conclusion}(b) we have  $I(\aleph_0,T)=\aleph_0$.

\subsection{Reproving Rubin's theorem} \label{subsection Rubin}
 In this subsection we will reprove Rubin's theorem: every complete theory of coloured orders has either $2^{\aleph_0}$ or finitely many  countable models (we emphasize that the colours are not necessarily  disjoint). 
That any complete theory $T$ of coloured orders satisfies a strong form of binarity  follows from  Rubin's Lemma 7.9 in \cite{Rub} (for a detailed discussion see \cite{TMI}), but here we will refer to Pierre Simon's Proposition 4.1 from \cite{SimLO}: 

\begin{prop}\label{Simon} Let $(M,<,C_i,R_j)_{i\in I,j\in J}$ be a linear order with colors $C_i$ and monotone relations $R_j$.
Assume that each unary $\emptyset$-definable set  is represented by one of the $C_i$ and each monotone 
$\emptyset$-definable binary relation is  represented by one of the $R_j$. Then the structure eliminates
quantifiers.
\end{prop}

\noindent{\bf Assumption.} From now on let $T$ be a complete theory of coloured orders. 

\smallskip
By Simon's result, if we name all unary definable sets and all monotone definable relations, the structure eliminates quantifiers. Hence $T$ is binary and any $\Mon$-definable subset of $\Mon$ is a Boolean combination of unary $L$-definable sets and fibers ($R(a,\Mon)$ and $R(\Mon,b)$) of   monotone relations. Since the fibers are convex sets, $T$ is  weakly quasi-\oo-minimal. Hence $T$ is a binary,  stationarily ordered theory. 
  
Non-orthogonality of 1-types in coloured orders is closely related to  interval types. Following Rosenstein \cite{Ros}, we say that an {\em interval type}\footnote{Other authors, notably Rubin in \cite{Rub} and Rast in \cite{Rast1}, use a term {\em convex type}.} of $T$ is a maximal partial type consisting of convex $L$-formulae, i.e.\  those defining convex sets.  We denote by $IT(T)$ the set of all interval types of $T$. For each $p\in S_1(\emptyset)$ let $p^{conv}$ be the set of all the convex formulae from $p$. Basic properties of interval types are given in the following lemma.

\begin{lem}\label{lem interval types}
\begin{enumerate}[(a)]
\item For $p\in S_1(\emptyset)$, $p^{conv}\in IT(T)$ and $p^{conv}(\Mon)= conv(p(\Mon))$ (the convex hull of  $p(\Mon)$ in $(\Mon,<)$). 
\item If $\Pi\in IT(T)$, then $\Pi=p^{conv}$ for any $p\in S_1(\emptyset)$ extending $\Pi$.
\item Distinct interval types have disjoint (convex) loci (in $\Mon$). In particular, $<$ naturally orders $IT(T)$.
\end{enumerate}
\end{lem}
\begin{proof}
(a) We prove that $p^{conv}$ is a maximal partial type consisting of convex $L$-formulae. If $\phi(x)$ is a convex $L$-formula not belonging to $p^{conv}(x)$, then $\lnot\phi(x)\in p(x)$ and $\lnot\phi(x)$ is equivalent to $\phi^-(x)\lor\phi^+(x)$, where $\phi^-(x)$ and $\phi^+(x)$ are convex $L$-formulae describing the sets $\{a\in\Mon\mid a<\phi(\Mon)\}$ and $\{a\in \Mon\mid \phi(\Mon)<a\}$, respectively. Since $\lnot\phi(x)$ is consistent, at least one of $\phi^-(x)$ and $\phi^+(x)$ is consistent and exactly one of them belongs to $p(x)$, i.e.\ to $p^{conv}(x)$. Thus $\phi(x)$ is inconsistent with $p^{conx}(x)$ as it is inconsistent with both $\phi^-(x)$ and $\phi^+(x)$. Therefore, $p^{conv}$ is indeed a maximal partial type consisting of convex $L$-formulae.

The set $p^{conv}(\Mon)$ is clearly convex, and since $p(\Mon)\subseteq p^{conv}(\Mon)$ we have $conv(p(\Mon))\subseteq p^{conv}(\Mon)$. If the inclusion is strict, then we can find $a\models p^{conv}$ such that either $a<conv(p(\Mon))$ or $conv(p(\Mon))<a$, e.g.\ $a<conv(p(\Mon))$. Then $a<p(\Mon)$ is also true, so by compactness there is $\phi'(x)\in p(x)$ such that $a<\phi'(\Mon)$. If $\phi(x)$ is the formula describing $conv(\phi'(\Mon))$ then $\phi(x)\in p(x)$, so $\phi(x)\in p^{conv}(x)$, and $a<\phi(\Mon)$ too. But this is not possible as $a\models p^{conv}$. 

(b) If $\Pi\in IT(T)$ and $p\supseteq\Pi$,  then $\Pi\subseteq p^{conv}$. By maximality $\Pi=p^{conv}$.

(c) Let $\Pi_1,\Pi_2\in IT(T)$ be distinct; then there is a convex $L$-formula $\phi(x)\in \Pi_1\smallsetminus\Pi_2$. By (b) we have $\Pi_2=p_2^{conv}$ where $p_2\in S_1(\emptyset)$ extends $\Pi_2$. As in the proof of part (a) we conclude that either $\phi^-(x)$ or $\phi^+(x)$ belongs to $p_2(x)$, and hence to $\Pi_2(x)$, too; say $\phi^-(x)\in\Pi_2(x)$. Now $\Pi_1(\Mon)$ and $\Pi_2(\Mon)$ are subsets of two disjoint convex sets $\phi(\Mon)$ and $\phi^-(\Mon)$, therefore they are disjoint, too.
\end{proof}

\begin{lem}\label{rub lem nwor}
For all non-algebraic types $p,q\in S_1(\emptyset)$: \  $p\nwor q$ \ if and only if \  $p^{conv}=q^{conv}$.
\end{lem}
\begin{proof}
($\Leftarrow$) \ Assume $p^{conv}=q^{conv}$. Then by Lemma \ref{lem interval types}(a) $conv(p(\Mon))=conv(q(\Mon))$ and  we can find $a,a'\models p$ and $b\models q$ such that $a<b<a'$;  hence $ab\not\equiv a'b$ and $p\nwor q$ follows.
 
$(\Rightarrow)$ \ 
Suppose  that   $p^{conv}\neq q^{conv}$. Then by Lemma \ref{lem interval types} $p^{conv}(\Mon)$ and $q^{conv}(\Mon)$ are disjoint convex sets so, without loss of generality, assume $p^{conv}(\Mon)<q^{conv}(\Mon)$; in particular,  $p(\Mon)<q(\Mon)$. 
Let   $b ,b' \models q$ and $a\models p$. Choose  $f\in\Aut(\Mon)$   such that $f(b)=b'$. Note that  $f$ fixes setwise $q^{conv}(\Mon)$. 
The mapping defined by $g(x)=x$ for $x\notin q^{conv}(\Mon)$ and $g(x)=f(x)$ for $x\in q^{conv}(\Mon)$ is an  automorphism of $\Mon$  because it preserves $<$ and all the unary predicates. Thus   $\tp(a,b)=\tp(a,b')$ and $p\wor q$. 
\end{proof}

Recall that a type $p\in S_1(\emptyset)$ is convex if there is a $\emptyset$-definable linear order $(D_p,<_p)$ such that $p(\Mon)$ is a $<_p$-convex subset of $D_p$. 

\begin{lem}\label{rub lem convex} 
If some  $\Pi\in IT(T)$ has infinitely many completions in $S_1(\emptyset)$, then one of them is a non-convex type and $I(\aleph_0,T)=2^{\aleph_0}$.
\end{lem}
\begin{proof}
Fix $\Pi\in IT(T)$,  let  $F=\{p\in S_1(\emptyset)\mid \Pi\subseteq p\}$. We will prove that every convex type $p\in F$ is an isolated point of $F$ (viewed as a subspace of  $S_1(\emptyset)$). Assume that $p\in F$ is convex.  Apply  Lemma \ref{Lema_convex_so_pair_witness1} to $\mathbf p=(p,<)$: there is $\phi_p(x)\in p(x)$ such that $p(\Mon)$ is  convex in $(\phi_p(\Mon),<)$.  
We {\it claim} that $\Pi(x)\cup\{\phi_p(x)\}\forces p(x)$. To prove it, assume  $a\in \Pi(\Mon)\cap \phi_p(\Mon)$ and let $q=\tp(a)$. Then $p^{conv}=q^{conv}=\Pi$, so by Lemma \ref{lem interval types}(a) there are $b,b'\models p$ with $b<a<b'$. Since $a,b,b'\in\phi_p(\Mon)$ and $p(\Mon)$ is a convex subset of $\phi_p(\Mon)$, we deduce $c\models p$; this proves the claim which  implies that $p$ is an isolated point of $F$.  

Clearly,  $F$ is a closed subset of $S_1(\emptyset)$ and  we have shown that every convex member of $F$ is an isolated point of $F$. Therefore, if $F$ is infinite, then $F$ contains a non-isolated type, so a non-convex type 
 and $I(\aleph_0,T)=2^{\aleph_0}$ follows by Proposition \ref{Prop weakly_reg_convexsimple}.
\end{proof} 
 
\noindent{\it Proof of Rubin's theorem.} \  Assume  $I(\aleph_0,T)<2^{\aleph_0}$. Then  by Theorem \ref{conclusion}(a) none of the conditions (C1)-(C5) holds; in particular,   the failure of (C5) implies that the cardinal $w_T$, i.e the number of $\nwor$-classes of non-isolated 1-types, is finite. By  Lemma  \ref{rub lem nwor} the  $\nwor$-classes correspond to interval types, so there are only finitely many interval types having a non-isolated completion in $S_1(\emptyset)$. By 
Lemma  \ref{rub lem convex}  each $\nwor$-class is finite, so there are only finitely many   non-isolated types in $S_1(\emptyset)$ (in other words: the Cantor-Bendixson rank of $S_1(\emptyset)$ is at most  $1$). In particular, there are finitely many $\nfor$-classes of non-isolated types in $S_1(\emptyset)$, so by Theorem \ref{conclusion}(c) we have $I(\aleph_0,T)<\aleph_0$. \qed

\end{document}